\documentclass[12pt]{elsarticle}
\usepackage{amsmath,amssymb,amsthm}
\usepackage{mathrsfs}
\usepackage[bookmarks=true,
            bookmarksnumbered=true,
            bookmarksopen=true,
            colorlinks,
            pdfborder=001,
            linkcolor=black]{hyperref}
\usepackage{bm}
\usepackage{array}
\usepackage{float}
\usepackage{color}
\usepackage{lineno}
\usepackage{subcaption}
\usepackage{stmaryrd}
\usepackage{extarrows}
\usepackage{subcaption}
\newtheorem{thm}{Theorem}[section]

\newtheorem{rmk}{Remark}[section]
\newtheorem{example}{Example}[section]
\newtheorem{proposition}{Proposition}[section]
\newtheorem{definition}{Definition}[section]
\newproof{pf}{Proof}
\numberwithin{equation}{section}
\numberwithin{figure}{section}
\numberwithin{table}{section}
\allowdisplaybreaks

\newcommand\diag{\mathrm{diag}}

\newcommand\dd{\mathrm{d}}

\newcommand\abs[1]{\lvert #1 \rvert}

\newcommand\pro[2]{\langle{#1},{#2}\rangle}
\newcommand\pd[2]{\dfrac{\partial {#1}}{\partial {#2}}}

\newcommand\be{\bm{e}}
\newcommand\bF{\bm{F}}
\newcommand\bG{\bm{G}}

\newcommand\ux{u^x}
\newcommand\uy{u^y}
\newcommand\uz{u^z}
\newcommand\vx{v_1}
\newcommand\vy{v_2}
\newcommand\vz{v_3}
\newcommand\mx{m_1}
\newcommand\my{m_2}
\newcommand\mz{m_3}
\newcommand\bx{b^x}
\newcommand\by{b^y}
\newcommand\bz{b^z}
\newcommand\Bx{B_1}
\newcommand\By{B_2}
\newcommand\Bz{B_3}
\newcommand\bs{\bm{s}}
\newcommand\br{\bm{r}}

\newcommand\bv{\bm{v}}
\newcommand\bw{\bm{w}}
\newcommand\bW{\bm{W}}
\newcommand\bV{\bm{V}}
\newcommand\bU{\bm{U}}
\newcommand\pt{p_\text{tot}}
\newcommand\bb{\abs{b}^2}
\newcommand\BB{\abs{\bm{B}}^2}
\newcommand\divB{\nabla\cdot{\bm{B}}}
\newcommand\vB{\bv\cdot\bm{B}}

\newcommand\parc{\nabla_{\alpha}}

\newcommand\xl{{i-\frac12}}
\newcommand\xr{{i+\frac12}}
\newcommand\yl{{j-\frac12}}
\newcommand\yr{{j+\frac12}}

\newcommand\jump[1]{\llbracket #1 \rrbracket}
\newcommand\mean[1]{\{\!\!\{ #1 \}\!\!\}}
\newcommand\meanln[1]{\{\!\!\{ #1 \}\!\!\}^{\text{ln}}}

\begin{document}

\begin{frontmatter}

  \title{High-order accurate entropy stable nodal discontinuous Galerkin schemes
  for the ideal special relativistic magnetohydrodynamics}

  \author{Junming Duan}
  \ead{duanjm@pku.edu.cn}
  \author{Huazhong Tang\corref{cor1}}
  \ead{hztang@math.pku.edu.cn}
  \address{Center for Applied Physics and Technology, HEDPS and LMAM,
    School of Mathematical Sciences, Peking University, Beijing 100871, P.R. China}
  \cortext[cor1]{Corresponding author. Fax:~+86-10-62751801.}

  \begin{abstract}
    This paper studies  high-order accurate entropy stable nodal discontinuous Galerkin (DG)
    schemes for the ideal special relativistic magnetohydrodynamics (RMHD).
    It is built on the  modified RMHD equations with a   particular source term,
    which is analogous to the Powell's eight-wave formulation and can be  symmetrized so that
   an ``entropy   pair''  is obtained.
    We design an affordable ``fully
    consistent'' two-point entropy conservative flux, which is not only consistent
    with the physical flux, but also maintains the zero parallel magnetic component,
    and then construct high-order accurate semi-discrete
    entropy stable DG schemes   based on the quadrature rules and the entropy conservative and stable fluxes.
   They satisfy the semi-discrete ``entropy inequality''
    for the given ``entropy pair'' and are integrated in time by using the high-order
    explicit strong stability preserving  Runge-Kutta schemes to get further the fully-discrete nodal DG schemes.
 %
    Extensive numerical tests are conducted to validate the accuracy and the
    ability to capture discontinuities of our schemes.
    Moreover, our entropy conservative flux is compared to an existing flux
    through some numerical tests. The results show that the zero parallel magnetic
    component in the numerical flux can help to decrease the error in the parallel magnetic component
    in one-dimensional tests, but  two entropy conservative fluxes give similar
    results since the error in  the magnetic field divergence
   seems dominated  in the two-dimensional tests.
  \end{abstract}


\end{frontmatter}

\section{Introduction}
This paper is concerned with the high-order accurate numerical schemes for the
ideal special relativistic magnetohydrodynamic (RMHD) equations. In the
covariant form, the four-dimensional space-time RMHD equations can be written as
follows \cite{Anile1987}
\begin{equation}\label{eq:RMHD}
    \begin{alignedat}{2}
       \parc(\rho u^\alpha)=0,           
      \ \
       \parc \mathrm{T}^{\alpha\beta}=0, 
      \ \
       \parc \Psi^{\alpha\beta}=0,       
  \end{alignedat}
\end{equation}
where the Einstein summation convention has been used,  $\rho$ and $u^\alpha$ denote the rest-mass density and   the
four-velocity vector, respectively,
$\parc$ denotes the covariant derivative with respect to
the four-dimensional space-time coordinates $(t,x^1,x^2,x^3)$,
the Greek indices run from $0$ to $3$ (or $t,x,y,z$),
\begin{equation}
  \Psi^{\alpha\beta}=u^\alpha b^\beta-u^\beta b^\alpha,
\end{equation}
and $\mathrm{T}^{\alpha\beta}$ is  the energy-momentum tensor and can be decomposed into
the fluid part $\mathrm{T}^{\alpha\beta}_{f}$ and the electromagnetic part
$\mathrm{T}^{\alpha\beta}_{m}$,   defined by
\begin{align}
  \mathrm{T}^{\alpha\beta}_{f}&=(e+p)u^\alpha u^\beta +pg^{\alpha\beta}, \\
  \mathrm{T}^{\alpha\beta}_{m}&=\bb(u^\alpha u^\beta+g^{\alpha\beta}/2)-b^\alpha b^\beta. \label{eq:RMHD2}
\end{align}
%
Here $b^\alpha$ and $e$ are the covariant magnetic field  and the total
energy density, respectively.

Throughout this paper, the metric tensor $g^{\alpha\beta}$ is taken as the
Minkowski tensor, i.e. $g^{\alpha\beta}=\diag\{-1,1,1,1\}$, and units in which the speed of light is equal to one will be used.
The relations between the four-vectors $u^\alpha$ and $b^\alpha$, and the spatial components
of the velocity $\bv=(\vx,\vy,\vz)$ and the laboratory magnetic field
$\bm{B}=(\Bx,\By,\Bz)$ are
\begin{align}
  &u^\alpha=W(1,\bv),\\
  &b^\alpha=W\left(\vB,\frac{\bm{B}}{W^2}+\bv(\vB)\right),
\end{align}
with
\begin{align}
  u^\alpha u_\alpha=-1,\quad u^\alpha b_\alpha=0,\quad \bb\equiv b^\alpha b_\alpha=\frac{\BB}{W^2}+(\vB)^2,
\end{align}
where $W=1/\sqrt{1-\abs{\bv}^2}$ is the Lorentz factor.
To close the system \eqref{eq:RMHD}-\eqref{eq:RMHD2}, this paper
  considers   the equation of state (EOS) for the perfect gas
\begin{equation*}
  p=(\Gamma-1)\rho \epsilon,
\end{equation*}
with the adiabatic index $\Gamma\in(1,2]$, where $\epsilon=e/\rho-1$ is the specific
internal energy.

For the computational purpose, the system \eqref{eq:RMHD}-\eqref{eq:RMHD2} should be rewritten in a
  lab frame as follows
\begin{equation}\label{eq:RMHDdiv1}
  \pd{\bU}{t}+\sum_k \pd{\bF_k(\bU)}{x^k}=0, 
\end{equation}
with the divergence-free constraint on the magnetic field
\begin{equation}
 \sum_k \pd{B_k}{x^k}=0, \label{eq:divB}
\end{equation}
where $\bU$ and $\bF_k$ are respectively
the conservative variables vector and the flux vector along the $x^k$-direction, and defined by
\begin{align}\label{eq:RMHDdiv2}\begin{aligned}
  &\bU=(D,\bm{m},E,\bm{B})^\mathrm{T},\\
  &\bF_k=(Dv_k,\bm{m}v_k-B_k(\bm{B}/W^2+(\vB)\bv)+\pt\be_k,m_k,v_k\bm{B}-B_k\bv)^\mathrm{T},
\end{aligned}\end{align}
with the mass density $D=\rho W$, the momentum density
$\bm{m}=(\rho hW^2+\BB)\bv-(\bv\cdot\bm{B})\bm{B}$,
the energy density $E=DhW-\pt+\BB$, and $\be_k$ denotes the $k$-th row of the
$3\times 3$ unit matrix.
Here, $\pt$ denotes the total pressure containing the gas pressure $p$ and
the magnetic pressure $p_m=\frac12\bb$,
and $h$ is the specific enthalpy defined by $h=(e+p)/\rho$.
Because there is no explicit expression for the primitive variables
$(\rho,\bv,p,\bm{B})^\mathrm{T}$ and the flux $\bF_k$ in terms of $\bU$,
a nonlinear algebraic equation should be solved (the approach \cite{Koldoba2002}
is used by us), in order to recover the values of the primitive variables and the flux from the given $\bU$.

The RMHD system  \eqref{eq:RMHD}-\eqref{eq:RMHD2} considers the relativistic description for the dynamics of
the fluid (gas) at nearly the speed of light when  the astrophysical
phenomena are investigated from stellar to galactic scales, e.g. coalescing neutron stars, core
collapse supernovae,  active galactic nuclei, superluminal jets, the formation of black holes,
and gamma-ray bursts, etc.
It is obvious that its nonlinearity   becomes much
stronger than the non-relativistic case because of the relativistic effect,
thus it is very difficult and challenging to treat it analytically.
Numerical simulation is a useful way leading us to a better understanding of the physical
mechanisms in the relativistic hydrodynamics (RHD) and the  RMHDs.
The first numerical work may date back to the artificial viscosity method
for the RHD equations in the Lagrangian coordinates \cite{may1,may2}
and the Eulerian coordinates \cite{wilson}.
Since the early 1990s, the modern shock-capturing methods were extended to the
RHDs and   RMHDs, such as
the Roe-type scheme \cite{Anton2010,eulderink},
the Harten-Lax-van Leer (HLL) method \cite{Del2002,Del2003,Mignone2009,schneider},
the Harten-Lax-van Leer Contact (HLLC) method \cite{Ling2019,Mignone2005,Mignone2006},
the essentially non-oscillatory (ENO) and the weighted ENO (WENO) methods
\cite{dolezal,Del2002,Del2003},
the piecewise parabolic methods \cite{marti3,Mignone2005PPM},
the Runge-Kutta discontinuous Galerkin (DG) methods with WENO limiter
\cite{Zhao2013,Zhao2017},
the direct Eulerian generalized Riemann problem schemes
\cite{wu2016,wu2014,yz1,yz2},
the two-stage fourth-order time discretization \cite{Yuan2019},
the adaptive moving mesh methods \cite{he1,he2}, and so on.
The readers are  referred to the early review articles
\cite{font,marti2,marti2015} for more references.
Recently, the properties of the admissible state set and the
physical-constraints-preserving (PCP) numerical schemes were well studied for
the RHDs and   RMHDs, see \cite{wu2017a,wu2017b,wu2015,wu2018,wu2017}.
For the numerical solutions of the  RMHD equations, we need to  deal carefully with
the   divergence-free constraint \eqref{eq:divB}.
In the non-relativistic case, many works have focused on this issue, for
example, the projection method \cite{Brackbill1980}, the constrained transport
(CT) method \cite{Evans1988}, the eight-wave formulation of the MHD equations
\cite{Powell1994}, the hyperbolic divergence cleaning method \cite{Dedner2002},
the locally divergence-free DG method \cite{LiFY2005}, the ``exactly''
divergence-free central DG method \cite{LiFY2011}. Some of those works have
been extended to the relativistic case, such as
\cite{Beckwith2011,Del2003,Nunez2016,Zhao2017}.

For the RMHD equations, the entropy condition is an
important property which must be satisfied according to the second law of
thermodynamics.
On the other hand, it is well known that the weak solution of the
quasi-linear hyperbolic conservation laws might not be unique so that the
entropy condition is needed to single out the unique physical relevant solution among
all the weak solutions. 
Thus it is a matter of cardinal significance to seek the entropy stable schemes (satisfying some discrete or semi-discrete  entropy conditions) for the quasi-linear system of hyperbolic conservation laws.

\begin{definition}
\label{Def:001}
  A strictly convex scalar function $\eta(\bU)$ is called an {\em entropy function} for the
  system \eqref{eq:RMHDdiv1}-\eqref{eq:RMHDdiv2} if there are associated entropy
  fluxes $q_k(\bU)$ such that
  \begin{equation}\label{eq:entropy}
    q_k'(\bU)=\bV^\mathrm{T}\bF_k'(\bU),
  \end{equation}
  where $\bV=\eta'(\bU)^\mathrm{T}$ is called the {\em entropy variables}, and
  $(\eta,q_k)$ is an {\em entropy pair}.
\end{definition}
For the smooth solutions of \eqref{eq:RMHDdiv1}-\eqref{eq:RMHDdiv2},
multiplying \eqref{eq:RMHDdiv1} by $\bV^\mathrm{T}$ gives the entropy identity
\begin{equation*}
  \pd{\eta(\bU)}{t}+\sum_k\pd{q_k(\bU)}{x^k} = 0.
\end{equation*}
However, if the solution  contains a discontinuity, then the above identity
does not hold.
\begin{definition}
  A weak solution $\bU$ of \eqref{eq:RMHDdiv1}-\eqref{eq:RMHDdiv2} is called an {\em entropy solution} if for
  all entropy functions $\eta$, the inequality
  \begin{equation}\label{eq:entropyineq}
    \pd{\eta(\bU)}{t}+\sum_k\pd{q_k(\bU)}{x^k} \leqslant 0,
  \end{equation}
  holds in the sense of distributions.
\end{definition}

For the scalar conservation laws, the conservative monotone schemes
  were nonlinearly stable and satisfied the discrete entropy conditions,
thus they could converge to the entropy solution \cite{Crandall1980,Harten}.
A class of so-called E-schemes satisfying the entropy condition for any convex
entropy was studied in \cite{Osher1984,Osher1988}, but they were only first-order accurate.
Generally, it is hard to show that the high-order schemes of
the scalar conservation laws and the schemes for the system of hyperbolic conservation laws
satisfy the entropy inequality for any convex entropy function.
Two relative works were presented in 
\cite{Bouchut1996}  and 
\cite{Hughes1986}.
The former is second-order accurate and not in the standard finite volume
form, while the latter approximates the entropy variables and needs solving
nonlinear equations at each time step.
A lot of people are trying  to study the high-order accurate
entropy stable schemes, which satisfy the entropy inequality for a given entropy pair.
The second-order entropy conservative schemes (satisfying the discrete entropy
identity) were studied in \cite{Tadmor1987,Tadmor2004}, and their higher-order
extension was considered in \cite{Lefloch2002}.
It is known that the entropy conservative schemes may become oscillatory near the
shock waves so that some additional dissipation term  has to be added
to obtain the entropy stable schemes (satisfying the
discrete entropy inequality).
Combining the entropy conservative flux 
with the ``sign'' property of the ENO reconstruction, the
arbitrary high-order entropy stable schemes were constructed by using high-order
dissipation terms \cite{Fjordholm2012}.
The entropy stable schemes based on summation-by-parts (SBP) operators were
developed for the Navier-Stokes equations \cite{Fisher2013}.
The semi-discrete DG schemes for scalar conservation laws were proved to satisfy a discrete entropy
inequality for the square entropy  \cite{Jiang1994},
and some entropy stable DG schemes   were also
studied, such as the space-time DG
formulation \cite{Barth1999,Hiltebrand2014} and the entropy stable nodal DG schemes using suitable quadrature rules
for the conservation laws \cite{Chen2017} and the MHD equations \cite{Liu2018}.
As a base of those works, constructing the affordable two-point entropy
conservative flux is very important, and has been
extended to the shallow water equations \cite{Gassner2016},
the RHD equations \cite{Duan2019},
the MHD equations \cite{Chandrashekar2016,Winters2016,Derigs2018},
very recently the RMHD equations \cite{Wu2019}, and so on.
Because it can be verified that the original MHD equations cannot be symmetrized,
which is also true for the RMHD equations,
most of the above mentioned MHD works are based on
the modified MHD equations with a non-conservative source term,
which is first introduced by Godunov \cite{Godunov1972} in one-dimensional case,
and then extended to multi-dimensional cases by Powell \cite{Powell1994}.
Due to  introducing the source term, the sufficient condition proposed in
 \cite{Tadmor1987} for a finite difference   or   volume
scheme to satisfy an entropy identity should be modified, see
\cite{Chandrashekar2016,Winters2016}, and in the DG framework, the
non-conservative source term should also be carefully discretized \cite{Liu2018}.

This paper aims at studying the high-order accurate entropy stable nodal DG schemes for the RMHD
equations. Because the conservative RMHD equations cannot be symmetrized,
a suitable non-conservative source term should be added to obtain the
modified RMHD equations, which can be symmetrized by the entropy pair following
  \cite{Godunov1972}.
By using the modified RMHD equations and suitable quadrature rules as well as the SBP,
high-order accurate entropy stable DG
schemes   satisfying the entropy inequality for the given entropy pair  are constructed, where
the so-called two-point entropy conservative flux is used in the integral over the cell,
while the entropy stable fluxes are used at the cell interfaces, e.g., the
Godunov flux, the HLL flux with suitably chosen wave speeds, the Lax-Friedrichs
flux, and so on. It can be shown that the temporal change of the total entropy in each cell
for such semi-discrete scheme is only effected from the entropy stable fluxes at the cell interfaces. 
One of our main tasks is to technically design the affordable
two-point entropy conservative flux with zero parallel magnetic component,
which  is ``fully consistent'' with the physical flux
 (Note that the parallel magnetic component of the physical flux is
always zero).
It is worth noting that in a very recent and independent work
\cite{Wu2019},
an entropy conservative scheme is constructed with a suitable discretization of the source term,
but its parallel magnetic component of the entropy conservative flux  does not always vanish.
This paper will give a  comparison of those entropy conservative fluxes by some numerical tests to
validate that our newly derived entropy conservative flux may serve as a better
base of the entropy conservative or   stable schemes for the RMHD
equations. 

The paper is organized as follows.
Section \ref{section:Symm} introduces the modified symmetrizable RMHD equations.
Section \ref{section:OneD} presents the 1D entropy stable nodal DG schemes
and constructs the affordable two-point entropy conservative flux for the
one-dimensional RMHD equations.
Section \ref{section:MultiD} extends the results to the two-dimensional cases.
Extensive numerical tests are conducted in Section \ref{section:Num}
to validate the effectiveness of our schemes.
Section \ref{section:Conclusion} gives some conclusions.

\section{Symmetrization of the ideal special RMHD equations}\label{section:Symm}
This section gives a derivation of the symmetrizable RMHD equations,
which is analogous to the non-relativistic case.
First of all, we want to show that the RMHD equations
\eqref{eq:RMHDdiv1}-\eqref{eq:RMHDdiv2} are not symmetrizable by using the
entropy pair based on the thermodynamical entropy as in the RHD case \cite{Duan2019}.
For the smooth solutions of the special RMHD equations
\eqref{eq:RMHDdiv1}-\eqref{eq:RMHDdiv2}, if
defining the thermodynamical entropy
\begin{equation*}
  s=\ln(p)-\Gamma\ln(\rho),
\end{equation*}
then from \eqref{eq:RMHDdiv1}-\eqref{eq:RMHDdiv2} and the first law of thermodynamics
\begin{equation*}
  T~\dd s=\dd(e/\rho)+p\dd(1/\rho),
\end{equation*}
one can show \cite{Anile1987} that
\begin{equation*}
  \pd{s}{t}+\sum_k v_k\pd{s}{x^k}+\frac{(\vB)\divB}{\rho\epsilon}=0,
\end{equation*}
where $T$ is the temperature,
which is equal to $\epsilon$ under the assumptions in this paper.
Combining it with the first equation in \eqref{eq:RMHDdiv1}  obtains
\begin{equation}
  \pd{(\rho Ws)}{t}+\sum_k \pd{(\rho v_k Ws)}{x^k}+\frac{(\Gamma-1)\rho W(\vB)\divB}{p}=0,
\end{equation}
which implies  that  the following quantities
\begin{equation}\label{eq:entropypair}
  {\eta}(\bU)=\dfrac{-\rho Ws}{\Gamma-1},\quad {q_k}(\bU)=\dfrac{-\rho v_k Ws}{\Gamma-1},
\end{equation}
satisfy an additional conservation law
\begin{equation}\label{eq:entropyID}
  \pd{\eta}{t}+\sum_k \pd{q_k}{x^k}=0,
\end{equation}
 under the constraint $\divB=0$.
However, unfortunately,  the pair $(\eta, q_k)$ defined in \eqref{eq:entropypair}
does not  satisfy  \eqref{eq:entropy}, since  
\begin{equation}\label{eq:qkFk1}
  q_k'(\bU)=\bV^\mathrm{T}\bF_k'(\bU)+\dfrac{\rho W}{p}(\vB)B_k'(\bU),
\end{equation}
where the  vector $\bV={\eta}'(\bU)^\mathrm{T}$ is explicitly given by
\begin{equation*}
  \bV=(V_1,\cdots,V_8)^\mathrm{T}=\left(\dfrac{\Gamma-s}{\Gamma-1}+\dfrac{\rho}{p},
  \dfrac{\rho u^x}{p}, \dfrac{\rho u^y}{p}, \dfrac{\rho u^z}{p},
  -\dfrac{\rho W}{p}, \dfrac{\rho b^x}{p}, \dfrac{\rho b^y}{p}, \dfrac{\rho b^z}{p} \right)^\mathrm{T}.
\end{equation*}
Moreover,
 the pair $(\eta, q_k)$ cannot symmetrize the RMHD system \eqref{eq:RMHDdiv1}-\eqref{eq:RMHDdiv2},
 because
it can be verified that  the matrix $\pd{\bU}{\bV}$ is symmetric
and   positive definite, but
  $\pd{\bF_k}{\bV}$ is not symmetric.


To derive a symmetrizable RMHD system for   \eqref{eq:RMHDdiv1}-\eqref{eq:RMHDdiv2},
one needs to add some non-conservative
source terms, similar to those for the non-relativistic ideal MHD equations
  \cite{Godunov1972,Powell1994},
 to get a modified RMHD system    as follows
\begin{equation}\label{eq:symm1}
  \pd{\bU}{t}+\pd{\bF_k}{x^k}+\Phi'(\bV)^\mathrm{T}\divB=0,
\end{equation}
where $\Phi(\bV)$ is a homogeneous function of degree one, i.e. $\Phi=\Phi'(\bV)\bV$.
Taking the dot product of $\bV=\eta'(\bU)^\mathrm{T}$ with \eqref{eq:symm1} yields
\begin{align*}
 \bV^\mathrm{T}\pd{\bU}{t} &+\left(\bV^\mathrm{T}\pd{\bF_k}{\bU}+\Phi(\bV)\pd{B_k}{\bU}\right)\pd{\bU}{x^k}\\
  =&\pd{\eta}{t}+\left(\pd{q_k}{\bU}
  +\left(\Phi(\bV) -\dfrac{\rho W}{p}(\vB)\right)\pd{B_k}{\bU}\right)\pd{\bU}{x^k}=0.
\end{align*}
It is obvious that the above equation becomes the   identity \eqref{eq:entropyID}
if  defining the homogeneous function $\Phi$ by
\begin{equation*}
  \Phi(\bV)=\dfrac{\rho W(\vB)}{p}=-\dfrac{V_2V_6+V_3V_7+V_4V_8}{V_5}.
\end{equation*}
One can  verify that 
$\Phi'(\bV)=\big(0,\dfrac{\bx}{W},\dfrac{\by}{W},\dfrac{\bz}{W},\vB,\bv\big)$ and
the pair $(\eta,q_k)$ can symmetrize the modified RMHD system \eqref{eq:symm1}
because  applying the change of variables $\bV=\bV(\bU)$   gives
\begin{equation*}
  \pd{\bU}{\bV}\pd{\bV}{t}+\left(\pd{\bF_k}{\bV}+\Phi'(\bV)\pd{B_k}{\bV}\right)\pd{\bV}{x^k}=0,
\end{equation*}
and $\pd{\bF_k}{\bV}+\Phi'(\bV)\pd{B_k}{\bV}$ is symmetric.
Notice that the   identity \eqref{eq:entropyID} is obtained without using
the divergence-free condition, and useful in constructing an entropy stable
scheme  because the numerical divergence may not be zero.
Moreover, we can define the ``entropy potential'' $\psi_k$ from given  $(\eta(\bU), q_k(\bU))$ by
\begin{equation}\label{eq:potential_final}
  \psi_k:= \bV^\mathrm{T} \bF_k(\bU)+\Phi(\bV) B_k-q_k(\bU)=\rho v_kW+\dfrac{\rho
  v_kW\bb}{2p},
\end{equation}
which makes the following identity true
\begin{align*}
	\int_{\Omega} \left( \pd{\eta}{t}+ \pd{q_k}{x^k} \right)\dd\bm{x}
	 =&\int_{\Omega} \bV^\mathrm{T}\left( \pd{\bU}{t}+ \pd{\bF_k(\bU)}{x^k}+\Phi'(\bV)^\mathrm{T}\divB\right)\dd\bm{x}\\
	 =&\int_{\Omega} \left( \pd{\eta(\bU)}{t}+ \pd{(\bV^\mathrm{T} \bF_k(\bU))}{x^k} - \pd{\bV^\mathrm{T}}{x^k} \bF(\bU)+\Phi(\bV)\divB \right)\dd\bm{x}\\
	 =&\int_{\Omega} \left( \pd{\eta(\bU)}{t}+ \pd{(\bV^\mathrm{T} \bF_k(\bU))}{x^k} - \pd{\psi_k(\bU)}{x^k}+\Phi(\bV)\divB \right)\dd\bm{x}.
\end{align*}
The ``entropy potential'' {plays}   an important role in obtaining the sufficient condition for the two-point entropy conservative fluxes.

\begin{rmk}
Combining the modified induction equations and the first equation in
\eqref{eq:RMHDdiv1} gives
\begin{equation*}
\dfrac{\partial}{\partial t}\left(\dfrac{\divB}{\rho W}\right)
+\bv\cdot\nabla\left(\dfrac{\divB}{\rho W}\right)=0,
  \end{equation*}
which implies that the errors in divergence may be transported by
  the flow. 
  However, one drawback is that the
  non-conservative source term may lead to incorrect solutions.  
\end{rmk}

\section{One-dimensional entropy stable DG schemes}\label{section:OneD}
This section considers the one-dimensional $x$-splitting system of \eqref{eq:symm1}, i.e.,
\begin{equation}\label{eq:RMHDsym1D}
  \pd{\bU}{t}+\pd{\bF_1(\bU)}{x}=-\Phi'(\bV)^\mathrm{T}\pd{\Bx}{x}.
\end{equation}

\subsection{Spatial DG discretization}
Assume that the computational domain $\Omega$ is divided into $N_x$ cells,
$I_i=(x_{\xl},x_{\xr})$, $i=1,2,\cdots,N_x$, where
$x_{\frac12}<x_{\frac32}<\dots<x_{N_x+\frac12}$.
Denote the center of $I_i$ and the mesh size by
$x_i=\frac12(x_{\xl}+x_{\xr})$
and $\Delta x_i=x_{\xr}-x_{\xl}$ respectively.
The spatial DG approximation space is
\begin{equation}
  \bW^r_h=\{\bw_h \in \left[L^2(\Omega)\right]^8:~ \bw_h|_{I_i}\in [P^r(I_i)]^8,i=1,2,\dots,N_x\},
\end{equation}
where $P^r(I_i)$ denotes all polynomials of degree at most $r$ on $I_i$.
Let us multiply \eqref{eq:RMHDsym1D} with the test function $\bw_h\in \bW^r_h$,
integrating it over the control volume $I_i$,
and introducing the numerical fluxes at the cell interfaces following \cite{Liu2018},
then the spatial DG approximation for \eqref{eq:RMHDsym1D} is to find $\bU_h\in\bW^r_h$,
satisfying
\begin{align}\label{eq:1DDG}
  \frac{d}{dt}\int_{I_i}\bU_h^\mathrm{T} \bw_h\dd x
  =&\int_{I_i}\bF_1(\bU_h)^\mathrm{T}\pd{\bw_h}{x}\dd x
  -\int_{I_i}\Phi'(\bU_h)\pd{(\Bx)_h}{x}\bw_h\dd x \nonumber \\
  &-\hat{\bF}_\xr^\mathrm{T}\bw_h(x_\xr^-) +\hat{\bF}_\xl^\mathrm{T}\bw_h(x_\xl^+) \nonumber \\
  &-\frac12\Phi'(\bU_h(x_\xr^-))\jump{(\Bx)_h}_\xr\bw_h(x_\xr^-) \nonumber \\
  &-\frac12\Phi'(\bU_h(x_\xl^+))\jump{(\Bx)_h}_\xl\bw_h(x_\xl^+),
\end{align}
 for any $\bw_h\in\bW^r_h$ and $i=1,2,\cdots,N_x$,
where $\jump{(\Bx)_h}=(\Bx)_h^+-(\Bx)_h^-$ denotes the jump of
$(\Bx)_h$ at the cell interface, with the superscripts $+,-$ denoting its right and
left limits, and $\hat\bF_\xr$ is a two-point numerical flux
\begin{equation}
  \hat\bF_\xr=\hat\bF(\bU_h(x_\xr^-),\bU_h(x_\xr^+)),
\end{equation}
which is consistent with the physical flux $\bF_1$ in \eqref{eq:RMHDsym1D}. 
Notice that the discretization of the non-conservative source term
is originally proposed for directly solving the Hamilton-Jacobi equations
\cite{Cheng2007}.

In what follows, the derivation of the entropy conservative and stable DG
schemes largely relies on the summation-by-parts (SBP) operator. The key idea
is to approximate the integrals in \eqref{eq:1DDG} by using the $(r+1)$-point
Legendre-Gauss-Lobatto quadrature rule.
For the reference element  $I=[-1,1]$,  denote the Legendre-Gauss-Lobatto
quadrature points as
\begin{equation*}
  -1=\xi_0<\xi_1<\cdots<\xi_r=1,
\end{equation*}
and  corresponding weights as $\omega_j,j=0,1,\cdots,r$.
Define the Lagrangian basis polynomials by
\begin{equation*}
  L_j(\xi)=\prod\limits_{l=0,l\not=j}^{r}\dfrac{\xi-\xi_l}{\xi_j-\xi_l}, \
  j=0,1,\cdots,r,
\end{equation*}
 the {difference} matrix $\bm{D}=({D}_{jl})$ and the mass matrix $\bm{M}=({M}_{jl})$ respectively by
\begin{equation*}
  {D}_{jl}=L_l'(\xi_j),
 \ \
   {M}_{jl}=\pro{L_j}{L_l}_\omega=\omega_j\delta_{jl},
\end{equation*}
where 
\begin{equation*}
  \pro uv_\omega=\sum_{j=0}^r \omega_ju(\xi_j)v(\xi_j).
\end{equation*}
It is obvious that $L_j(\xi_l)=\delta_{jl}$, $\bm{M}=\text{diag}\{\omega_0,\cdots,\omega_r\}$,
and the matrices $\bm{M}$ and $\bm{D}$ satisfy the following property.

\begin{proposition}
The identity for the matrices $\bm{M}$ and $\bm{D}$
  \begin{equation}\label{eq:SPB01}
    \bm{M}\bm{D}+\bm{D}^\mathrm{T}\bm{M}
    =\bm{B},\
     \end{equation}
holds,  where   $\bm{B}$ is
 the boundary matrix defined by
  \begin{equation*}
    \bm{B}=\text{diag}\{-1,0,\cdots,0,1\}=:\text{diag}\{\tau_0,\cdots,\tau_r\}.
  \end{equation*}
\end{proposition}
 The above property is usually considered as  the summation-by-parts (SBP) property, which
 is a discrete analog of the integration by parts, that is,
 $ \bm{\xi}^T \bm{M}\bm{D} \bm{\eta} +
 \bm{\xi}^T\bm{D}^\mathrm{T}\bm{M}\bm{\eta}
    =\bm{\xi}^T\bm{B}\bm{\eta}$, $\bm{\xi},\bm{\eta}\in\mathbb R^{r+1}$.

Now  the integrals in \eqref{eq:1DDG} can be approximated by using the
$(r+1)$-point Legendre-Gauss-Lobatto quadrature rule and   the
SBP property \eqref{eq:SPB01}  to derive the semi-discrete nodal DG scheme   as follows
\begin{align}\nonumber
   \dfrac{\dd \bU_i^l}{\dd t}=
  &-\dfrac{4}{\Delta x_i}\sum_{p=0}^r {D}_{lp}{\tilde{\bF}_{1}(\bU_i^p,\bU_i^l)}
 -\dfrac{2}{\Delta x_i}\sum_{p=0}^r {D}_{lp}(\Phi_i^{'l})^\mathrm{T}(\Bx)_i^p\\
  &+\dfrac{2}{\Delta x_i}\dfrac{\tau_l}{\omega_l}\left(\bF_{i,1}^{l}-\bF_{i,1,*}^{l}\right)
  +\dfrac{2}{\Delta x_i}\dfrac{\tau_l}{\omega_l}\bs_i^{l}, \
  l=0,\cdots,r, \label{eq:1DES}
\end{align}
where the transformation between $I_i$ and $I$ is $x_i(\xi)=\frac12(x_\xl+x_\xr)+\frac{\xi}{2}\Delta x_i$, and
\begin{align*}
  &\bU_i^l=\bU_h(x_i(\xi_l)),\quad
  \bF_{i,1}^l=\bF_1(\bU_i^l),\quad
  (\Phi_i^{'l})^\mathrm{T}=\Phi'(\bV(\bU_i^l))^\mathrm{T},\quad l=0,1,\cdots,r,\\
  &\vec{\bF}_{i,1,*}
  =\begin{bmatrix} \bF_{i,1,*}^0,\bF_{i,1,*}^1,\cdots,\bF_{i,1,*}^r\end{bmatrix}
  :=\begin{bmatrix} \hat\bF_\xl (\bU_{i-1}^r,\bU_{i}^0),0,\cdots,0,\hat\bF_\xr (\bU_{i}^r,\bU_{i+1}^0)\end{bmatrix},\\
  &\vec{\bs}_i
  =\begin{bmatrix} \bs_{i}^0,\bs_i^1,\cdots,\bs_i^r\end{bmatrix}
  :=\begin{bmatrix} \dfrac12(\Phi_i^{'0})^\mathrm{T}[(\Bx)_{i}^0-(\Bx)_{i-1}^r],0,
  &\cdots,0,-\dfrac12(\Phi_i^{'r})^\mathrm{T}[(\Bx)_{i+1}^0-(\Bx)_i^r]\end{bmatrix}.
\end{align*}
The numerical fluxes {$\tilde{\bF}_{1}(\bU_i^p,\bU_i^l)$}, $l,p=0,1,\cdots, r$, 
are further used to approximate the fluxes in the volume integrals 
in \eqref{eq:1DDG}, respectively. The purpose of this work is to develop
the entropy stable DG scheme, that is, to derive
the two-point entropy  conservative flux $\bF_{1}^{EC}(\bU_i^p,\bU_i^l)$,
{and use $\bF_{1}^{EC}(\bU_i^p,\bU_i^l)$ and the entropy stable flux $\hat\bF^{ES}_{i\pm \frac12}$ to replace respectively  $\tilde{\bF}_{1}(\bU_i^p,\bU_i^l)$ and $\hat\bF_{i\pm \frac12}$ in \eqref{eq:1DES}.}

\begin{definition}
 For a given entropy pair, if  a consistent, symmetric two-point numerical flux $\tilde{\bF}_1(\bU_L,\bU_R)$  satisfies
  \begin{equation}\label{eq:conserFlux}
    \jump{\bV}^\mathrm{T}\cdot {\tilde{\bF}_1}=\jump{\psi_1}-\jump{\Phi}\mean{\Bx},
  \end{equation}
  then we call it an {\em entropy conservative flux}, denoted by $\bF_1^{EC}(\bU_L,\bU_R)$,
  where $\jump{\cdot}=(\cdot)_R-(\cdot)_L$  and  $\mean{\cdot}=\frac12((\cdot)_L+(\cdot)_R)$ denote the jump  and
  the mean value, respectively.
\end{definition}
Notice that the condition \eqref{eq:conserFlux} is different from that in \cite{Tadmor1987},
due to the last term for the special numerical approximation of the source term in \eqref{eq:RMHDsym1D}. 

\begin{definition}
For a given entropy pair, if   a consistent two-point numerical flux $\hat\bF(\bU_L,\bU_R)$ satisfies
  \begin{equation}\label{eq:stableFlux}
    \jump{\bV}^\mathrm{T}\cdot
    \hat\bF+\jump{\Phi}\mean{\Bx}-\jump{\psi_1}\leqslant 0,
  \end{equation}
  then  we  call it an {\em entropy stable flux}, denoted by $\hat\bF^{ES}(\bU_L,\bU_R)$.
\end{definition}

Following \cite{Liu2018} can easily give the following conclusions.
\begin{proposition} 
  If the numerical flux {$\tilde\bF_1(\bU_L,\bU_R)$} satisfies
  \eqref{eq:conserFlux}, 
     then the scheme \eqref{eq:1DES} is entropy conservative in the sense of that
     the identity
  \begin{equation*}
    \dfrac{\dd}{\dd t}\left(\sum_{l=0}^r\dfrac{\Delta x_i}{2}
    \omega_l\eta_l\right)=\mathcal{F}_i^0-\mathcal{F}_i^r,
  \end{equation*}
  holds, where
  \begin{align*}
    &\mathcal{F}_i^r=((\bV^r_i)^\mathrm{T}\bF_{i,1,*}^r-\psi_i^r)+\Phi_i^r\mean{\Bx}_i^r,\\
    &\mathcal{F}_i^0=((\bV^0_i)^\mathrm{T}\bF_{i,1,*}^0-\psi_i^0)+\Phi_i^0\mean{\Bx}_i^0.
  \end{align*}
  Moreover, the scheme is at least $r$-th order accurate measured by local
  truncation errors.
\end{proposition}
\begin{proposition}
 If the numerical fluxes {$\tilde\bF_1(\bU_L,\bU_R)$} and $\hat\bF_{i\pm\frac12}$ in \eqref{eq:1DES} are entropy conservative and stable, respectively,
then the scheme \eqref{eq:1DES} is entropy stable
in the sense  that the inequality
  \begin{equation*}
    \dfrac{\dd}{\dd t}\left(\sum_{l=0}^r\dfrac{\Delta x_i}{2}\omega_l\eta_l\right)
    +\left(\mathcal{Q}_{\xr} - \mathcal{Q}_{\xl} \right)\leqslant 0,
  \end{equation*}
 holds, where
  \begin{align*}
 \mathcal{Q}_{\xr}=\dfrac12\left(\mathcal{F}_i^r + \mathcal{F}_{i+1}^0\right),\quad
 \mathcal{Q}_{\xl}=\dfrac12\left(\mathcal{F}_{i-1}^r + \mathcal{F}_i^0 \right),
 \end{align*}
 which are consistent with the entropy flux $q_1$ in \eqref{eq:entropypair}.
\end{proposition}


\subsection{Two-point entropy conservative flux}\label{section:3.2}
This section  is going to derive the entropy conservative flux satisfying
\eqref{eq:conserFlux} for \eqref{eq:RMHDsym1D}.
The key is to use the identity
\begin{equation}\label{eq:jumpid}
  \jump{ab}=\mean{a}\jump{b}+\mean{b}\jump{a},
\end{equation}
and rewrite the jumps of the entropy variables $\bV$, the
potential $\psi_1$, and $\Phi$ as some linear combinations (the coefficients do
not depend on the jumps) of the jumps of a specially chosen
parameter vector $\widetilde{\bm{V}}$.
For simplicity in derivation, we choose the parameter vector as
\begin{align*}
  \widetilde{\bm{V}}=(\rho,\beta,u^x,u^y,u^z,b^x,b^y,b^z)^\mathrm{T},\ \
  \beta=\rho/p.
\end{align*}
At this point, we have
\begin{align*}
  s=-\ln\beta-(\Gamma-1)\ln\rho,~W=\sqrt{1+(\ux)^2+(\uy)^2+(\uz)^2},~
  b^0=\dfrac{u^kb^k}{W},~k=x,y,z,
\end{align*}
and the jump of $W$ can be expressed as
\begin{equation*}
  \begin{aligned}
    \jump{W}=&\dfrac{\mean{u^k}\jump{u^k}}{\mean{W}},
  \end{aligned}
\end{equation*}
which plays an important role in rewriting the jumps of the other components of $\bV$.
With the help of the above identities, we have
\begin{equation}\label{eq:Vjump}
  \begin{aligned}
    &\jump{\bV_1}=\dfrac{\jump{\rho}}{\meanln{\rho}}+\left(\dfrac{1}{(\Gamma-1)\meanln{\beta}}+1\right)\jump{\beta}
    =:\dfrac{\jump{\rho}}{\meanln{\rho}}+\alpha_0\jump{\beta},\\
    &\jump{\bV_2}=\mean{u^x}\jump{\beta}+\mean{\beta}\jump{u^x},\\
    &\jump{\bV_3}=\mean{u^y}\jump{\beta}+\mean{\beta}\jump{u^y},\\
    &\jump{\bV_4}=\mean{u^z}\jump{\beta}+\mean{\beta}\jump{u^z},\\
    &\jump{\bV_5}=-\mean{W}\jump{\beta}-\dfrac{\mean{\beta}\mean{u^k}}{\mean{W}}\jump{u^k},\\
    &\jump{\bV_6}=\mean{b^x}\jump{\beta}+\mean{\beta}\jump{b^x},\\
    &\jump{\bV_7}=\mean{b^y}\jump{\beta}+\mean{\beta}\jump{b^y},\\
    &\jump{\bV_8}=\mean{b^z}\jump{\beta}+\mean{\beta}\jump{b^z},
  \end{aligned}
\end{equation}
with the logarithmic mean $\meanln{a}= {\jump{a}}/{\jump{\ln{a}}}$
introduced in \cite{Ismail2009}, where its stable numerical implementation can be found.
Next we need to rewrite the right hand side (RHS) of \eqref{eq:conserFlux} as
some linear combinations similar to the above results
\begin{equation*}
  \begin{aligned}
    \jump{\rho\ux+\frac{\beta\ux\bb}{2}}&-\jump{\Phi}\mean{\Bx}
    =\mean{u^x}\jump{\rho}+\mean{\rho}\jump{u^x}\\
    &+\dfrac12\jump{\beta\ux\left(b^kb^k-\dfrac{(u^kb^k)^2}{W^2}\right)}
    -\jump{\dfrac{\beta u^mb^m}{W}}\mean{W\bx-\dfrac{\ux u^lb^l}{W}}\\
    =&\mean{u^x}\jump{\rho}+\mean{\rho}\jump{u^x}
    +\dfrac12\jump{\beta\ux\left(b^kb^k\right)}
    -\jump{\dfrac{\beta u^kb^k}{W}}\mean{W\bx}\\
    &-\dfrac12\jump{\beta\ux\dfrac{(u^kb^k)^2}{W^2}}
    +\jump{\dfrac{\beta u^mb^m}{W}}\mean{\dfrac{\ux u^lb^l}{W}},
    \quad m,l=x,y,z.
  \end{aligned}
\end{equation*}
A special attention should be paid to those terms
because  the parallel magnetic component of
the entropy conservative flux  $\bF_1^{EC}=(F_1,
F_2, F_3, F_4, F_5, F_6, F_7, F_8)^\mathrm{T}$ should be zero for consistency.
From \eqref{eq:Vjump},  it is not difficult to know that the term with the jump of $\bx$ only appears as $F_6\jump{\bx}$ at the left hand side  of \eqref{eq:conserFlux},
thus the term with $\jump{\bx}$ has to be zero at the RHS of \eqref{eq:conserFlux}.
{Different from the non-relativistic case \cite{Chandrashekar2016},
 }
 it is not enough to use the identity \eqref{eq:jumpid} to
rewrite the RHS of \eqref{eq:conserFlux} when one wants to eliminate terms containing
$\jump{\bx}$. Here a more straightforward way is used to handle the RHS of
\eqref{eq:conserFlux}.
By some manipulations, we have
  \begin{align}  \nonumber
\dfrac12\jump{\beta\ux(\bx)^2}&
    -\jump{\dfrac{\beta\ux\bx}{W}}\mean{W\bx}
    =-\dfrac{\bx_L\bx_RW_LW_R}{2}\jump{\dfrac{\beta\ux}{W^2}}
    \\  \nonumber
    &=-\dfrac{\bx_L\bx_RW_LW_R}{2\mean{W^2}}\left(\mean{\ux}\jump{\beta}+\mean{\beta}\jump{\ux}
    -2\mean{\dfrac{\beta\ux}{W^2}}\mean{u^k}\jump{u^k}\right)
    \\
    &=:-\alpha_1\left(\mean{\ux}\jump{\beta}+\mean{\beta}\jump{\ux}
    -2\mean{\dfrac{\beta\ux}{W^2}}\mean{u^k}\jump{u^k}\right),
  \label{eq:RHSjump1a}
    \\  \nonumber
\dfrac12\jump{\beta\ux(b^m)^2}
    &-\jump{\dfrac{\beta u^mb^m}{W}}\mean{W\bx}\\  \nonumber
    =&\dfrac12\mean{(b^m)^2}\mean{\ux}\jump{\beta}+\dfrac12\mean{(b^m)^2}\mean{\beta}\jump{\ux}+\mean{\beta\ux}\mean{b^m}\jump{b^m}
    \\ \nonumber
    &-\dfrac{\mean{W\bx}}{\mean{W}}\Big(\mean{b^m}\mean{u^m}\jump{\beta}+\mean{\beta}\mean{b^m}\jump{u^m}+\mean{\beta u^m}\jump{b^m}
    \\
    &-\mean{\dfrac{\beta u^mb^m}{W}}\dfrac{\mean{u^k}}{\mean{W}}\jump{u^k}\Big),
    \quad m=y,z,\\
-\dfrac12\jump{\dfrac{\beta\ux u^mb^mu^mb^m}{W^2}}
  &  +\jump{\dfrac{\beta u^mb^m}{W}}\mean{\dfrac{\ux u^mb^m}{W}}
  =-\dfrac{\beta_L\beta_R\left(b^m_Lb^m_Ru^m_Lu^m_R\right)}{2W_LW_R}\jump{\dfrac{\ux}{\beta}}
   \nonumber \\
    =&-\dfrac{\beta_L\beta_R\left(b^m_Lb^m_Ru^m_Lu^m_R\right)}{2W_LW_R\mean{\beta}}
    \left(\jump{\ux}-\mean{\dfrac{\ux}{\beta}}\jump{\beta}\right),
    \quad m=x,y,z,\\
 \nonumber
  -\jump{\dfrac{\beta\ux u^mb^mu^lb^l}{W^2}}
  &  +\jump{\dfrac{\beta u^mb^m}{W}}\mean{\dfrac{\ux u^lb^l}{W}}
    +\jump{\dfrac{\beta u^lb^l}{W}}\mean{\dfrac{\ux u^mb^m}{W}}
    \\
    =&-\dfrac{\beta_L\beta_R\left(b^m_Lb^l_Ru^m_Lu^l_R+b^m_Rb^l_Lu^m_Ru^l_L\right)}{2W_LW_R\mean{\beta}}
    \left(\jump{\ux}-\mean{\dfrac{\ux}{\beta}}\jump{\beta}\right),\nonumber
   \\ & \qquad  \ m,l=x,y,z,~m\neq l, \label{eq:RHSjump1d}
  \end{align}
where the symbol with the subscript $L$ (resp. $R$)  denotes
the left (resp. right) state 
and $\jump{\bx}$ does not appear. 
Using the identities \eqref{eq:RHSjump1a}-\eqref{eq:RHSjump1d} gives
  \begin{align}
\nonumber    -\dfrac12\jump{\beta&\ux\dfrac{(u^kb^k)^2}{W^2}}
    +\jump{\dfrac{\beta u^mb^m}{W}}\mean{\dfrac{\ux u^lb^l}{W}}
    \\ \nonumber
    =&\sum_{m=x,y,z}
    \left[-\dfrac12\jump{\dfrac{\beta\ux u^mb^mu^mb^m}{W^2}}
    +\jump{\dfrac{\beta u^mb^m}{W}}\mean{\dfrac{\ux u^mb^m}{W}}\right]
    \\
 \nonumber   &+\sum_{m,l=x,y,z,~m\neq l}
    \dfrac12\left[-\jump{\dfrac{\beta\ux u^mb^mu^lb^l}{W^2}}
      +\jump{\dfrac{\beta u^mb^m}{W}}\mean{\dfrac{\ux u^lb^l}{W}}
    +\jump{\dfrac{\beta u^lb^l}{W}}\mean{\dfrac{\ux u^mb^m}{W}}\right]
    \\ 
    =&-(b^m_Lu^m_L)(b^l_Ru^l_R) \dfrac{\beta_L\beta_R}{2W_LW_R\mean{\beta}}
    \left(\jump{\ux}-\mean{\dfrac{\ux}{\beta}}\jump{\beta}\right) 
    =:-\tau \left(\jump{\ux}-\mean{\dfrac{\ux}{\beta}}\jump{\beta}\right).
 \label{eq:RHSjump2} \end{align}

Substituting the identities
\eqref{eq:Vjump}, \eqref{eq:RHSjump1a}-\eqref{eq:RHSjump1d}, and \eqref{eq:RHSjump2} into \eqref{eq:conserFlux}, and equating the coefficients of the same jump terms gives
\begin{equation*}
  \left\{
    \begin{aligned}
      &F_1=\meanln{\rho}\mean{\ux},\\
      &\alpha_0F_1 +\mean{u^x}{F_2}+\mean{u^y}{F_3}+\mean{u^z}{F_4}
      -\mean{W}F_5+\mean{\bx}F_6+\mean{\by}F_7+\mean{\bz}F_8\\
      &=-\alpha_1\mean{\ux}
      +\dfrac12\left(\mean{(\by)^2} +\mean{(\bz)^2}\right)\mean{\ux}
      -\dfrac{\mean{W\bx}}{\mean{W}}\left(\mean{\by}\mean{\uy}+\mean{\bz}\mean{\uz}\right)
      +\tau\mean{\dfrac{\ux}{\beta}},\\
      &\mean{\beta}F_2-\dfrac{\mean{\beta}\mean{\ux}}{\mean{W}}F_5
      =\mean{\rho}-\alpha_1\mean{\beta}+2\alpha_1\mean{\dfrac{\beta\ux}{W^2}}\mean{\ux}
      +\dfrac12\left(\mean{(\by)^2}+\mean{(\bz)^2}\right)\mean{\beta}\\
      &+\dfrac{\mean{W\bx}}{\mean{W}}\left(\mean{\dfrac{\beta\uy\by}{W}}
      +\mean{\dfrac{\beta\uz\bz}{W}}\right)\dfrac{\mean{\ux}}{\mean{W}}
      -\tau,\\
      &\mean{\beta}F_3-\dfrac{\mean{\beta}\mean{\uy}}{\mean{W}}F_5
      =2\alpha_1\mean{\dfrac{\beta\ux}{W^2}}\mean{\uy}\\
      &-\dfrac{\mean{W\bx}}{\mean{W}}\left(\mean{\beta}\mean{\by}
        -\mean{\dfrac{\beta\uy\by}{W}}\dfrac{\mean{\uy}}{\mean{W}}
      -\mean{\dfrac{\beta\uz\bz}{W}}\dfrac{\mean{\uy}}{\mean{W}}\right),\\
      &\mean{\beta}F_4-\dfrac{\mean{\beta}\mean{\uz}}{\mean{W}}F_5
      =2\alpha_1\mean{\dfrac{\beta\ux}{W^2}}\mean{\uz}\\
      &-\dfrac{\mean{W\bx}}{\mean{W}}\left(\mean{\beta}\mean{\bz}
        -\mean{\dfrac{\beta\uy\by}{W}}\dfrac{\mean{\uz}}{\mean{W}}
      -\mean{\dfrac{\beta\uz\bz}{W}}\dfrac{\mean{\uz}}{\mean{W}}\right),\\
      &\mean{\beta}F_6=0,\\
      &\mean{\beta}F_7=\mean{\beta\ux}\mean{\by}-\dfrac{\mean{W\bx}}{\mean{W}}\mean{\beta\uy},\\
      &\mean{\beta}F_8=\mean{\beta\ux}\mean{\bz}-\dfrac{\mean{W\bx}}{\mean{W}}\mean{\beta\uz},
    \end{aligned}
  \right.
\end{equation*}
from which we can directly get the following four components of the entropy conservative flux
\begin{equation}\label{eq:ecFlux1D1}
  \left\{
    \begin{aligned}
      &F_1=\meanln{\rho}\mean{\ux},\\
      &F_6=0,\\
      &F_7=\dfrac{\mean{\beta\ux}\mean{\by}}{\mean{\beta}}-\dfrac{\mean{W\bx}\mean{\beta\uy}}{\mean{W}\mean{\beta}},\\
      &F_8=\dfrac{\mean{\beta\ux}\mean{\bz}}{\mean{\beta}}
      -\dfrac{\mean{W\bx}\mean{\beta\uz}}{\mean{W}\mean{\beta}}.
    \end{aligned}
  \right.
\end{equation}
The four components $F_2,F_3,F_4,F_5$ of the entropy conservative flux
satisfy the following  system of linear equations
\begin{equation}\label{eq:ecFlux1D2}
  \begin{bmatrix}
    \mean{\ux} & \mean{\uy} & \mean{\uz} & -\mean{W} \\
    \mean{\beta} & 0 & 0 & -\frac{\mean{\beta}\mean{\ux}}{\mean{W}} \\
    0 & \mean{\beta} & 0 & -\frac{\mean{\beta}\mean{\uy}}{\mean{W}} \\
    0 & 0 & \mean{\beta} & -\frac{\mean{\beta}\mean{\uz}}{\mean{W}} \\
  \end{bmatrix}
  \begin{bmatrix}
    F_2 \\ F_3 \\ F_4 \\ F_5 \\
  \end{bmatrix}
  =\begin{bmatrix}
    (RHS)_1 \\ (RHS)_2 \\ (RHS)_3 \\ (RHS)_4
  \end{bmatrix},
\end{equation}
where
\begin{equation*}
  \begin{aligned}
    (RHS)_1=&-\alpha_0F_1-\alpha_1\mean{\ux}
    +\dfrac12\left(\mean{(\by)^2} +\mean{(\bz)^2}\right)\mean{\ux}
    -\dfrac{\mean{W\bx}}{\mean{W}}\left(\mean{\by}\mean{\uy}+\mean{\bz}\mean{\uz}\right)\\
    &+\tau\mean{\dfrac{\ux}{\beta}}-\mean{\by}F_7-\mean{\bz}F_8,\\
    (RHS)_2=&\mean{\rho}-\alpha_1\mean{\beta}+2\alpha_1\mean{\dfrac{\beta\ux}{W^2}}\mean{\ux}
    +\dfrac12\left(\mean{(\by)^2}+\mean{(\bz)^2}\right)\mean{\beta}\\
    &+\dfrac{\mean{W\bx}}{\mean{W}}\left(\mean{\dfrac{\beta\uy\by}{W}}
    +\mean{\dfrac{\beta\uz\bz}{W}}\right)\dfrac{\mean{\ux}}{\mean{W}}
    -\tau,\\
    (RHS)_3=&2\alpha_1\mean{\dfrac{\beta\ux}{W^2}}\mean{\uy}
    -\dfrac{\mean{W\bx}}{\mean{W}}\left(\mean{\beta}\mean{\by}
      -\mean{\dfrac{\beta\uy\by}{W}}\dfrac{\mean{\uy}}{\mean{W}}
    -\mean{\dfrac{\beta\uz\bz}{W}}\dfrac{\mean{\uy}}{\mean{W}}\right),\\
    (RHS)_4=&2\alpha_1\mean{\dfrac{\beta\ux}{W^2}}\mean{\uz}
    -\dfrac{\mean{W\bx}}{\mean{W}}\left(\mean{\beta}\mean{\bz}
      -\mean{\dfrac{\beta\uy\by}{W}}\dfrac{\mean{\uz}}{\mean{W}}
    -\mean{\dfrac{\beta\uz\bz}{W}}\dfrac{\mean{\uz}}{\mean{W}}\right).
  \end{aligned}
\end{equation*}
Solving the system \eqref{eq:ecFlux1D2} can give the explicit expressions of $F_2,F_3,F_4,F_5$
as follows
\begin{equation}\label{eq:ecFlux1D3}
  \left\{
    \begin{aligned}
      F_5&=\mathcal{D}^{-1}\left[\mean{\ux}(RHS)_2+\mean{\uy}(RHS)_3+\mean{\uz}(RHS)_4-\mean{\beta}(RHS)_1\right],\\
      F_2&=\mean{\ux}F_5/\mean{W}+(RHS)_2/\mean{\beta},\\
      F_3&=\mean{\uy}F_5/\mean{W}+(RHS)_3/\mean{\beta},\\
      F_4&=\mean{\uz}F_5/\mean{W}+(RHS)_4/\mean{\beta},
    \end{aligned}
  \right.
\end{equation}
where $\mathcal{D}=\dfrac{\mean{\beta}(\mean{W}^2-\sum\mean{u^k}^2)}{\mean{W}}$.

\begin{thm}
  For the $x$-splitting system  \eqref{eq:RMHDsym1D},
  the flux $\bF_1^{EC}$ presented in \eqref{eq:ecFlux1D1}-\eqref{eq:ecFlux1D3} is entropy conservative  and
  consistent with the physical flux \eqref{eq:RMHDdiv2}.
\end{thm}

\begin{proof}
  First of all, it needs to show that the entropy conservative flux is well defined,
  i.e.,
  \begin{equation}
    \mathcal{D}=\dfrac{\mean{\beta}\left(\mean{W}^2-\sum\mean{u^k}^2\right)}{\mean{W}}>0,
  \end{equation}
  which is equivalent to that $\mean{W}^2-\sum\mean{u^k}^2>0$,
  because $\mean{W}>0$ and $\mean{\beta}>0$.
%
 In fact, using the Cauchy inequality gives
  \begin{align*}
    \mean{W}^2-\sum\mean{u^k}^2&=\dfrac{1}{2}\left(1+\sqrt{1+\sum
    (u^k_L)^2}\sqrt{1+\sum (u^k_R)^2}-\sum u^k_Lu^k_R\right)\\
    &>\dfrac{1}{2}\left(\sqrt{\sum (u^k_L)^2}\sqrt{\sum (u^k_R)^2}-\sum
    u^k_Lu^k_R\right)\geqslant 0.
  \end{align*}

  Next, we  verify that the entropy conservative flux
  \eqref{eq:ecFlux1D1}-\eqref{eq:ecFlux1D2} is consistent with the flux $\bF_1$.
  If letting $\bU_L=\bU_R=\bU$, then we can simplify \eqref{eq:ecFlux1D1} as follows
  \begin{equation}\label{eq:consflux1}
    \left\{
      \begin{aligned}
        F_1=&D\vx,\\
        F_6=&0,\\
        F_7=&\ux\by-\uy\bx=\vx\By-\Bx\vy,\\
        F_8=&\ux\bz-\uz\bx=\vx\Bz-\Bx\vz,
      \end{aligned}
    \right.
  \end{equation}
 and the linear system \eqref{eq:ecFlux1D2} can be rewritten as follows
  \begin{equation*}
    \begin{bmatrix}
      {\ux} & {\uy} & {\uz} & -{W} \\
      {\beta} & 0 & 0 & -\frac{\beta\ux}{W} \\
      0 & {\beta} & 0 & -\frac{\beta\uy}{W} \\
      0 & 0 & {\beta} & -\frac{\beta\uz}{W} \\
    \end{bmatrix}
    \begin{bmatrix}
      F_2 \\ F_3 \\ F_4 \\ F_5 \\
    \end{bmatrix}
    =\begin{bmatrix}
      (RHS)_1 \\ (RHS)_2 \\ (RHS)_3 \\ (RHS)_4
    \end{bmatrix},
  \end{equation*}
  where
  \begin{equation*}
    \begin{aligned}
      (RHS)_1=&-\left(1+\dfrac{p}{(\Gamma-1)\rho}\right)F_1-\dfrac12(\bx)^2\ux
      +\dfrac12\left((\by)^2 +(\bz)^2\right)\ux
      -\bx\left({\by}{\uy}+{\bz}{\uz}\right)\\
      &+\dfrac{\ux}{2}(b^0)^2-{\by}(\ux\by-\uy\bx)-{\bz}(\ux\bz-\uz\bx)\\
      =&-\left(1+\dfrac{p}{(\Gamma-1)\rho}\right)F_1
      -\dfrac12\ux\bb,\\
      (RHS)_2=&{\rho}-\dfrac12\beta(\bx)^2+{\dfrac{\beta(\ux)^2(\bx)^2}{W^2}}
      +\dfrac12\left({(\by)^2}+{(\bz)^2}\right){\beta}
      +\left({\beta\uy\by}
      +{\beta\uz\bz}\right)\dfrac{{\ux}\bx}{{W^2}}
      -\dfrac12\beta(b^0)^2,\\
      =&\beta\left(\pt-\dfrac{\bx\Bx}{W}\right),\\
      (RHS)_3=&2\alpha_1{\dfrac{\beta\ux}{W^2}}{\uy}
      -\bx\left({\beta}{\by}
        -{\dfrac{\beta\uy\by\uy}{W^2}}
      -{\dfrac{\beta\uz\bz\uy}{W^2}}\right)=-\dfrac{\beta\bx\By}{W},\\
      (RHS)_4=&2\alpha_1{\dfrac{\beta\ux}{W^2}}{\uz}
      -\bx\left({\beta}{\bz}
        -{\dfrac{\beta\uy\by\uz}{W}}
      -{\dfrac{\beta\uz\bz\uz}{W}}\right)=-\dfrac{\beta\bx\Bz}{W}.
    \end{aligned}
  \end{equation*}
Operating the row  elimination  gives
  \begin{equation}\label{eq:leqcons}
    \begin{bmatrix}
      0 & 0 & 0 & -1/{W} \\
      1 & 0 & 0 & -{\ux}/{W} \\
      0 & 1 & 0 & -{\uy}/{W} \\
      0 & 0 & 1 & -{\uz}/{W} \\
    \end{bmatrix}
    \begin{bmatrix}
      F_2 \\ F_3 \\ F_4 \\ F_5 \\
    \end{bmatrix}=
    \begin{bmatrix}
      -\rho hW\vx-\ux\bb+\bx(\vB) \\
      \pt-{\bx\Bx}/{W} \\
      -{\bx\By}/{W} \\
      -{\bx\Bz}/{W} \\
    \end{bmatrix},
  \end{equation}
  which implies
  \begin{equation}\label{eq:consflux2}
      F_5=\rho hW^2\vx+\bb W^2\vx-W\bx(\vB)=(\rho hW^2+\BB)\vx-(\vB)\Bx=\mx.
  \end{equation}
  Substituting such $F_5$ back into \eqref{eq:leqcons} yields
  \begin{equation}\label{eq:consflux3}
    \begin{aligned}
      F_2=&\pt-{\bx\Bx}/{W}+\ux\mx/W=\mx\vx-\Bx({\Bx}/W^2+(\vB)\vx)+\pt,\\
      F_3=&-{\bx\By}/{W}+\uy\mx/W=\my\vx-\Bx({\By}/W^2+(\vB)\vy),\\
      F_4=&-{\bx\Bz}/{W}+\uz\mx/W=\mz\vx-\Bx({\Bz}/W^2+(\vB)\vz).
    \end{aligned}
  \end{equation}
  From \eqref{eq:consflux1}, \eqref{eq:consflux2}, and \eqref{eq:consflux3},
  one can see that the entropy conservative flux is consistent.
\end{proof}

\begin{rmk}
The entropy conservative fluxes for the $y$- and $z$-splitting system can be derived by using  the rotational invariance and alternating the indices $x,y,z$ in \eqref{eq:ecFlux1D1}-\eqref{eq:ecFlux1D3}.
\end{rmk}

\begin{rmk}
 The entropy conservative flux presented recently in \cite{Wu2019}
  is consistent, but its parallel magnetic component does not always vanish,
  so that it may lead to a large error in   the parallel component of the magnetic field.
  It will be compared to ours  by some numerical tests in Section \ref{section:Num}.
\end{rmk}


\begin{rmk}
  If the magnetic field disappears, i.e. $\bm{B}\equiv 0$, our entropy conservative flux reduces to
  \begin{equation*}
    \begin{aligned}
      F_1=&\meanln{\rho}\mean{u^x},\\
      F_5=&\left(\mean{W}^2-\sum\mean{u^k}^2\right)^{-1}\mean{W}\left(\mean{\rho}\mean{u^x}/\mean{\beta}+\alpha_0F_1\right),\\
      F_2=&\dfrac{\mean{u^x}}{\mean{W}}F_5+\dfrac{\mean{\rho}}{\mean{\beta}},\\
      F_3=&\dfrac{\mean{u^y}}{\mean{W}}F_5,\
      F_4=\dfrac{\mean{u^z}}{\mean{W}}F_5,\
      F_6=F_7=F_8=0,
    \end{aligned}
  \end{equation*}
  which is the same as the flux  in (3.13) of \cite{Wu2019}.
  Formally, those expressions are simpler than those in \cite{Duan2019},
  where the primitive variable vector is used as the parameter vector.
\end{rmk}

\subsection{Entropy stable fluxes}\label{section:3.3}
For the sake of simplicity, the entropy stable flux is chosen as the Lax-Friedrichs flux
\begin{equation}
  \hat\bF^{ES}(\bU_L,\bU_R)=\dfrac12\left(\bF_1(\bU_L)+\bF_1(\bU_R)\right)-\dfrac12\alpha\left(\bU_R-\bU_L\right),
\end{equation}
where $\alpha=\max\{\abs{\varrho(\bU_L)}, \abs{\varrho(\bU_R)}\}$
and $\varrho(\bU)$ is the spectral radius of $\pd{\bF_1}{\bU}+\Phi'\pd{\Bx}{\bU}$.
\subsection{Time discretization}
The semi-discrete scheme \eqref{eq:1DES} is further approximated
in time by using the following third-order accurate explicit SSP (strong stability preserving) Runge-Kutta method
\begin{equation}
  \label{eq:rk3}
\begin{aligned}
  &\bU^{(1)}=\bU^n+\Delta t \bm{L}(\bU^n),\\
  &\bU^{(2)}=\dfrac34\bU^n+\dfrac14\left(\bU^{(1)}+\Delta t \bm{L}(\bU^{(1)})\right),\\
  &\bU^{n+1}=\dfrac13\bU^n+\dfrac23\left(\bU^{(2)}+\Delta t \bm{L}(\bU^{(2)})\right),
\end{aligned}
\end{equation}
where $\bm{L}(\bU)$ denotes the spatial discrete operator of the DG scheme \eqref{eq:1DES}.

\section{Two-dimensional entropy stable DG schemes}\label{section:MultiD}
This section discusses the extension of the above entropy stable nodal DG schemes to the
two-dimensional special RMHD equations  on
the Cartesian mesh. They will be built on  the approximation of the 2D RMHD equations
with source terms
\begin{equation}\label{eq:RMHDsym2D}
  \pd{\bU}{t}+\pd{\bF_1(\bU)}{x}+\pd{\bF_2(\bU)}{y}
  =-\Phi'(\bV)^\mathrm{T}\pd{\Bx}{x}-\Phi'(\bV)^\mathrm{T}\pd{\By}{y},
\end{equation}
by using the tensor product technique.

The computational domain $\Omega$ is divided into  $N_x\times N_y$ cells,
$I_{i,j}=J_i\times K_j$ with $J_i = [x_\xl,x_\xr]$, $K_j=[y_\yl,y_\yr]$,
$\Delta x_i=x_\xr-x_\xl$, and $\Delta y_j=y_\yr-y_\yl$,
$i=1,\dots,N_x$, $j=1,\dots,N_y$.
The spatial DG approximation space is taken as
\begin{equation}
  \bW^r_h=\left\{\bw_h\in [L^2(\Omega)]^8:\bw_h|_{I_{i,j}}\in
  [P^r(J_i)]^8\otimes[P^r(K_j)]^8,i=1,2,\dots,N_x,j=1,2,\dots,N_y\right\}.
\end{equation}
Following the 1D case, the  spatial DG approximation of \eqref{eq:RMHDsym2D} is to find
$\bU_h\in\bW^r_h$  such that
\begin{align}\label{eq:2DDG}
  \int_{I_{i,j}}\pd{\bU_h^\mathrm{T}}{t}\bw_h\dd x\dd y
  =&\int_{I_{i,j}}\bF_1(\bU_h)^\mathrm{T}\pd{\bw_h}{x}\dd x\dd y
  -\int_{I_{i,j}}\Phi'(\bU_h)\pd{(\Bx)_h}{x}\bw_h\dd x\dd y \nonumber \\
  &+\int_{I_{i,j}}\bF_2(\bU_h)^\mathrm{T}\pd{\bw_h}{y}\dd x\dd y
  -\int_{I_{i,j}}\Phi'(\bU_h)\pd{(\By)_h}{y}\bw_h\dd x\dd y \nonumber \\
  &-\int_{K_j}\left[\hat{\bF}_\xr^\mathrm{T}(y)\bw_h(x_\xr^-,y)
  -\hat{\bF}_\xl^\mathrm{T}(y)\bw_h(x_\xl^+,y)\right]\dd y \nonumber \\
  &-\int_{J_i}\left[\hat{\bG}_\yr^\mathrm{T}(x)\bw_h(x,y_\yr^-)
  -\hat{\bG}_\yl^\mathrm{T}(x)\bw_h(x,y_\yl^+)\right]\dd x \nonumber \\
  &-\int_{K_j}\frac12\Phi'(\bU_h(x_\xr^-,y))\jump{(\Bx)_h}_\xr(y)\bw_h(x_\xr^-,y)\dd y \nonumber \\
  &-\int_{K_j}\frac12\Phi'(\bU_h(x_\xl^+,y))\jump{(\Bx)_h}_\xl(y)\bw_h(x_\xl^+,y)\dd y \nonumber \\
  &-\int_{J_i}\frac12\Phi'(\bU_h(x,y_\yr^-))\jump{(\By)_h}_\yr(x)\bw_h(x,y_\yr^-)\dd x \nonumber \\
  &-\int_{J_i}\frac12\Phi'(\bU_h(x,y_\yl^+))\jump{(\By)_h}_\yl(x)\bw_h(x,y_\yl^+)\dd x,
\end{align}
 for any $\bw_h\in\bW^r_h$ and
$i=1,2,\cdots,N_x,j=1,2,\cdots,N_y$.

If using the tensor product Legendre-Gauss-Lobatto quadrature rule and introducing
the notations
\begin{align*}
  &x_i(\xi)=\frac12(x_\xl+x_\xr)+\frac{\xi}{2}\Delta x_i,~
  y_j(\xi)=\frac12(y_\yl+y_\yr)+\frac{\xi}{2}\Delta y_j,~
  \bU_{i,j}^{l,m}=\bU_h(x_i(\xi_l),y_j(\xi_m)),\\
  &\bF_{i,j,1}^{l,m}=\bF_1(\bU_{i,j}^{l,m}),\quad
  \bF_{i,j,2}^{l,m}=\bF_2(\bU_{i,j}^{l,m}),\quad
  \Phi_{i,j}^{'l,m}=\Phi'(\bV(\bU_{i,j}^{l,m})),\quad l,m=0,1,\cdots,r,\\
  &\vec{\bF}_{1,*,m}^{i,j}
  :=\begin{bmatrix} \bF_{i,j,1,*}^{0,m},\bF_{i,j,1,*}^{1,m},\cdots,\bF_{i,j,1,*}^{r,m}\end{bmatrix}
  =\begin{bmatrix} \hat\bF_\xl(\bU_{i-1,j}^{r,m},\bU_{i,j}^{0,m}),0,\cdots,0,\hat\bF_\xr(\bU_{i,j}^{r,m},\bU_{i+1,j}^{0,m})\end{bmatrix},\\
  &\vec{\bF}_{2,*,l}^{i,j}
  :=\begin{bmatrix} \bF_{i,j,2,*}^{l,0},\bF_{i,j,2,*}^{l,1},\cdots,\bF_{i,j,2,*}^{l,r}\end{bmatrix}
  =\begin{bmatrix} \hat\bG_\yl(\bU_{i,j-1}^{l,r},\bU_{i,j}^{l,0}),0,\cdots,0,\hat\bG_\yr(\bU_{i,j}^{l,r},\bU_{i,j+1}^{l,0})\end{bmatrix},\\
  &\vec{\bs}^{i,j}_m
  :=\begin{bmatrix} \bs_{i,j}^{0,m},\bs_{i,j}^{1,m},\cdots,\bs_{i,j}^{r,m}\end{bmatrix}\\
  &~~~~~=\begin{bmatrix} \dfrac12(\Phi_{i,j}^{'l,0})^\mathrm{T}[(\Bx)_{i,j}^{0,m}-(\Bx)_{i-1,j}^{r,m}],0,
  \cdots,0,-\dfrac12(\Phi_{i,j}^{'r,m})^\mathrm{T}[(\Bx)_{i+1,j}^{0,m}-(\Bx)_{i,j}^{r,m}]\end{bmatrix},\\
  &\vec{\br}^{i,j}_l
  :=\begin{bmatrix} \br_{i,j}^{l,0},\br_{i,j}^{l,1},\cdots,\br_{i,j}^{l,r}\end{bmatrix}\\
  &~~~~~=\begin{bmatrix} \dfrac12(\Phi_{i,j}^{'l,0})^\mathrm{T}[(\By)_{i,j}^{l,0}-(\By)_{i,j-1}^{l,r}],0,
  \cdots,0,-\dfrac12(\Phi_{i,j}^{'l,r})^\mathrm{T}[(\By)_{i,j+1}^{l,0}-(\By)_{i,j}^{l,r}]\end{bmatrix},
\end{align*}
then the 2D semi-discrete  entropy stable nodal DG schemes \eqref{eq:2DDG} for the $(i,j)$-th cell can be derived
as follows
\begin{align}\label{eq:2DES}
  &\dfrac{\dd \bU_{i,j}^{l,m}}{\dd t}
  +\dfrac{4}{\Delta x_i}\sum_{p=0}^r {D}_{lp}\bF_{1}^{EC}(\bU_{i,j}^{p,m},\bU_{i,j}^{l,m})
  +\dfrac{4}{\Delta y_j}\sum_{q=0}^r {D}_{mq}\bF_{2}^{EC}(\bU_{i,j}^{l,q},\bU_{i,j}^{l,m}) \nonumber\\
  &+\dfrac{2}{\Delta x_i}\sum_{p=0}^r {D}_{lp}(\Phi_{i,j}^{'l,m})^{\mathrm{T}}(\Bx)_{i,j}^{p,m}
   +\dfrac{2}{\Delta y_j}\sum_{q=0}^r {D}_{mq}(\Phi_{i,j}^{'l,m})^{\mathrm{T}}(\By)_{i,j}^{l,q} \nonumber\\
   &=\dfrac{2}{\Delta x_i}\dfrac{\tau_l}{\omega_l}\left(\bF_{i,j,1}^{l,m}-\bF_{i,j,1,*}^{l,m}\right)
   +\dfrac{2}{\Delta y_j}\dfrac{\tau_m}{\omega_m}\left(\bF_{i,j,2}^{l,m}-\bF_{i,j,2,*}^{l,m}\right)
   +\dfrac{2}{\Delta x_i}\dfrac{\tau_l}{\omega_l}\bs_{i,j}^{l,m}
  +\dfrac{2}{\Delta y_j}\dfrac{\tau_m}{\omega_m}\br_{i,j}^{l,m},
\end{align}
for $l,m=0,\cdots,r$, where $\bF_{1}^{EC}$ and $\bF_{2}^{EC}$ are the two-point entropy conservative fluxes in the
$x$- and $y$-directions, respectively, while
$\hat\bF_\xr$ and $\hat\bG_\yr$ are the entropy stable fluxes in the $x$- and
$y$-directions, respectively.
The 2D fully discrete explicit nodal DG schemes will be gotten
by approximating the time derivatives in \eqref{eq:2DES}  by using the third-order Runge-Kutta method \eqref{eq:rk3}.

\section{Numerical results}\label{section:Num}
This section conducts several numerical experiments to validate the performance and accuracy
of our entropy stable DG schemes for the 1D and 2D ideal special RMHD problems.
%
In the light of the article length, here only present numerical results for $r=2$ in the DG approximation
space, the CFL number as $0.2$, and  the adiabatic index $\Gamma=5/3$, unless otherwise stated.
\subsection{1D case}
\begin{example}[1D Alfv\'en wave]\label{ex:acc1D}\rm
  This test is used to verify the accuracy.
  The computational domain $[0,1]$ with periodic boundary conditions
  is divided into $N_x$ uniform cells.
  The exact solutions \cite{Zhao2017} are given by
  \begin{align*}
    &\rho(x, t)=1, \quad \vx(x, t)=0,\quad \vy(x, t)=0.1 \sin (2 \pi(x+t / \kappa)),\quad \vz(x, t)=0.1 \cos (2 \pi(x+t / \kappa)),\\
    &\Bx(x, t)=1, \quad \By(x, t)=\kappa \vy(x, t),\quad \Bz(x, t)=\kappa \vz(x, t),\quad  p(x, t)=0.1,
  \end{align*}
  where $\kappa=\sqrt{1+\rho h W^{2}}$.

Table \ref{tab:acc1D} lists the errors and the orders of convergence in $\By$ at $t=1$
obtained by using our entropy stable DG scheme.
It is seen that our scheme gets the $(r+1)$th-order accuracy as expected.
\end{example}

\begin{table}[!ht]
  \centering
  \begin{tabular}{r|cc|cc|cc} \hline
 $N_x$& $\ell^1$ error & order & $\ell^2$ error & order & $\ell^\infty$ error & order \\
 \hline
 20 & 4.395e-05 &  -   & 5.799e-05 &  -   & 1.641e-04 &  -   \\
 40 & 5.398e-06 & 3.03 & 7.575e-06 & 2.94 & 2.260e-05 & 2.86 \\
 80 & 6.685e-07 & 3.01 & 9.696e-07 & 2.97 & 2.951e-06 & 2.94 \\
160 & 8.322e-08 & 3.01 & 1.227e-07 & 2.98 & 3.766e-07 & 2.97 \\
320 & 1.038e-08 & 3.00 & 1.543e-08 & 2.99 & 4.756e-08 & 2.99 \\
 \hline
  \end{tabular}
  \caption{Example \ref{ex:acc1D}: Errors and orders of convergence in $\By$ at $t=1$.}
  \label{tab:acc1D}
\end{table}

\begin{example}[Riemann problem {\uppercase\expandafter{\romannumeral1}}]\label{ex:RP1}\rm
  The initial data of the first 1D Riemann problem \cite{Zhao2017} are
  \begin{equation*}
    (\rho,\vx,\vy,\vz,\Bx,\By,\Bz,p)=\begin{cases}
      (1,~0,~0,~0,~0.5,~1,~0,~1),        &\quad x<0.5, \\
      (0.125,~0,~0,~0,~0.5,-1,~0,~0.1), &\quad x>0.5,
    \end{cases}
  \end{equation*}
  with $\Gamma=2$.
As the time increases, the initial discontinuity will be decomposed into
a left-moving fast rarefaction wave, a slow compound wave, a contact
discontinuity, a right-moving slow shock wave, and a right-moving fast
rarefaction wave. 

The rest-mass density $\rho$, the Lorentz factor $W$, and the $y$ component
of the magnetic field $\By$ at $t=0.4$ obtained by the
entropy stable DG schemes with $800$ cells are shown in Figure \ref{fig:RP1}.
One can see that our numerical solutions (``$\circ$'') are in good agreement with the
reference solutions (solid line) and the discontinuities are well captured.
The reference solutions are obtained by using a first-order finite volume scheme with HLLD flux on a very fine mesh.
It is worth noting that some slightly numerical oscillations are observed, similar to the results obtained by the $P^K$-based non-central DG method in \cite{Zhao2017} for $K=2,3$, although the TVB limiter \cite{Chen2017} has been performed with the parameter $M=10$ on the local characteristic fields. 
On the other hand, as shown in \cite{Chen2017}, the TVB limiter  cannot
guarantee the entropy non-increasing property in the case
that for the system it is performed on the local characteristic fields,
even though it does not destroy the entropy stable property in the scalar case.
It will be given later  the evolution of the total entropy in several 2D examples
to check whether the total entropy decays with time as expected.
\end{example}
\begin{figure}[ht!]
  \begin{subfigure}[b]{0.32\textwidth}
    \centering
    \includegraphics[width=1.0\textwidth, trim=40 0 50 30, clip]{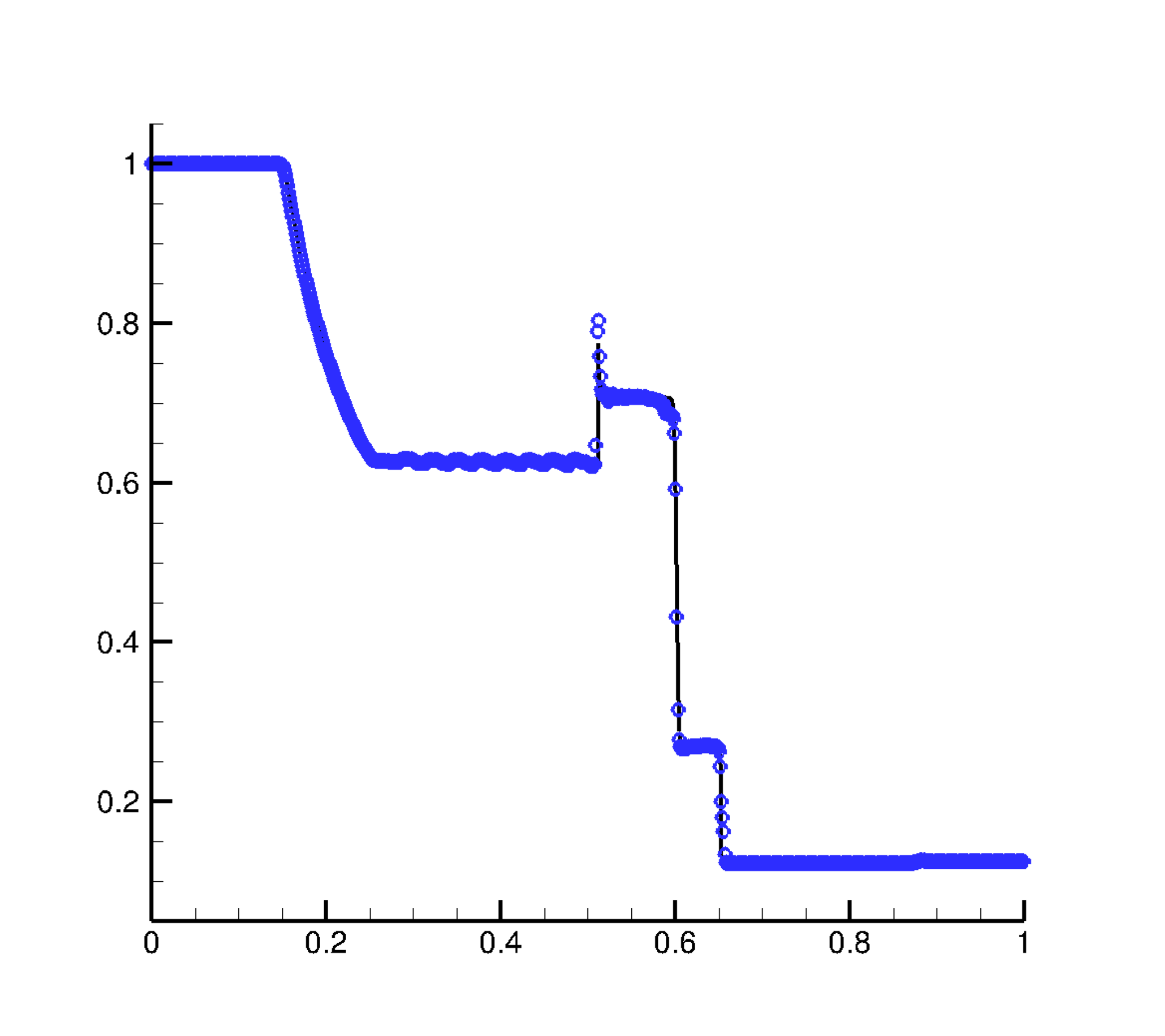}
  \end{subfigure}
  \begin{subfigure}[b]{0.32\textwidth}
    \centering
    \includegraphics[width=1.0\textwidth, trim=40 0 50 30, clip]{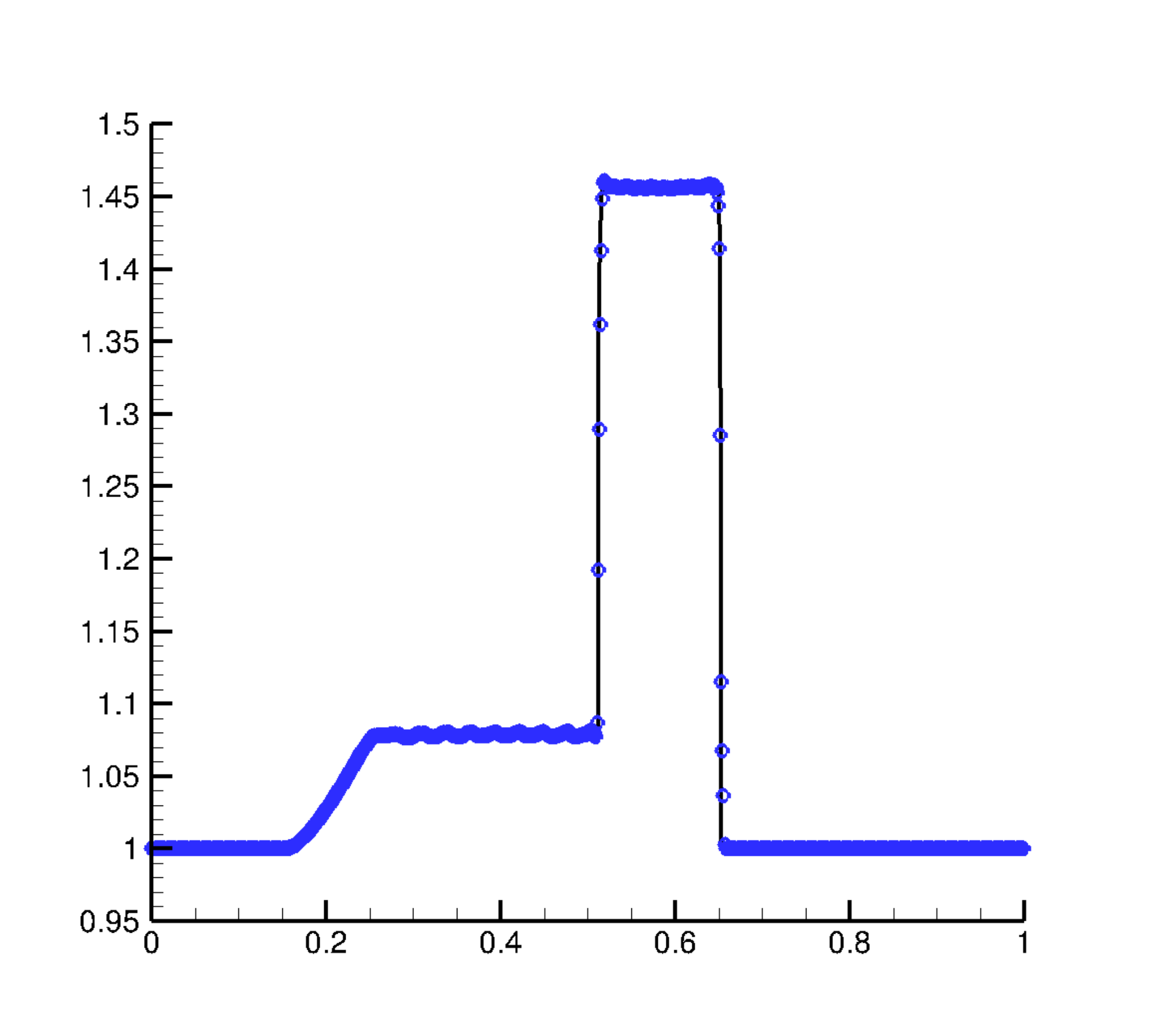}
  \end{subfigure}
  \begin{subfigure}[b]{0.32\textwidth}
    \centering
    \includegraphics[width=1.0\textwidth, trim=40 0 50 30, clip]{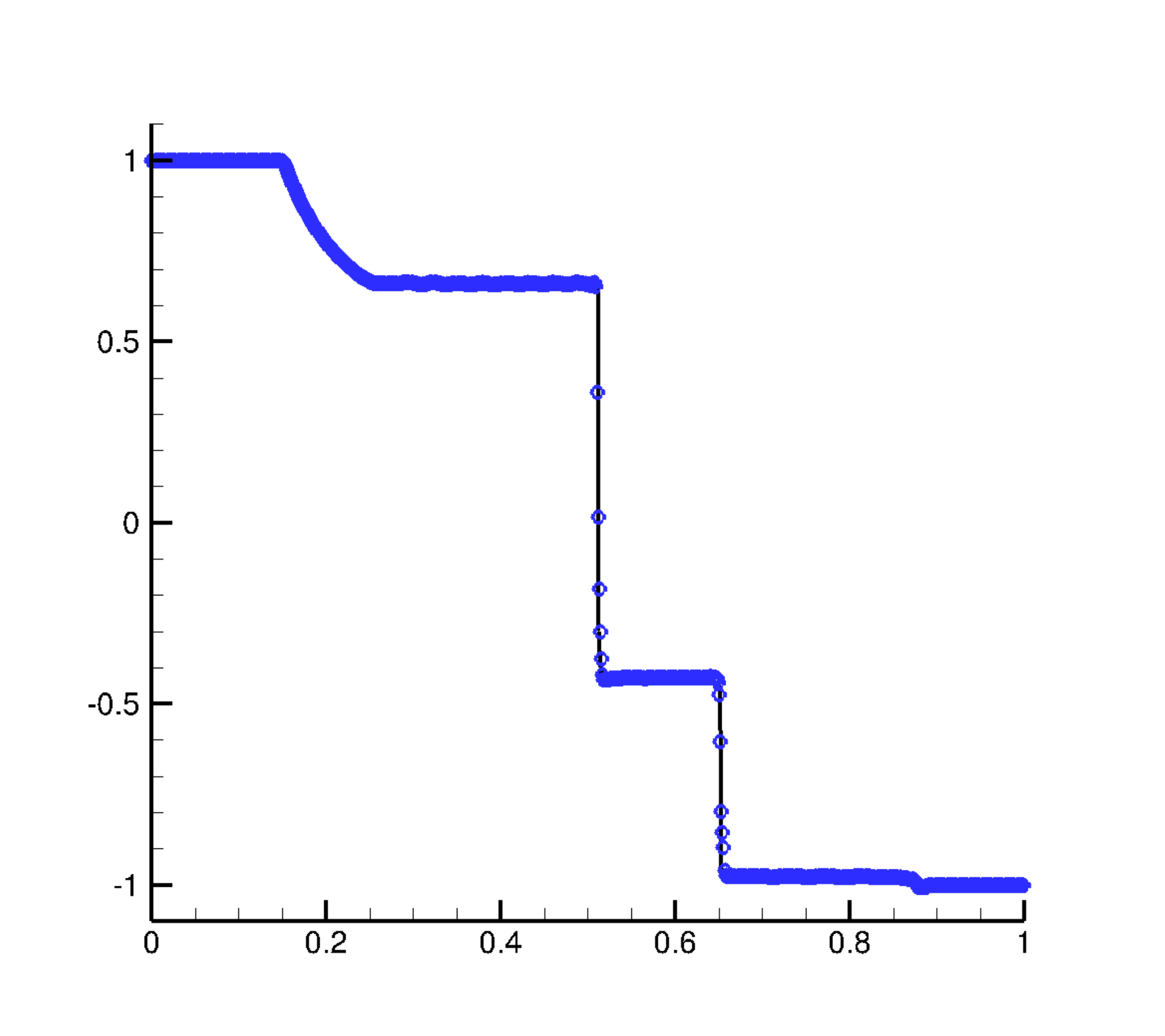}
  \end{subfigure}
  \caption{Example \ref{ex:RP1}:
  The rest-mass density $\rho$, the Lorentz factor $W$, and the $y$ component
of the magnetic field $\By$  at $t=0.4$ with $N_x=800$ cells (from left to right). The symbol ``$\circ$''
and the solid line are the numerical and reference solutions, respectively.}
  \label{fig:RP1}
\end{figure}

\begin{example}[Riemann problem {\uppercase\expandafter{\romannumeral2}}]\label{ex:RP2}\rm
  The initial data of the second 1D Riemann problem \cite{Zhao2017} are
  \begin{equation*}
    (\rho,\vx,\vy,\vz,\Bx,\By,\Bz,p)=\begin{cases}
      (1,~0,~0,~0,~5,~6,~6,~30),        &\quad x<0.5, \\
      (1,~0,~0,~0,~5,~0.7,~0.7,~1), &\quad x>0.5.
    \end{cases}
  \end{equation*}
Figure \ref{fig:RP2}   plots the rest-mass density $\rho$, the $x$ component of the velocity $\vx$,
and the $y$ component of the magnetic field $\By$ at $t=0.4$ obtained by the
entropy stable DG scheme with $800$ cells and the TVB limiter {($M=10$)}.
It can be seen that the solution consists of two left-moving rarefaction waves, a contact
discontinuity, and two right-moving shock waves, and our scheme can still resolve the waves well,
although small overshoot appears near the contact discontinuity.
\end{example}
\begin{figure}[ht!]
  \begin{subfigure}[b]{0.32\textwidth}
    \centering
    \includegraphics[width=1.0\textwidth, trim=40 0 50 30, clip]{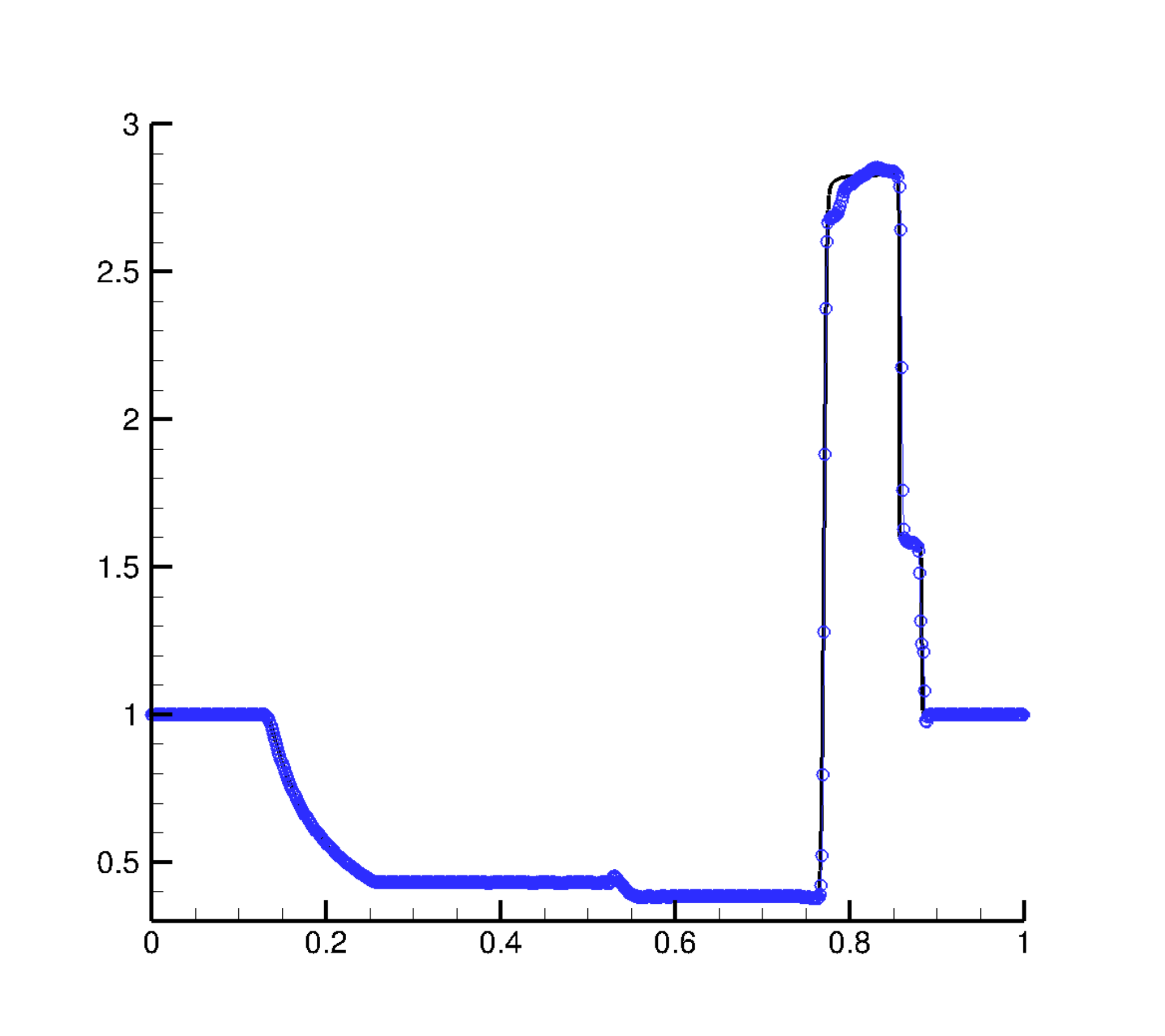}
  \end{subfigure}
  \begin{subfigure}[b]{0.32\textwidth}
    \centering
    \includegraphics[width=1.0\textwidth, trim=40 0 50 30, clip]{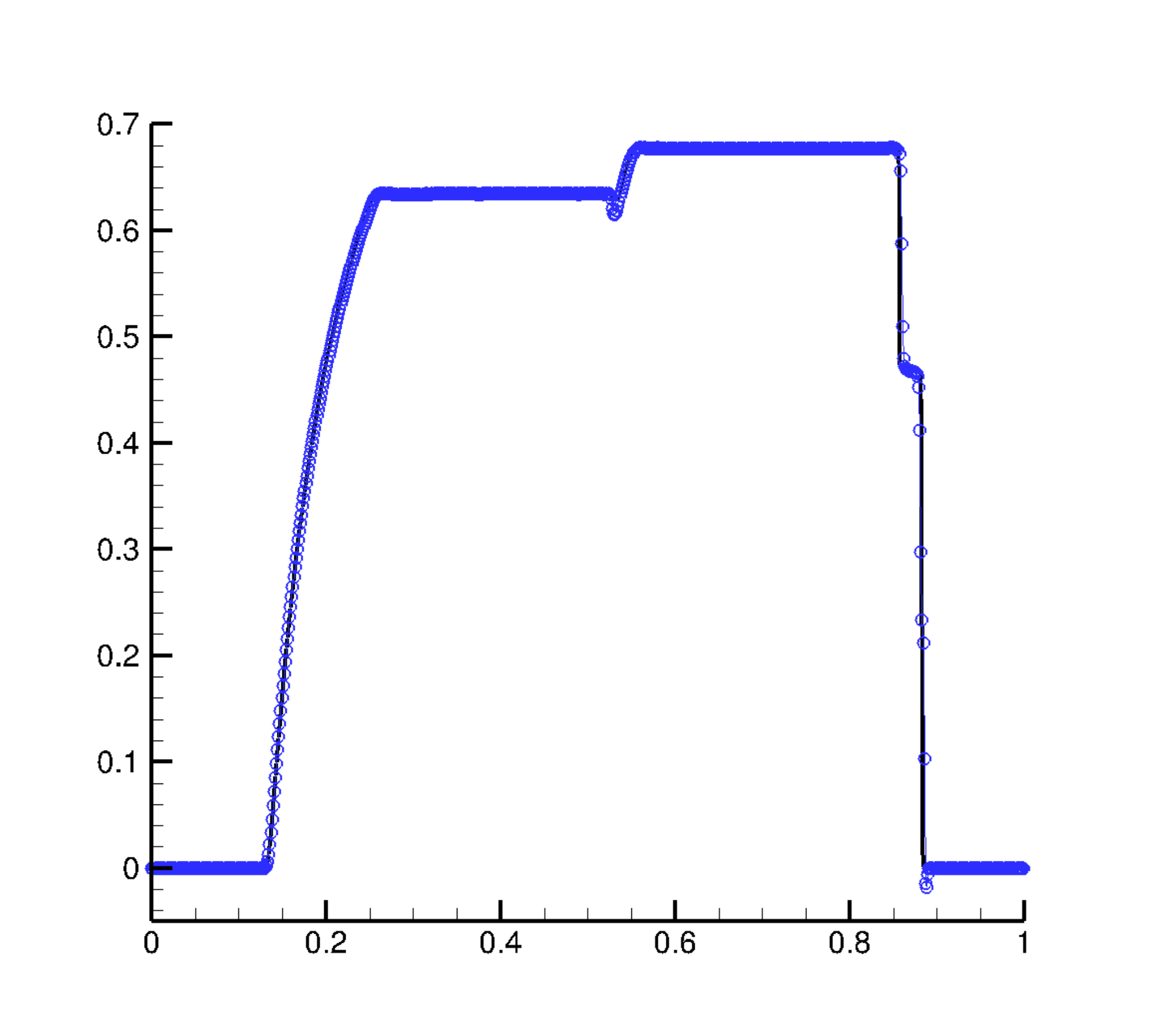}
  \end{subfigure}
  \begin{subfigure}[b]{0.32\textwidth}
    \centering
    \includegraphics[width=1.0\textwidth, trim=40 0 50 30, clip]{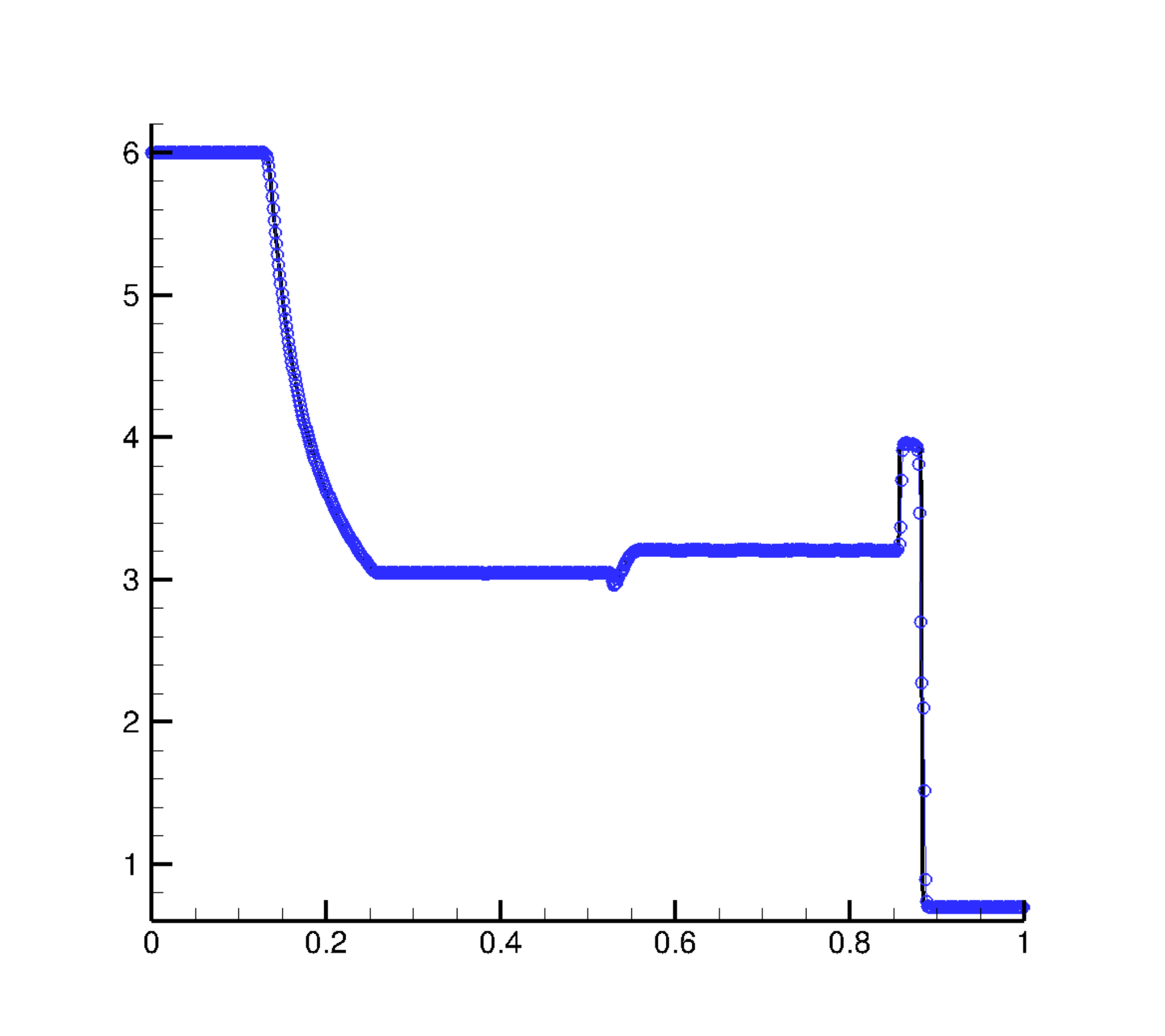}
  \end{subfigure}
  \caption{Example \ref{ex:RP2}:
  The rest-mass density $\rho$, the $x$ component of the velocity $\vx$,
and the $y$ component of the magnetic field $\By$ at $t=0.4$
 with $800$ cells (from left to right). The symbol ``$\circ$''
and the solid line are the numerical    and   reference solutions, respectively. }
  \label{fig:RP2}
\end{figure}

\begin{example}[Riemann problem {\uppercase\expandafter{\romannumeral3}}]\label{ex:RP3}\rm
  The initial data of the third 1D Riemann problem \cite{Mignone2009} are
  \begin{equation*}
    (\rho,\vx,\vy,\vz,\Bx,\By,\Bz,p)=\begin{cases}
      (1,~0,~0.3,~0.4,~1,~6,~2,~5), &\quad x<0.5, \\
      (0.9,~0,~0,~0,~1,~5,~2,~5.3), &\quad x>0.5.
    \end{cases}
  \end{equation*}
The initial discontinuity will break into seven waves:
a fast rarefaction wave, a rotational wave, and a slow shock wave moving to the
left of the contact discontinuity, and a slow shock wave, an Alfv\'en wave, and a
fast shock wave moving to the right of the contact discontinuity.
Figure \ref{fig:RP3} shows the rest-mass density $\rho$, the $x$ component of
the velocity $\vx$, and the $y$ component of the magnetic field $\By$ at $t=0.4$ obtained by the
entropy stable DG scheme with $800$ cells and the TVB limiter {($M=10$)}.
From those plots, we can see that our scheme can capture the discontinuities
well, as well as two narrow regions between
{the rotational wave  at $x\approx 0.455$ and the slow
shock wave at $x\approx 0.468$,
and the slow shock wave at $x\approx 0.558$ and the Alfv\'en wave at
$x\approx 0.565$}.
\end{example}
\begin{figure}[ht!]
  \begin{subfigure}[b]{0.32\textwidth}
    \centering
    \includegraphics[width=1.0\textwidth, trim=40 0 50 30, clip]{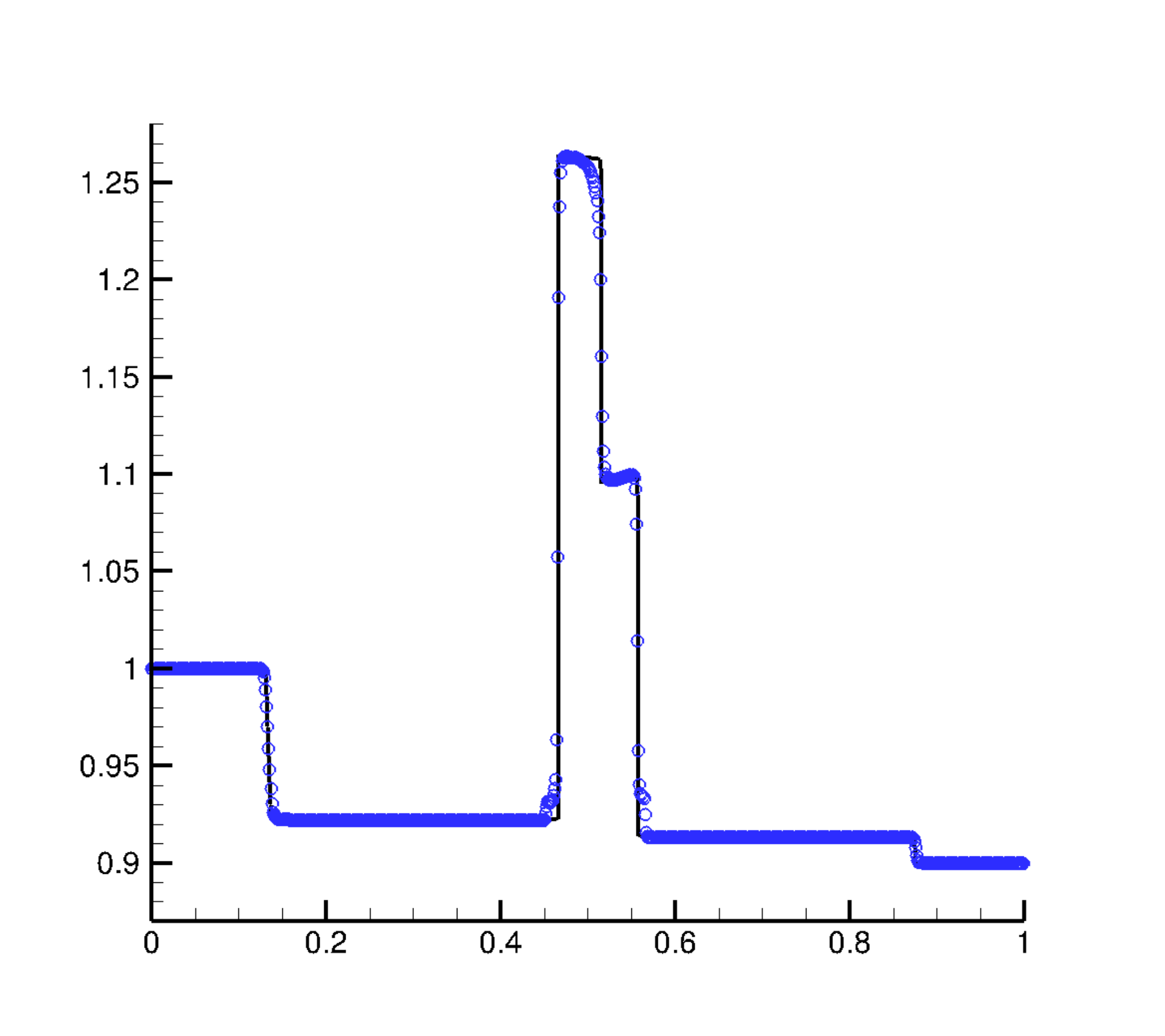}
  \end{subfigure}
  \begin{subfigure}[b]{0.32\textwidth}
    \centering
    \includegraphics[width=1.0\textwidth, trim=40 0 50 30, clip]{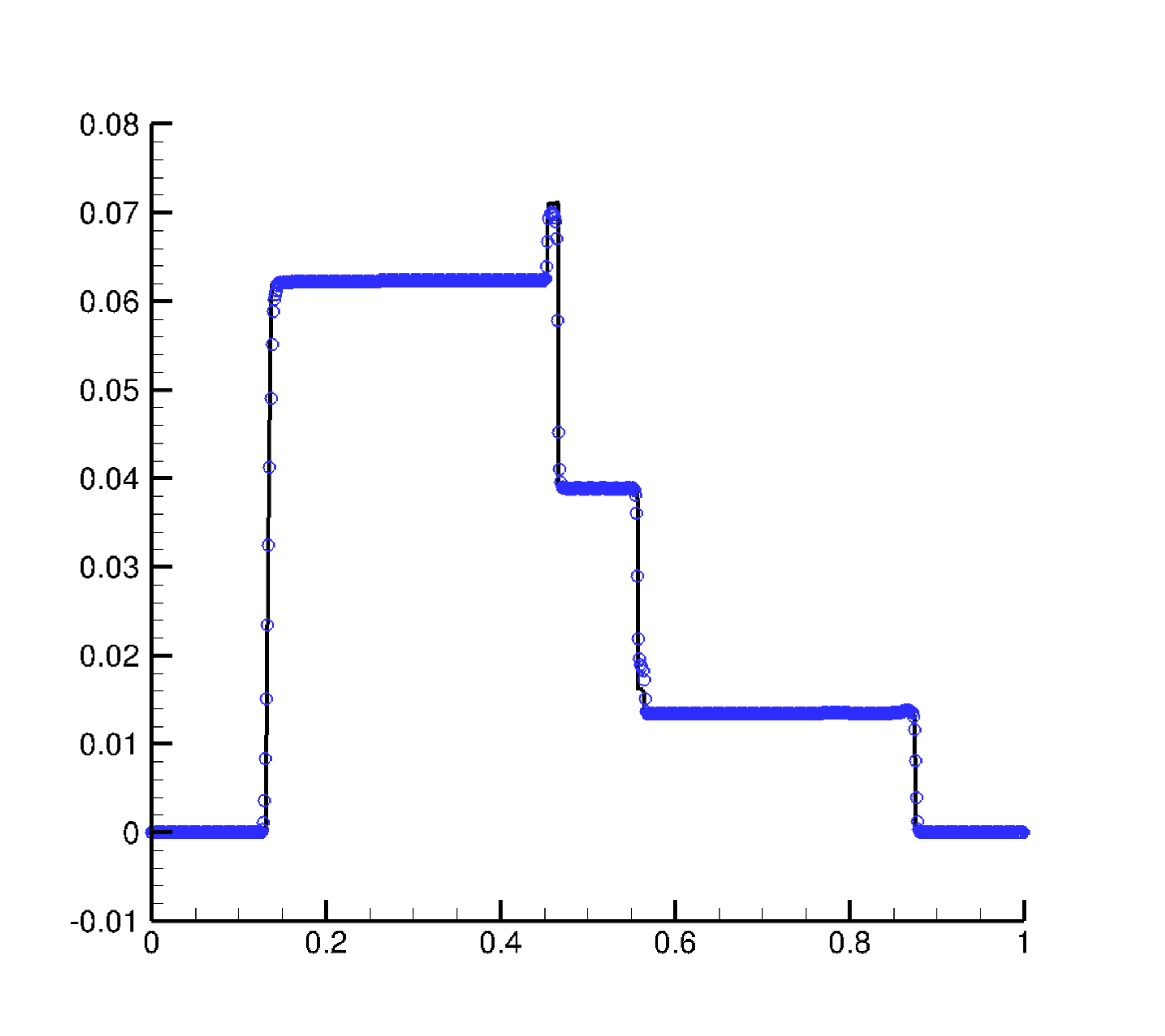}
  \end{subfigure}
  \begin{subfigure}[b]{0.32\textwidth}
    \centering
    \includegraphics[width=1.0\textwidth, trim=40 0 50 30, clip]{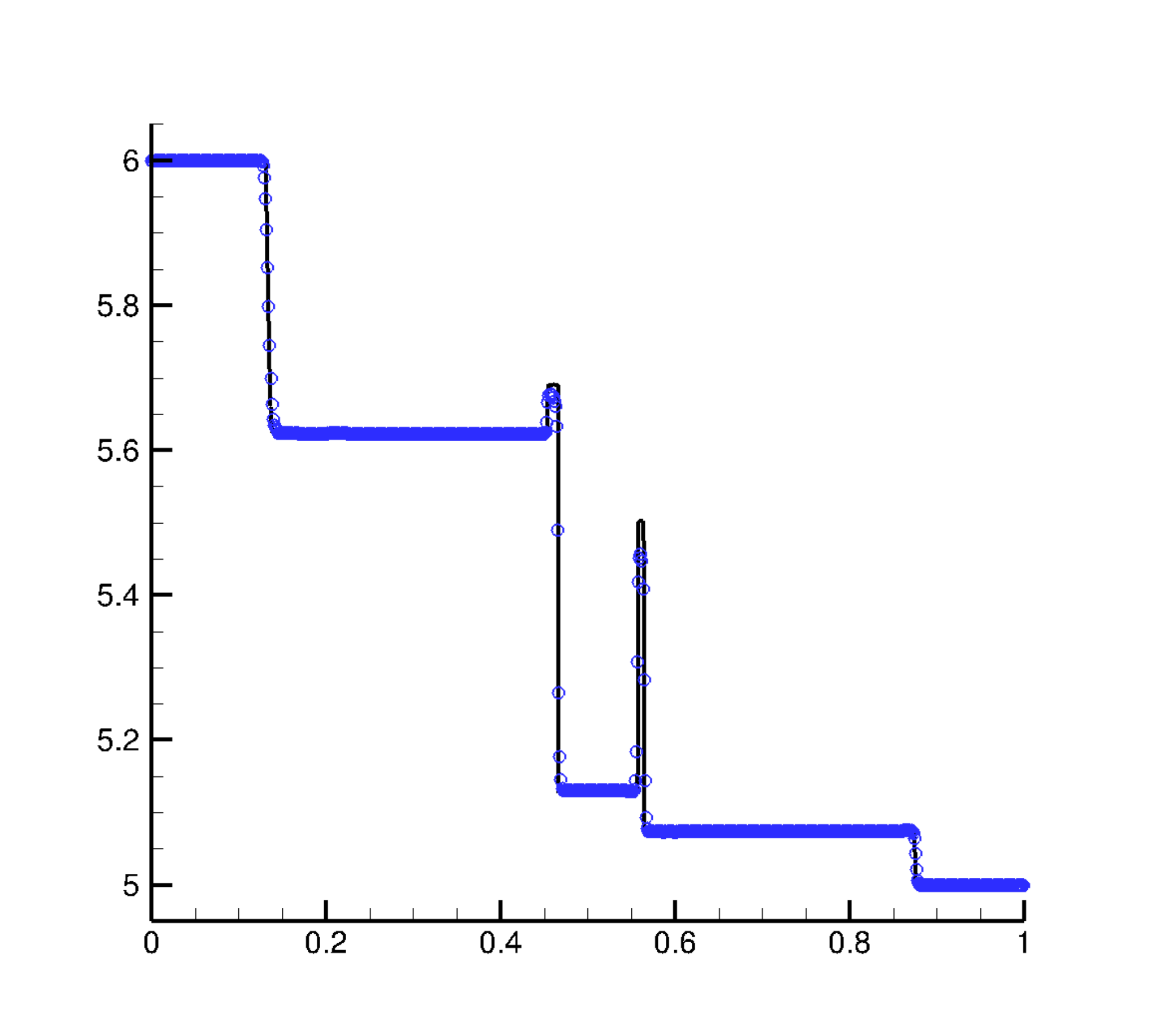}
  \end{subfigure}
  \caption{Example \ref{ex:RP3}: The rest-mass density $\rho$, the $x$ component of
the velocity $\vx$, and the $y$ component of the magnetic field $\By$ at $t=0.4$ with  $800$ cells {(from left to right). The symbol ``$\circ$''
and the solid line are the numerical and reference solutions, respectively.}}
  \label{fig:RP3}
\end{figure}

\subsection{2D case}
\begin{example}[2D Alfv\'en wave]\label{ex:acc2D}\rm
  This test is used to verify the accuracy of the 2D entropy stable scheme.
  The computational domain   $[0,2/\sqrt{3}]\times[0,2]$ with periodic boundary conditions
  is divided into $N_x\times N_y$ uniform cells.
  A sine wave is propagating in the direction $(\sqrt{3}/2,1/2)$, i.e., the angle
  with the $x$-axis is $\alpha=\pi/6$.
  The exact solutions \cite{Zhao2017} are given by
  \begin{align*}
    &\rho(x,y,t)=1, \quad \vx(x,y,t)=-0.1\sin (2\pi(\xi+t/\kappa))\sin\alpha, \\
    &\vy(x,y,t)=0.1\sin (2\pi(\xi+t/\kappa))\cos\alpha,\quad \vz(x,y,t)=0.1\cos(2\pi(\xi+t/\kappa)), \\
    &\Bx(x,y,t)=\cos\alpha+\kappa\vx(x,y,t), \quad \By(x,y,t)=\sin\alpha+\kappa\vy(x,y,t), \\
    &\Bz(x,y,t)=\kappa\vz(x,y,t), \quad p(x,y,t)=0.1,
  \end{align*}
  where $\xi=x\cos\alpha+y\sin\alpha,\kappa=\sqrt{1+\rho h W^{2}}$.

Table \ref{tab:acc2D} lists the errors and the orders of convergence in $\By$ at $t=1$
obtained by using our entropy stable DG scheme.
It is seen that the orders of convergence are half order less than the optimal
order. It is similar to the phenomenon observed in \cite{Chen2017}. The possible reason is the
inaccurate numerical quadrature rule used in the element.
\end{example}
\begin{table}[!ht]
  \centering
  \begin{tabular}{r|cc|cc|cc} \hline
 $N_x=N_y$& $\ell^1$ error & order & $\ell^2$ error & order & $\ell^\infty$ error & order \\
 \hline
 10 & 4.180e-04 &  -   & 4.994e-04 &  -   & 1.608e-03 &  -   \\
 20 & 5.550e-05 & 2.91 & 6.741e-05 & 2.89 & 2.313e-04 & 2.80 \\
 40 & 7.054e-06 & 2.98 & 8.820e-06 & 2.93 & 2.896e-05 & 3.00 \\
 80 & 9.656e-07 & 2.87 & 1.192e-06 & 2.89 & 4.674e-06 & 2.63 \\
160 & 1.526e-07 & 2.66 & 1.964e-07 & 2.60 & 7.962e-07 & 2.55 \\
 \hline
  \end{tabular}
  \caption{Example \ref{ex:acc2D}: Errors and orders of convergence in $\By$ at $t=1$.}
  \label{tab:acc2D}
\end{table}

\begin{example}[Isentropic vortex]\label{ex:vortex}\rm
The   2D relativistic isentropic vortex problem constructed in
  \cite{Balsara2016} is used to further test the accuracy and performance of our 2D scheme.
  The computational domain   $[-5,5]^2$ with periodic boundary conditions is divided into $N_x\times N_y$ uniform cells.
  The setup of a steady vortex centered at $(0,0)$ in a coordinate system $S$
  with the space-time coordinates $(t,x,y)$ is
  \begin{align*}
    &(\vx, \vy)=v_{\text{max}}^{\phi} e^{0.5(1-r^2)}(-y, x),\\
    &(\Bx, \By)=B_{\text{max}}^{\phi} e^{0.5(1-r^2)}(-y, x),
  \end{align*}
  with $v_{\text{max}}^\phi=B_{\text{max}}^\phi=0.7$ and $r=\sqrt{x^2+y^2}$.
  The vortex is isentropic so that $p=\rho^\Gamma$, and then the pressure can be solved by
  \begin{align*}
    r\dfrac{\dd \pt}{\dd r}=\left(\rho h+\bb\right) W^{2} |v|^2-\bb, \quad
    \pt(0)=1.
  \end{align*}
  Next, assume that a coordinate system $S'$ with the spacetime coordinates $(t',x',y')$
  is in motion relative to the coordinate system $S$ with a constant velocity of magnitude
  $w$ along the $(1,1)$ direction, from the perspective of an observer stationary
  in $S$. The relationship between the  coordinate systems $S$ and $S'$ is given by the Lorentz
  transformation  as follows
  \begin{align*}
    &\gamma = \dfrac{1}{\sqrt{1-w^2}},\quad t=\gamma\big(t'+\dfrac{w}{\sqrt{2}}(x'+y')\big),\\
    &x=x'+\dfrac{\gamma-1}{2}(x'+y')+\dfrac{\gamma t'w}{\sqrt{2}},
    \quad y=y'+\dfrac{\gamma-1}{2}(x'+y')+\dfrac{\gamma t'w}{\sqrt{2}}.
  \end{align*}
  Corresponding transformations between the velocities and the magnetic field are
  given by
  \begin{align*}
    \vx'&=\dfrac{1}{1-\frac{w(\vx+\vy)}{\sqrt{2}}}\left[\dfrac{\vx}{\gamma}
    -\dfrac{w}{\sqrt{2}}+\dfrac{\gamma w^2}{2(\gamma+1)}(\vx+\vy)\right],\\
    \vy'&=\dfrac{1}{1-\frac{w(\vx+\vy)}{\sqrt{2}}}\left[\dfrac{\vy}{\gamma}
    -\dfrac{w}{\sqrt{2}}+\dfrac{\gamma w^2}{2(\gamma+1)}(\vx+\vy)\right],
  \end{align*}
  and
  \begin{align*}
    \Bx'&=\Bx + \frac{\gamma-1}{2}(\Bx-\By),\\
    \By'&=\By - \frac{\gamma-1}{2}(\Bx-\By),
  \end{align*}
  respectively.
  Using those transformations  gives a time-dependent solution
  $(\rho',\vx',\vy',\Bx',\By',p')$ in the coordinate system $S'$.
  This test describes a RMHD vortex moves with a constant speed of magnitude
  $w$ in $(-1,-1)$ direction.

We choose $w=0.5\sqrt2$, and the output time is $t=20$ so that the vortex
travels and returns to the original position after a period.
The errors in the mass density $D$ and
corresponding orders of convergence listed in Table \ref{tab:vortex}   show
 that the orders of convergence
of the present entropy stable DG scheme are nearly $2.5$ as the mesh is refined.
Figure \ref{fig:vortex} plots the contours of the rest-mass density $\rho$ and the
magnitude of the magnetic field $|\bm{B}|$ with $40$ equally spaced contour lines.
The results show that due to the Lorentz contraction, the vortex becomes
elliptical, and our scheme can preserve the shape of the vortex well after a whole period.
Figure \ref{fig:decay} presents the evolutions of the total entropy
$\int_\Omega\eta(\bU_{i,j})\dd x\dd y$
with respect to the time obtained by the entropy stable DG scheme with
different resolutions. We can see that the total entropy decay as expected  and
 they converge as the resolution increases.
\end{example}
\begin{table}[!ht]
  \centering
  \begin{tabular}{r|cc|cc|cc} \hline
 $N_x=N_y$& $\ell^1$ error & order & $\ell^2$ error & order & $\ell^\infty$ error & order \\
 \hline
 20 & 4.239e-02 &  -   & 1.142e-01 &  -   & 1.012e+00 &  -   \\
 40 & 6.172e-03 & 2.78 & 1.878e-02 & 2.60 & 2.151e-01 & 2.23 \\
 80 & 5.333e-04 & 3.53 & 1.971e-03 & 3.25 & 3.297e-02 & 2.71 \\
160 & 6.180e-05 & 3.11 & 2.507e-04 & 2.97 & 5.476e-03 & 2.59 \\
320 & 9.280e-06 & 2.74 & 4.221e-05 & 2.57 & 1.008e-03 & 2.44 \\
 \hline
  \end{tabular}
  \caption{Example \ref{ex:vortex}: Errors and orders of convergence in $D$ at $t=20$.}
  \label{tab:vortex}
\end{table}

\begin{figure}[ht!]
  \begin{subfigure}[b]{0.5\textwidth}
    \centering
    \includegraphics[width=1.0\textwidth, trim=10 40 70 20, clip]{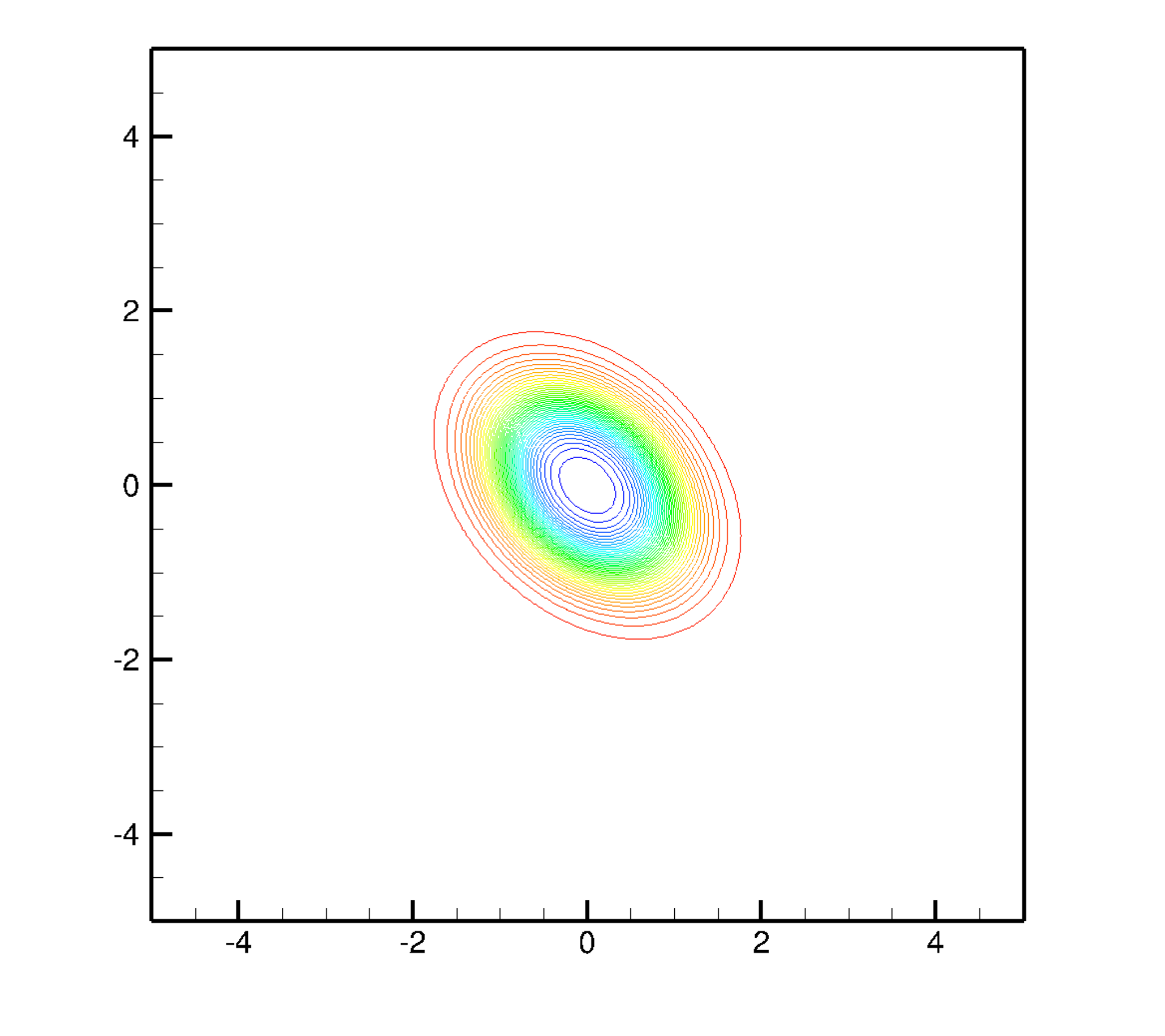}
  \end{subfigure}
  \begin{subfigure}[b]{0.5\textwidth}
    \centering
    \includegraphics[width=1.0\textwidth, trim=10 40 70 20, clip]{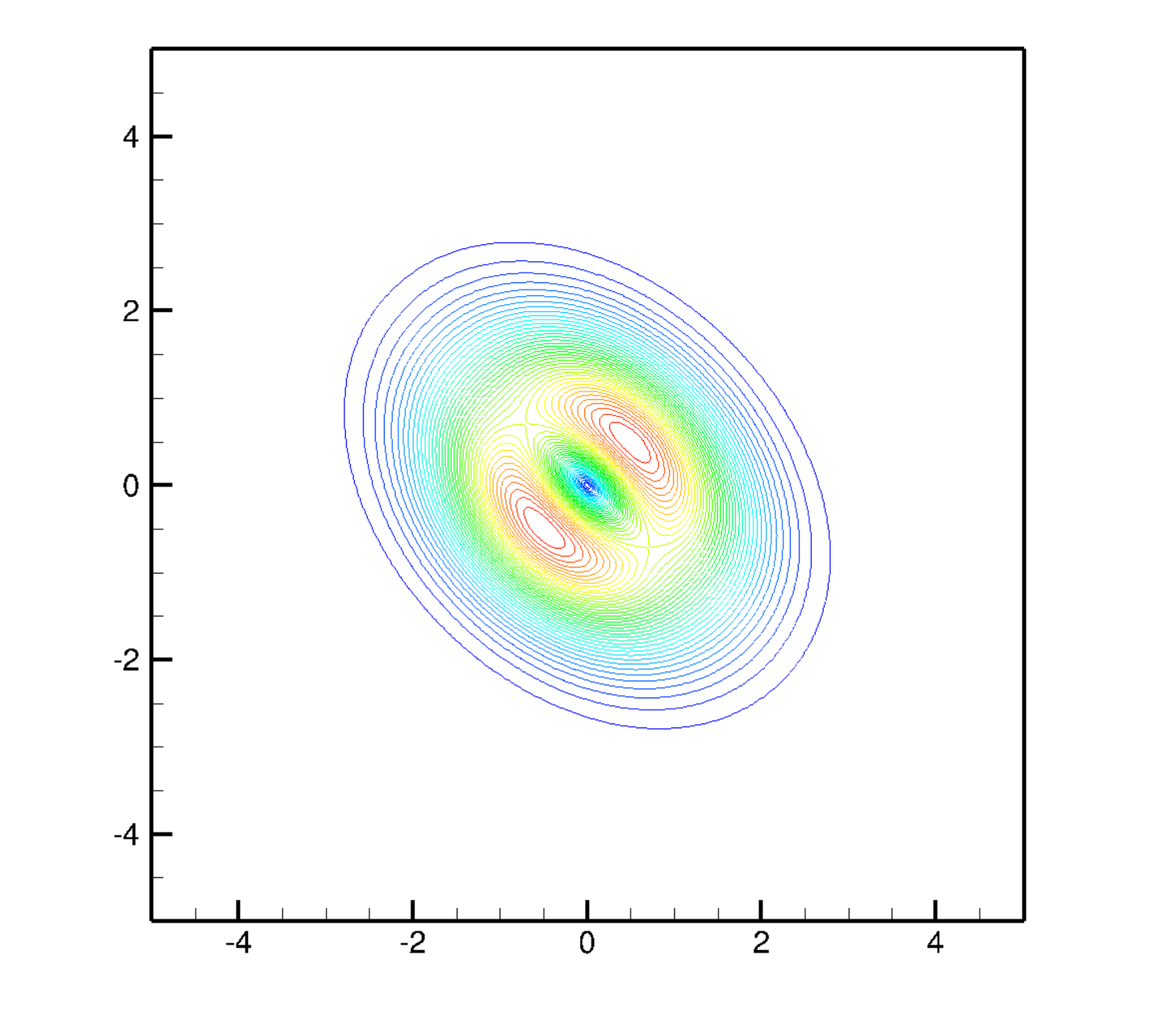}
  \end{subfigure}
  \caption{Example \ref{ex:vortex}: The rest-mass density $\rho$ and the magnitude of the magnetic field $|B|$ (from left to right) at $t=20$ with $40$ equally spaced contour lines   obtained
    by using the entropy stable scheme with $N_x=N_y=320$.}
  \label{fig:vortex}
\end{figure}

\begin{figure}[!ht]
  \centering
  \includegraphics[width=0.50\textwidth, trim=20 40 50 50, clip]{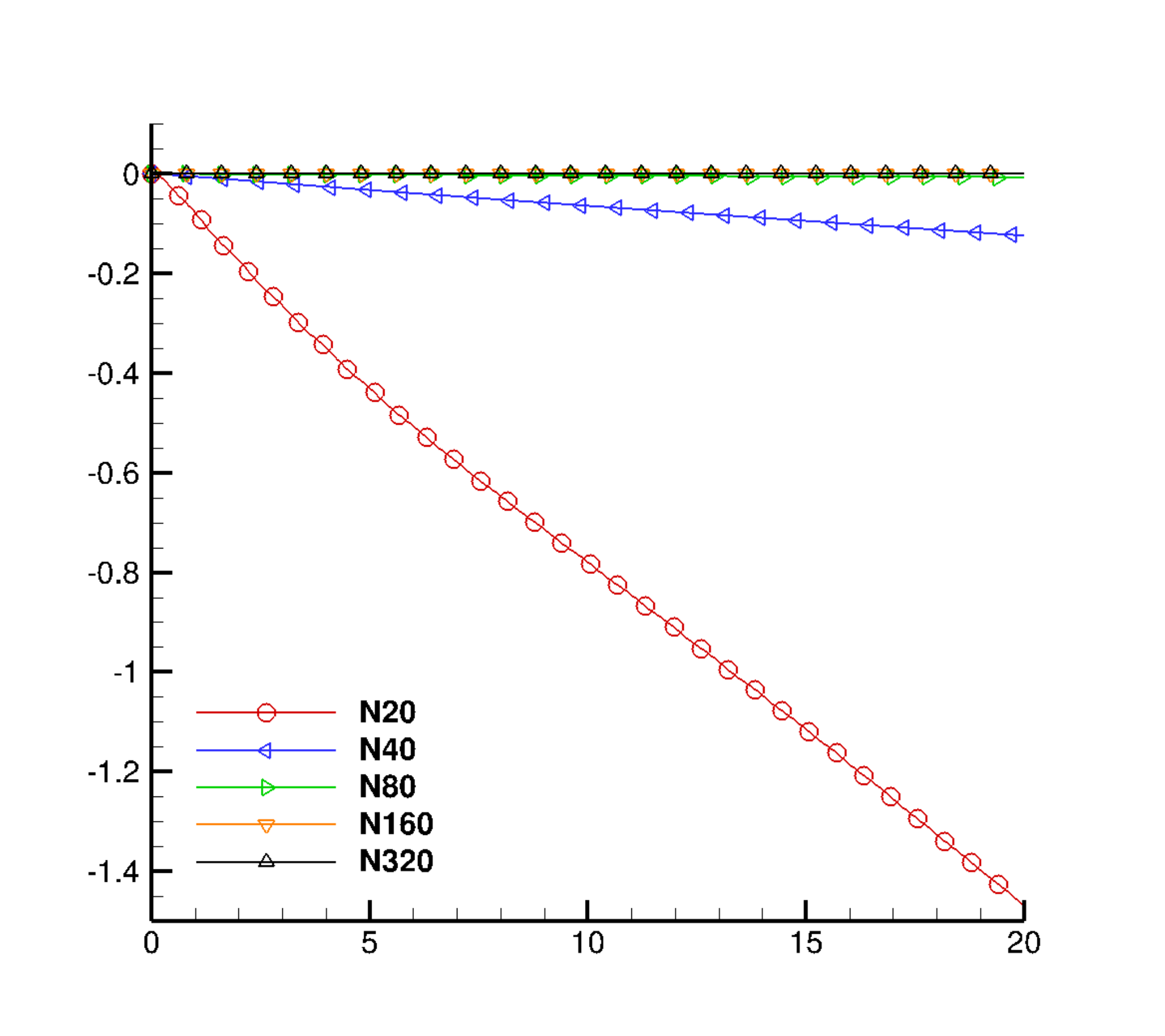}
  \caption{Example \ref{ex:vortex}: The time evolution of the total entropy obtained
    by using the entropy stable DG scheme  with different spatial resolutions of $N_x=N_y=20,40,80,160,320$.}
  \label{fig:decay}
\end{figure}

\begin{example}[Orszag-Tang problem]\label{ex:OrszagTang}\rm
  It is a benchmark test for the RMHD equations \cite{Zhao2017}. The
  initial data   are taken as follows
  \begin{align*}
    & {\rho(x, y, 0)=\frac{25}{36 \pi},\quad \vx(x, y, 0)=0.5 \sin (2 \pi y), \quad \vy(x, y, 0)=0.5 \sin (2 \pi x)}, \\
    & {\vz(x, y, 0)=0, \quad \Bx(x, y, 0)=-\frac{1}{\sqrt{4 \pi}} \sin (2 \pi y), \quad \By(x, y, 0)=\frac{1}{\sqrt{4 \pi}} \sin (4 \pi x)}, \\
    & {\Bz(x, y, 0)=0, \quad p(x, y, 0)=\frac{5}{12 \pi}}.
  \end{align*}
  The computational domain is $[0,1]^2$ with periodic boundary conditions.
 As time increases, complex wave patterns will emerge and the solution will present turbulent behavior.

 In order to get a better performance of the nodal DG scheme, for this test and the following tests, we will first employ the KXRCF discontinuity
indicator \cite{Krivodonova2004} to detect the ``trouble cells'',
and then use the TVB limiter to modify the nodal values in   the ``trouble
cells''. 
Moreover, the physical-constraints-preserving limiter \cite{wu2017b} is
also used to guarantee that the numerical solutions are in the physical admissible state
set. Figure \ref{fig:OrszagTang} shows the contours of the rest-mass density $\rho$
and the Lorentz factor $W$ at $t=1$ with $40$ equally spaced contour lines with the TVB limiter
parameter $M=10$.
It can be seen   that our scheme can resolve the wave patterns
well and the results are comparable to those in \cite{Zhao2017}.
\end{example}
\begin{figure}[ht!]
  \begin{subfigure}[b]{0.5\textwidth}
    \centering
    \includegraphics[width=1.0\textwidth, trim=40 40 70 50, clip]{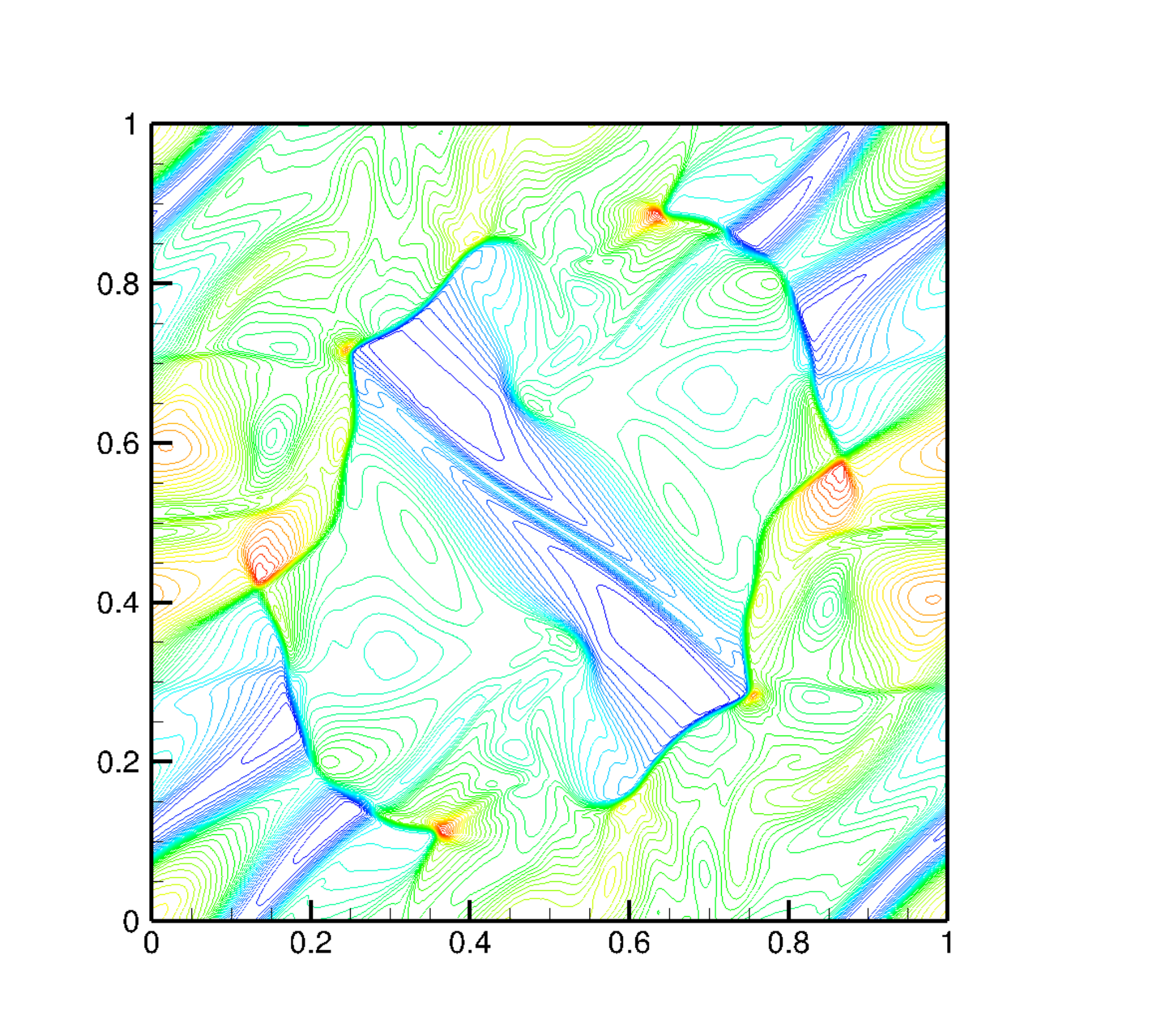}
  \end{subfigure}
  \begin{subfigure}[b]{0.5\textwidth}
    \centering
    \includegraphics[width=1.0\textwidth, trim=40 40 70 50, clip]{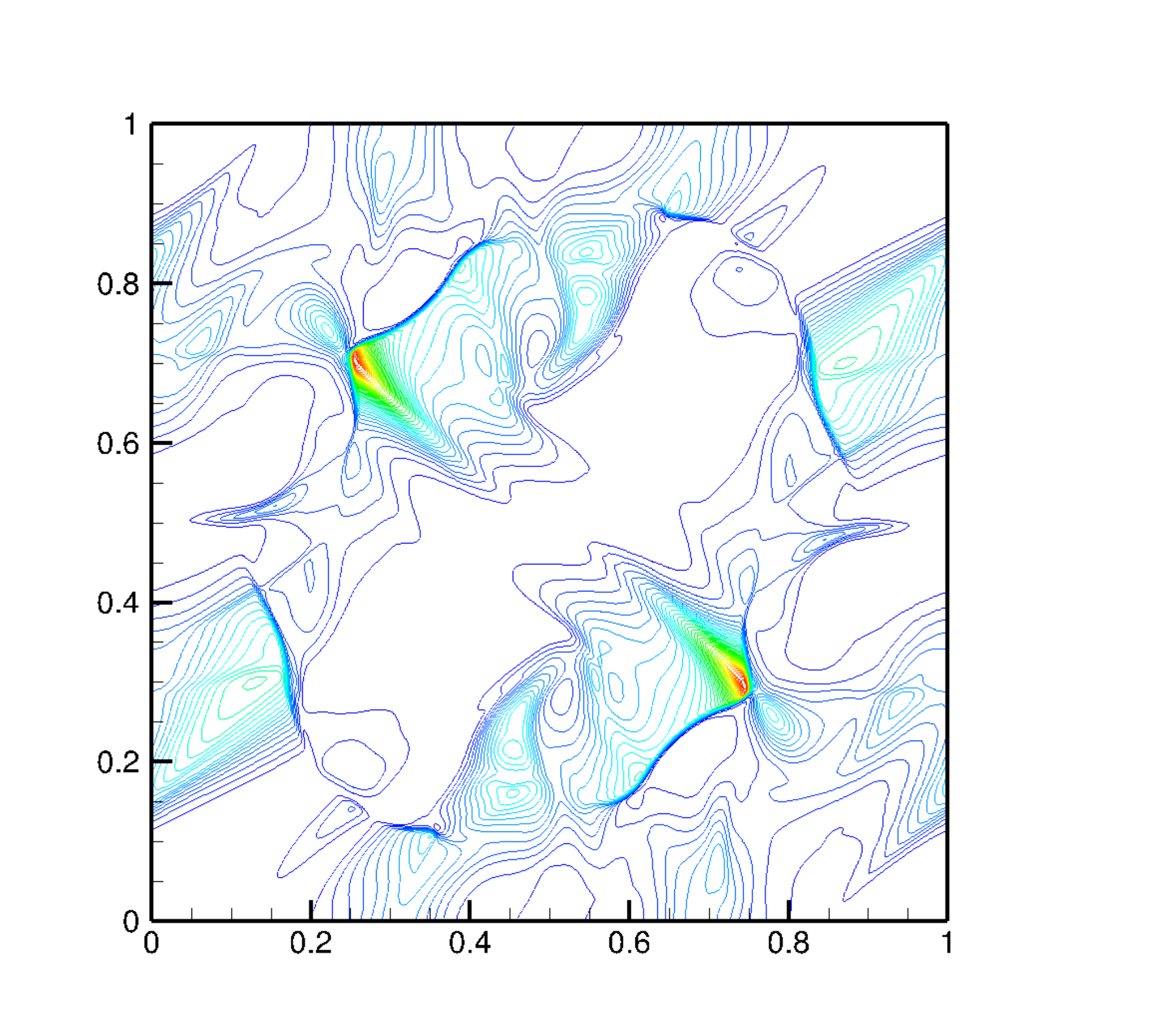}
  \end{subfigure}
  \caption{Example \ref{ex:OrszagTang}: The rest-mass density $\rho$
and the Lorentz factor $W$ (from left to right) at $t=1$ with $40$ equally spaced contour lines  obtained
    by using the entropy stable scheme with $N_x=N_y=400$.}
  \label{fig:OrszagTang}
\end{figure}

\begin{example}[Blast problem]\label{ex:2DBlast}\rm
 It describes a 2D RMHD blast problem.
 The initial setup is the same as that in \cite{Balsara2016,Del2003,Mignone2006}.
  The computation domain  $[-6,6]^2$ with outflow boundary conditions
  consists of three parts. The inner part is the explosion zone with
  a radius of $0.8$, and $\rho=0.01,~p=1$. The outer part is the ambient medium with the
  radius larger than $1$, and $\rho=10^{-4},~p=5\times 10^{-4}$. The intermediate
  part is a linear taper applied to the density and the pressure from the radius
  $0.8$ to $1$. The magnetic field is only initialized in the $x$-direction as
  $\Bx=0.1$ or 0.5. The adiabatic index $\Gamma=4/3$ is used in this test.

Figures \ref{fig:2DBlast} and \ref{fig:2DBlast_Bx05} plot the contours of the rest-mass density logarithm, the pressure logarithm, the Lorentz factor, and the magnitude of the magnetic field at $t=4$ with $40$
equally spaced contour lines obtained by using the entropy stable scheme with
the TVB limiter parameter   $M=0.01$.
We can see that the solutions are well gotten and
those for the case of $\Bx=0.1$ and $0.5$
 are in agreement with those in
\cite{Balsara2016} and \cite{wu2017b}, respectively.

The time evolution of the total entropy in Example \ref{ex:OrszagTang} and Example
\ref{ex:2DBlast} is shown in Figure \ref{fig:2DTotalEntropy},
 and   the  observed  monotonic decay   implies that the fully discrete scheme is entropy stable.

\end{example}

\begin{figure}[ht!]
  \begin{subfigure}[b]{0.5\textwidth}
    \centering
    \includegraphics[width=1.0\textwidth, trim=50 40 70 50, clip]{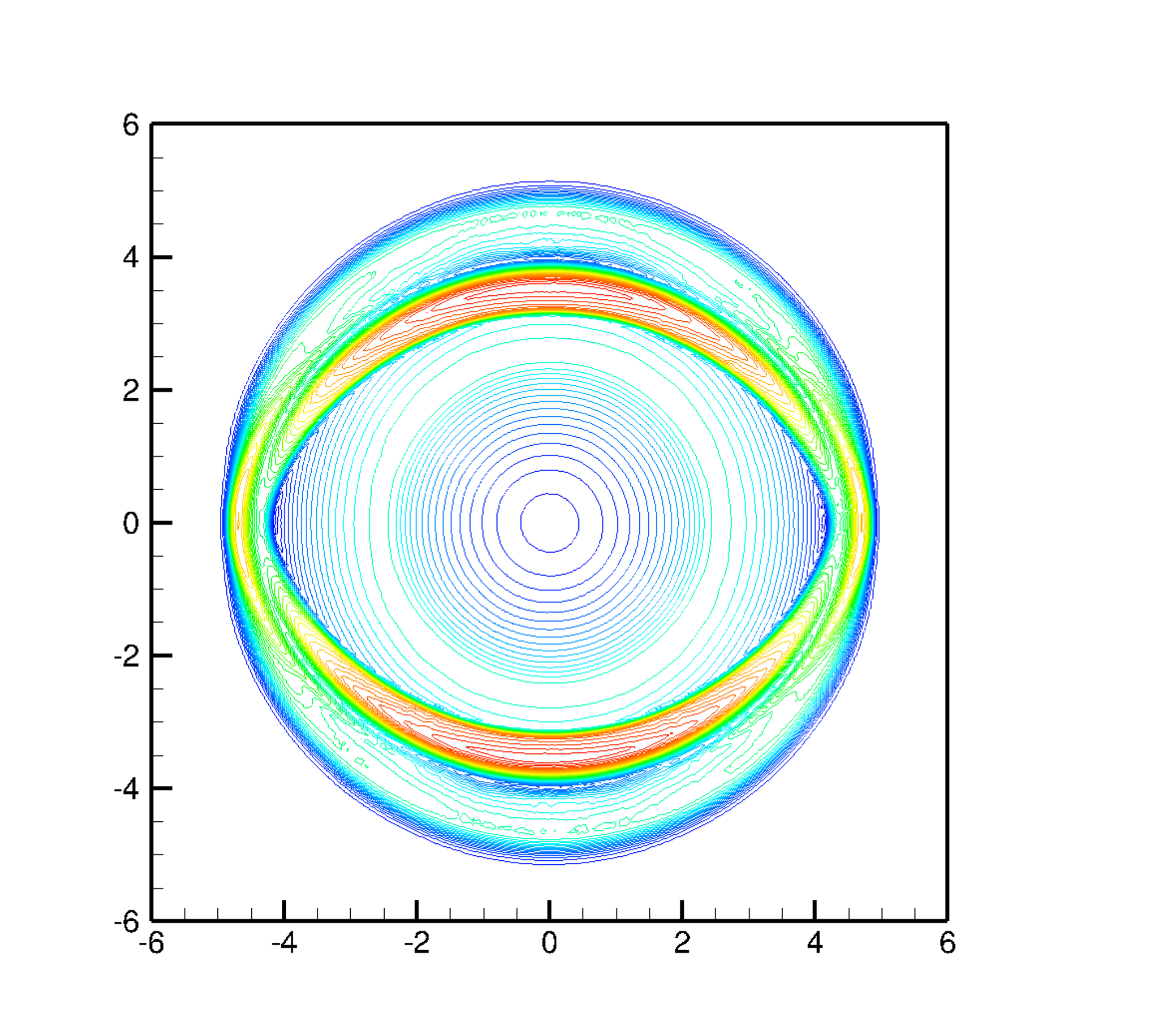}
    \caption{$\log_{10}\rho$}
  \end{subfigure}
  \begin{subfigure}[b]{0.5\textwidth}
    \centering
    \includegraphics[width=1.0\textwidth, trim=50 40 70 50, clip]{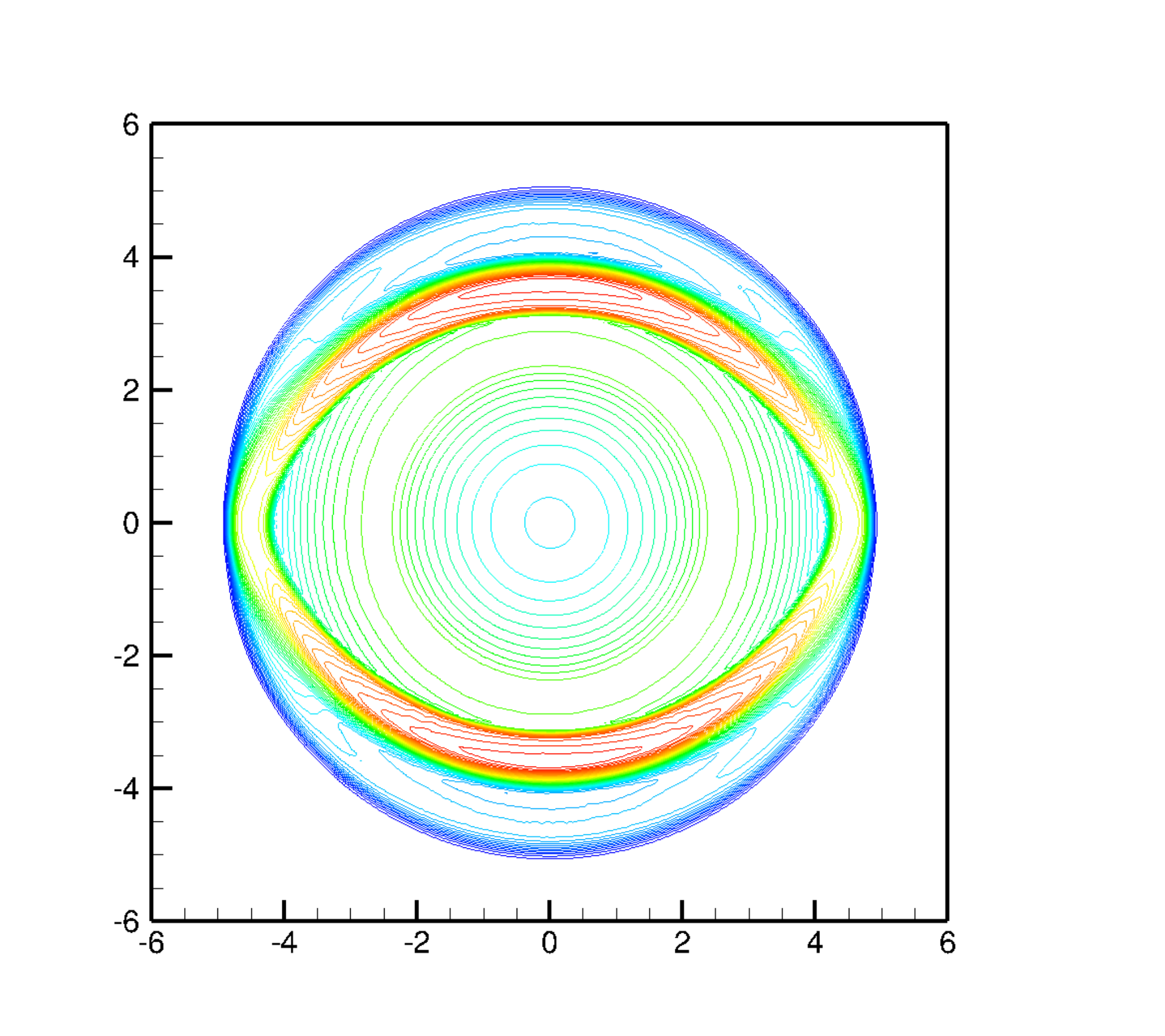}
    \caption{$\log_{10}p$}
  \end{subfigure}
  \begin{subfigure}[b]{0.5\textwidth}
    \centering
    \includegraphics[width=1.0\textwidth, trim=50 40 70 50, clip]{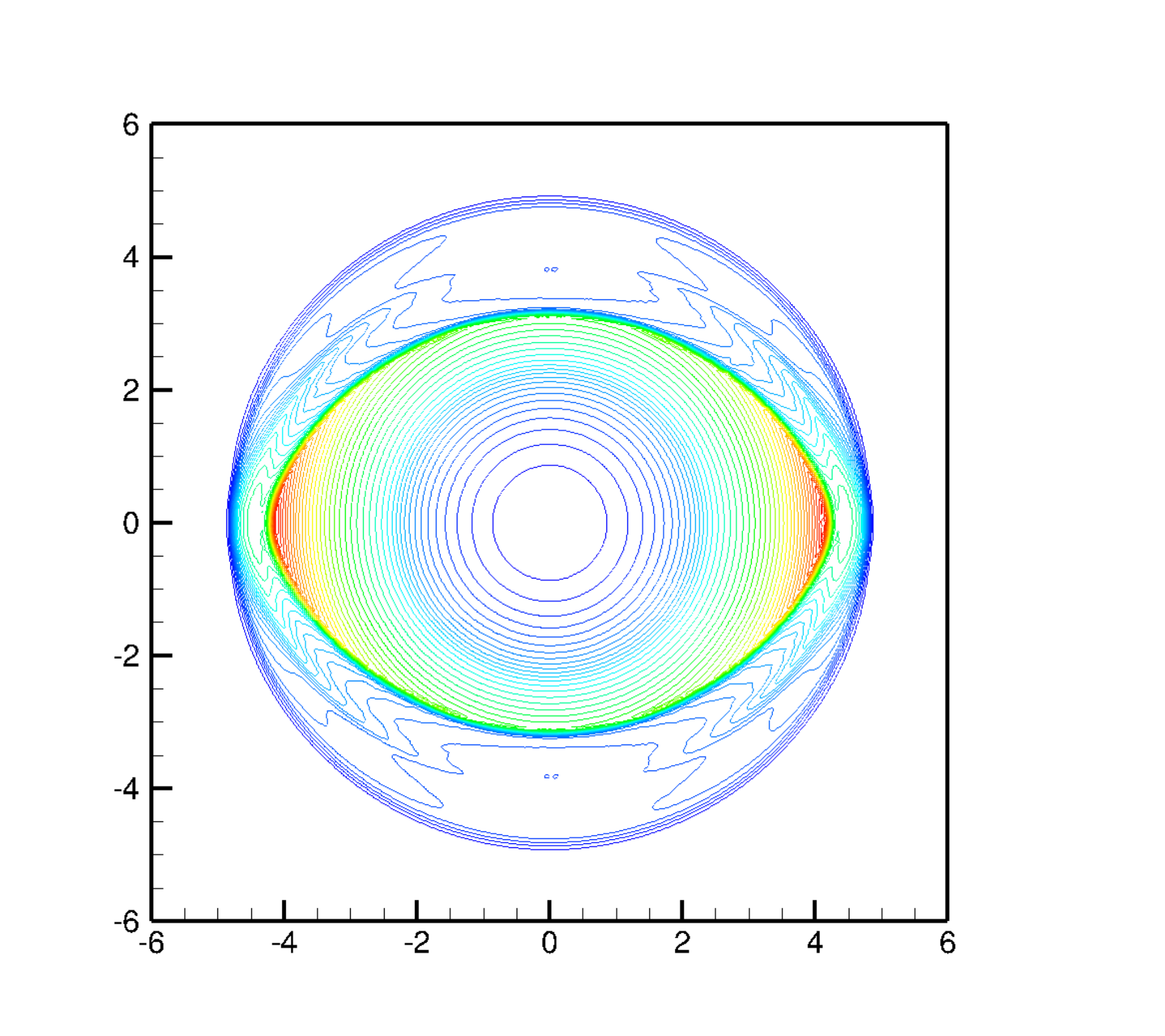}
    \caption{$W$}
  \end{subfigure}
  \begin{subfigure}[b]{0.5\textwidth}
    \centering
    \includegraphics[width=1.0\textwidth, trim=50 40 70 50, clip]{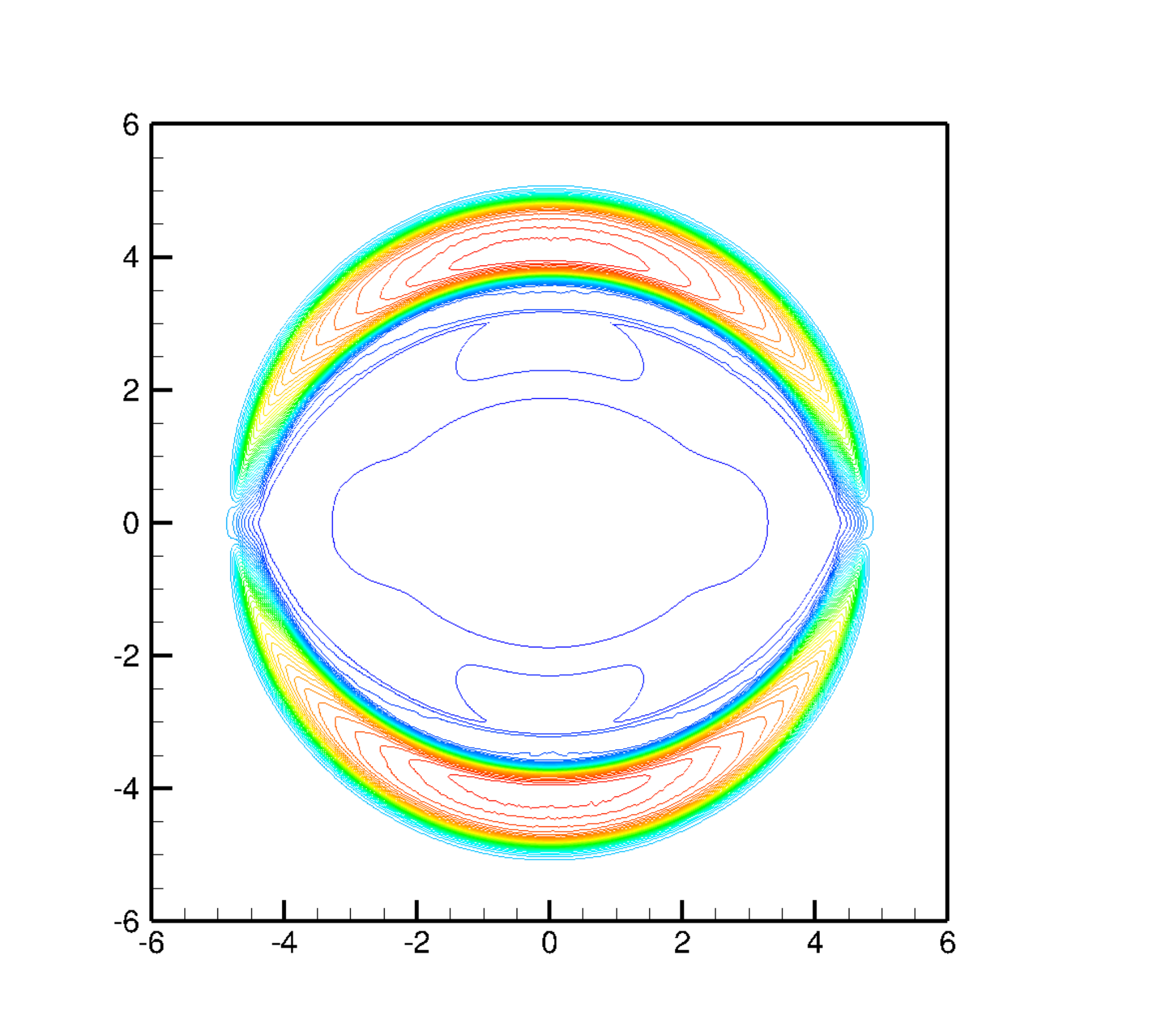}
    \caption{$|\bm{B}|$}
  \end{subfigure}
  \caption{Example \ref{ex:2DBlast} {with initial $\Bx=0.1$}: Numerical solutions at $t=4$
 with $40$ equally spaced contour lines obtained
 by using the entropy stable scheme with $N_x=N_y=400$.}
  \label{fig:2DBlast}
\end{figure}

\begin{figure}[ht!]
  \begin{subfigure}[b]{0.5\textwidth}
    \centering
    \includegraphics[width=1.0\textwidth, trim=50 40 70 50, clip]{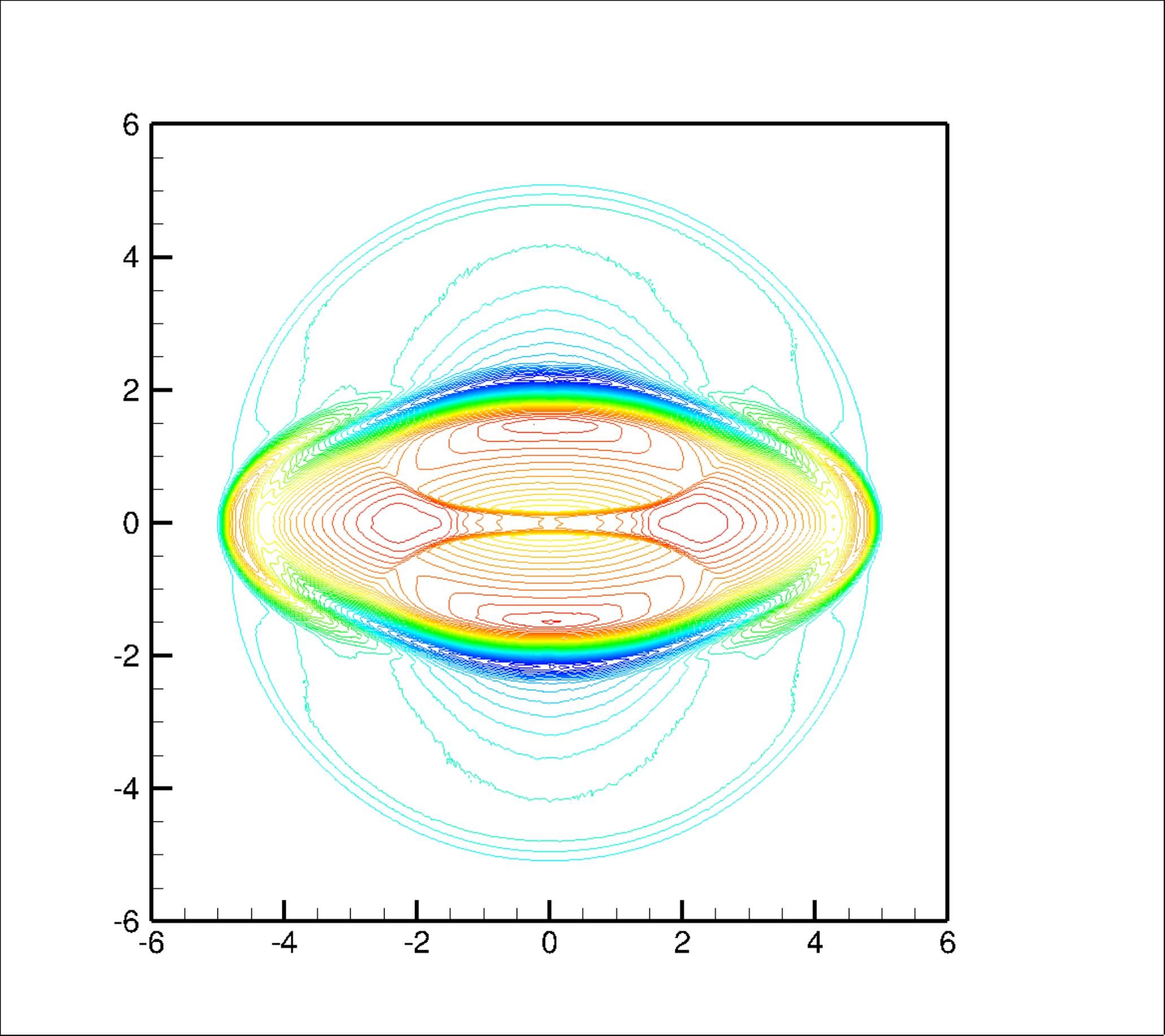}
    \caption{$\log_{10}\rho$}
  \end{subfigure}
  \begin{subfigure}[b]{0.5\textwidth}
    \centering
    \includegraphics[width=1.0\textwidth, trim=50 40 70 50, clip]{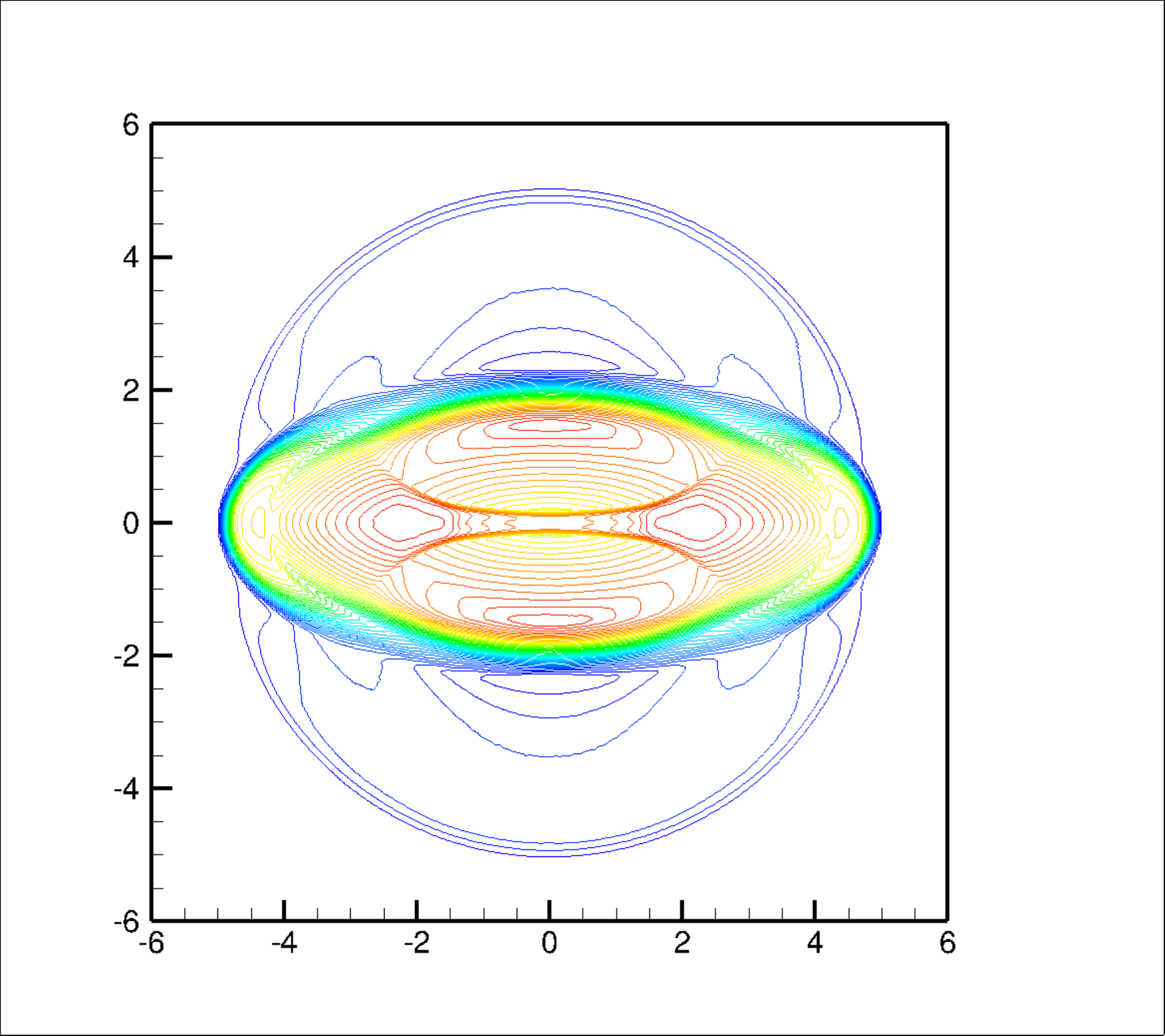}
    \caption{$\log_{10}p$}
  \end{subfigure}
  \begin{subfigure}[b]{0.5\textwidth}
    \centering
    \includegraphics[width=1.0\textwidth, trim=50 40 70 50, clip]{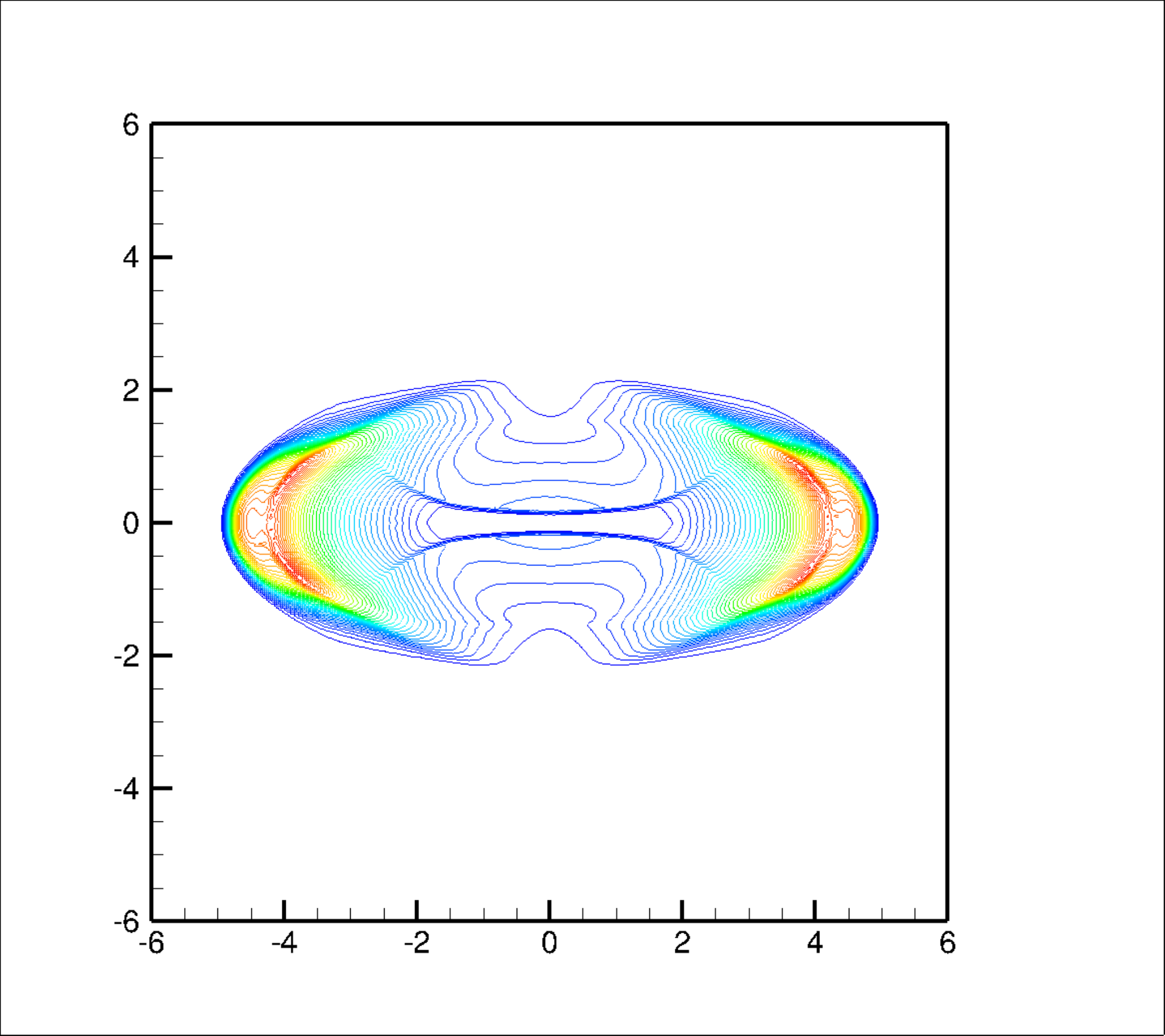}
    \caption{$W$}
  \end{subfigure}
  \begin{subfigure}[b]{0.5\textwidth}
    \centering
    \includegraphics[width=1.0\textwidth, trim=50 40 70 50, clip]{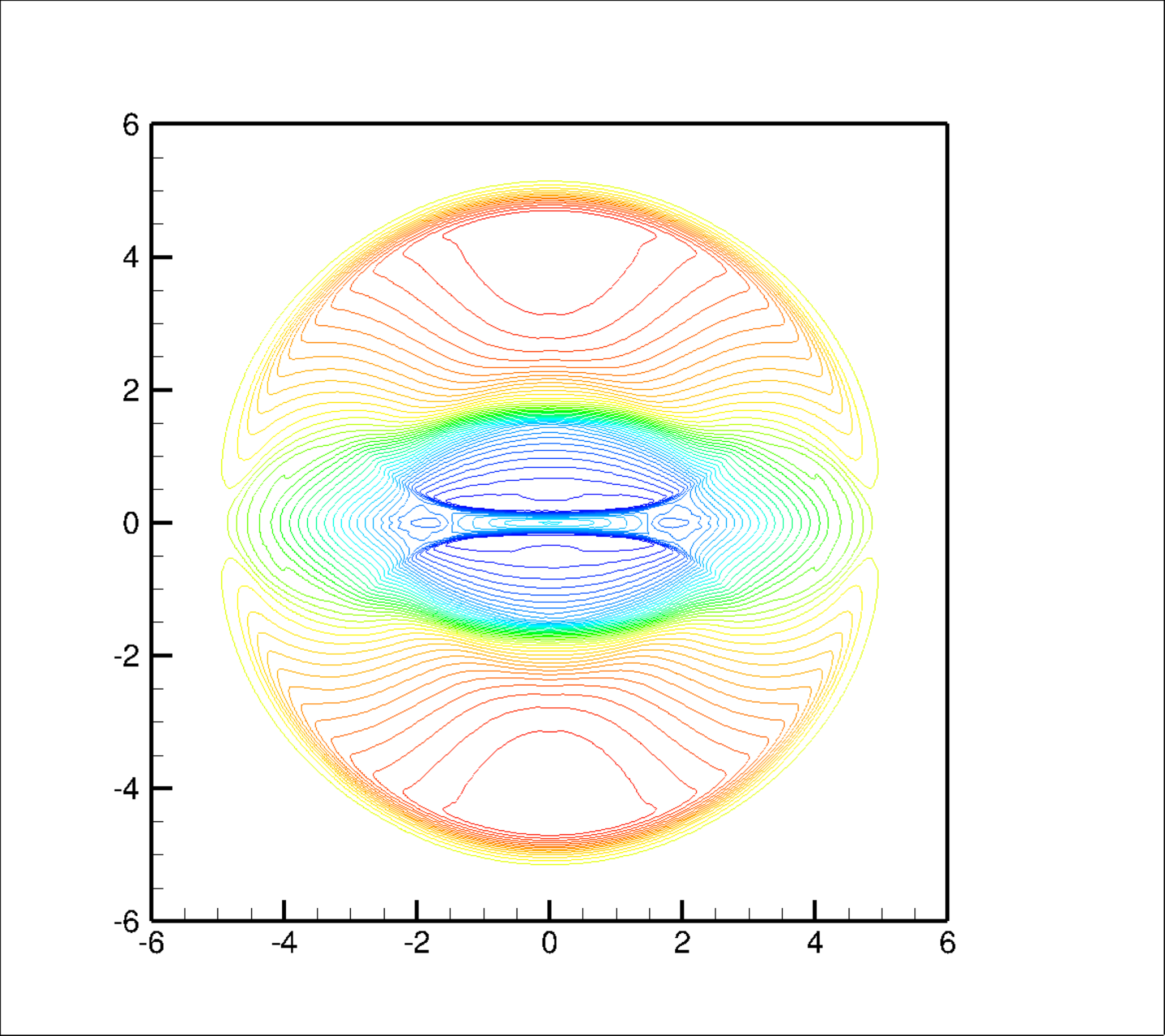}
    \caption{$|\bm{B}|$}
  \end{subfigure}
  \caption{Same as Figure \ref{fig:2DBlast} except for
  the initial $\Bx=0.5$.}
  \label{fig:2DBlast_Bx05}
\end{figure}

\begin{figure}[ht!]
  \begin{subfigure}[b]{0.5\textwidth}
    \centering
    \includegraphics[width=1.0\textwidth, trim=10 40 50 50, clip]{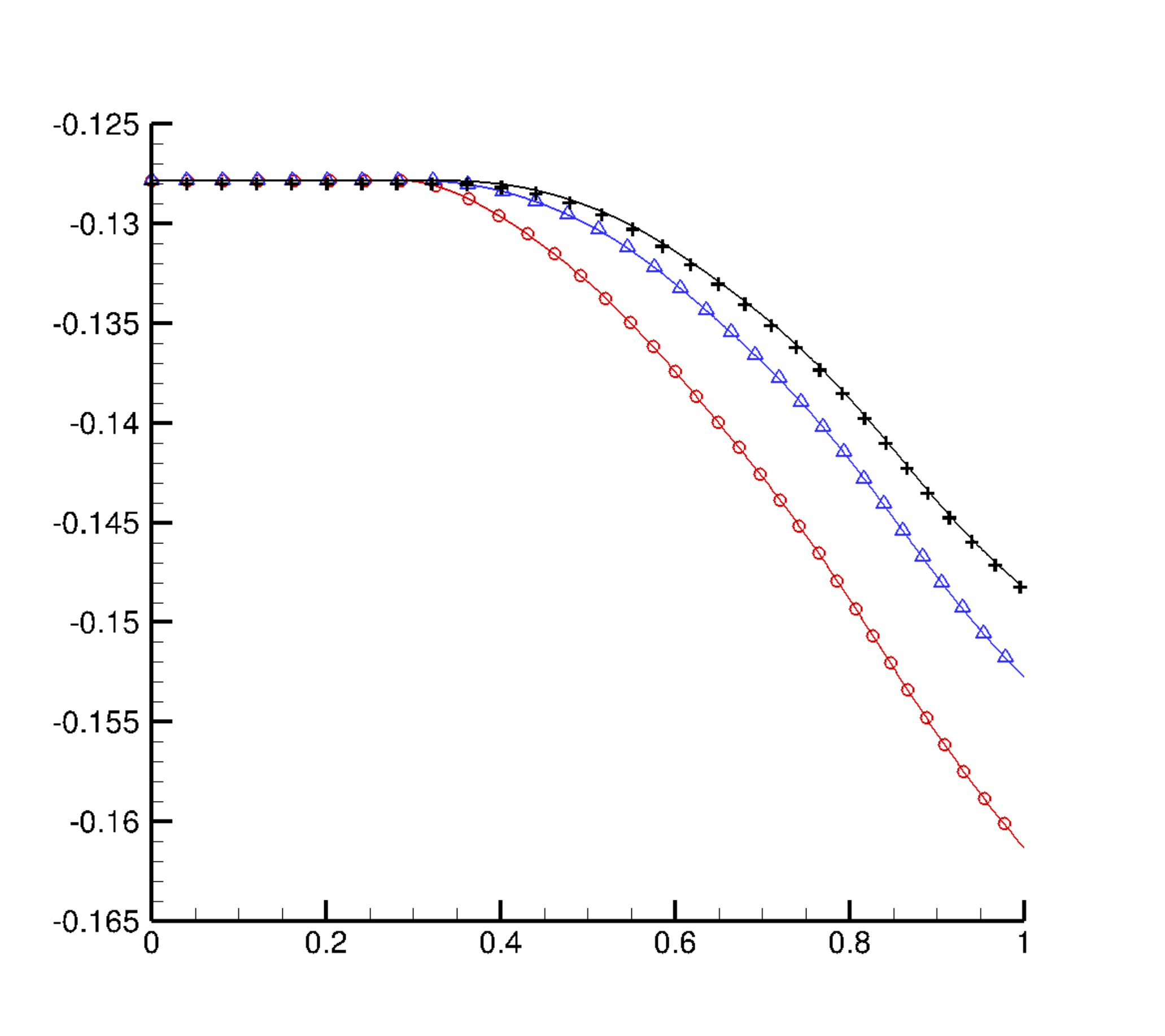}
    \caption{Example \ref{ex:OrszagTang}.}
  \end{subfigure}
  \begin{subfigure}[b]{0.5\textwidth}
    \centering
    \includegraphics[width=1.0\textwidth, trim=10 40 50 50, clip]{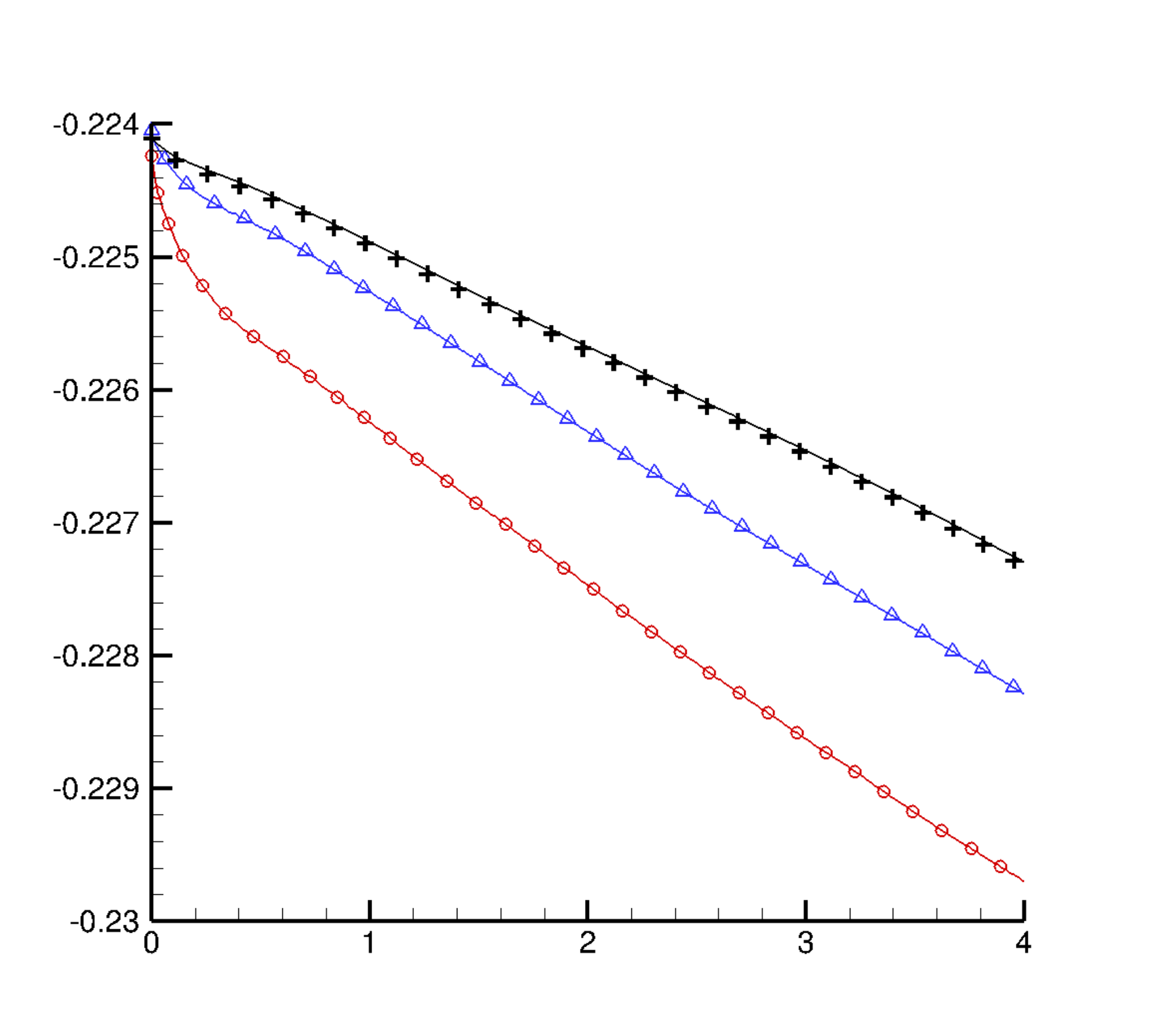}
    \caption{Example \ref{ex:2DBlast} for the initial $\Bx=0.1$.}
  \end{subfigure}
  \caption{The evolution of the total entropy obtained by using the entropy stable scheme.
    The symbols ``$\circ$'', ``$\Delta$'', and ``$+$'' with line denote the results obtained by using
  $100\times 100, 200\times 200, 400\times400$ cells, respectively.}
  \label{fig:2DTotalEntropy}
\end{figure}

\begin{example}[Shock-vortex interaction]\label{ex:SV}\rm
  This test is about the interaction between a shock wave and a vortex, which is
  constructed in \cite{Balsara2016}.
 Here we rotate the shock wave and   vortex clockwise by $\pi/4$, in order to eliminate the
  boundary effects in the computational domain.
  The present computational domain is taken as $[-9,9]^2$, {and} an  isentropic
  vortex initially centered at $(-3,0)$, similar to that in Example
  \ref{ex:vortex}, is {put},
  except for $w=-0.6\sqrt{2}$.
  A planar stationary shock wave placed at $x=2\sqrt{2}-3$ is
  initially far away from the vortex so that the pre-shock state is a constant state
  $$(\rho,\vx,\vy,\Bx,\By,p)=(6.73586072,~0.6\sqrt{2},~0,~0,~0,~24.02454458).$$
  Following \cite{Balsara2016}, the post-shock state is
  $$(\rho,\vx,\vy,\Bx,\By,p)=(10.47090373,~0.507707117\sqrt{2},~0,~0,~0,~50.44557978),$$
 when $x\geqslant 2\sqrt{2}-3$.
The problem is solved until $t=10$, but the results at $t=3.4$  will also be given
to show that the vortex is half-way through the shock wave.

Figure \ref{fig:SV} plots the contours of $\rho,W,|\bm{B}|$ at $t=3.4, 10$
with $40$ equally spaced contour lines  obtained by using our
entropy stable scheme with $N_x=N_y=600$ and the TVB limiter parameter $M=10$.
Those results show that the shock wave is still located at $x=2\sqrt{2}-3$ after
the interaction of the vortex and the shock wave, {and} our entropy stable scheme can capture
the complicated structures of the solutions. They are very similar to those
in \cite{Balsara2016}, even though our result in the center of the vortex is not as smooth as that in \cite{Balsara2016}.
\end{example}
\begin{figure}[ht!]
    \centering
  \begin{subfigure}[b]{0.32\textwidth}
    \centering
    \includegraphics[width=1.0\textwidth, trim=55 40 110 60, clip]{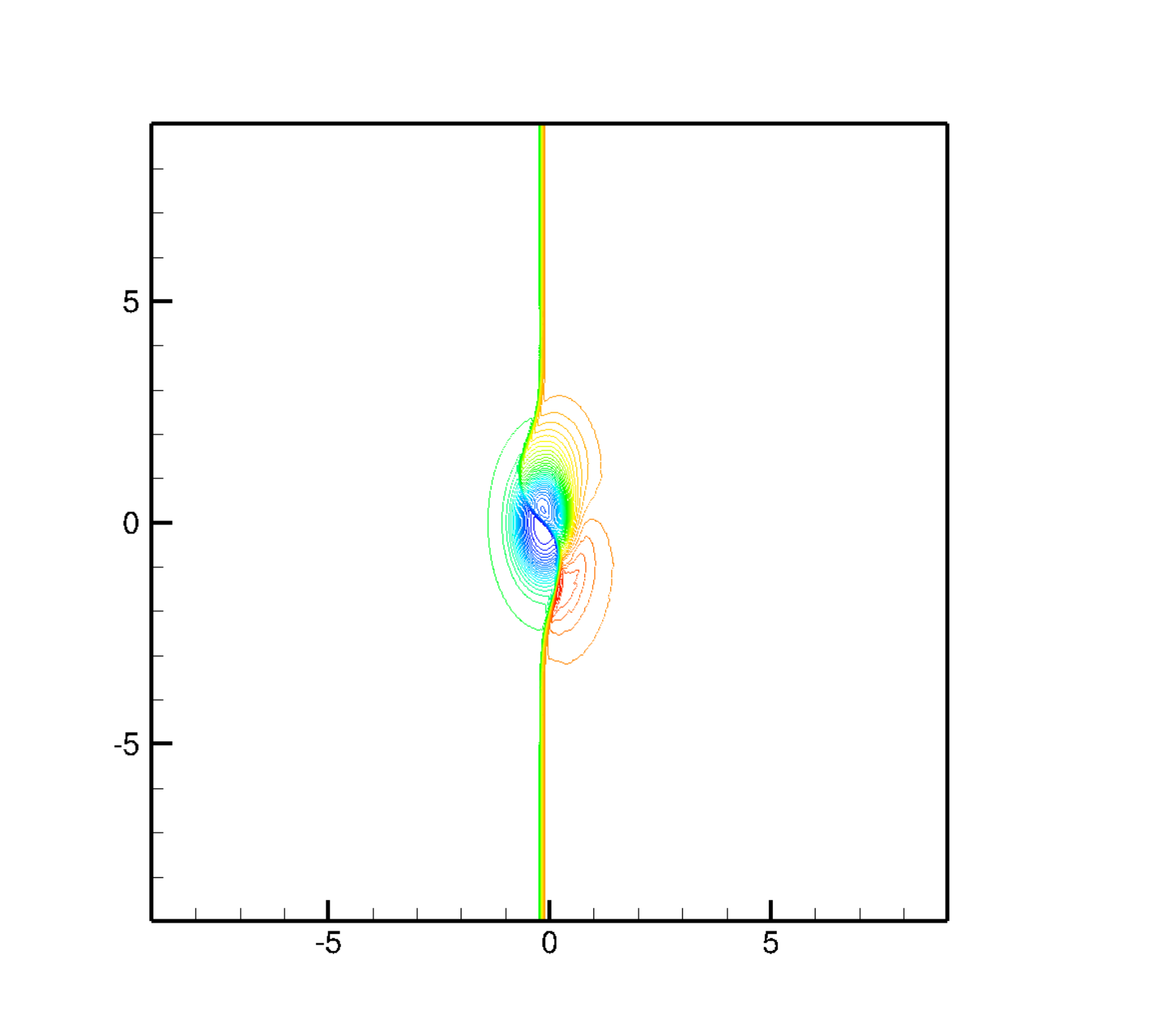}
    \caption{$\rho$}
  \end{subfigure}
  \begin{subfigure}[b]{0.32\textwidth}
    \centering
    \includegraphics[width=1.0\textwidth, trim=55 40 110 60, clip]{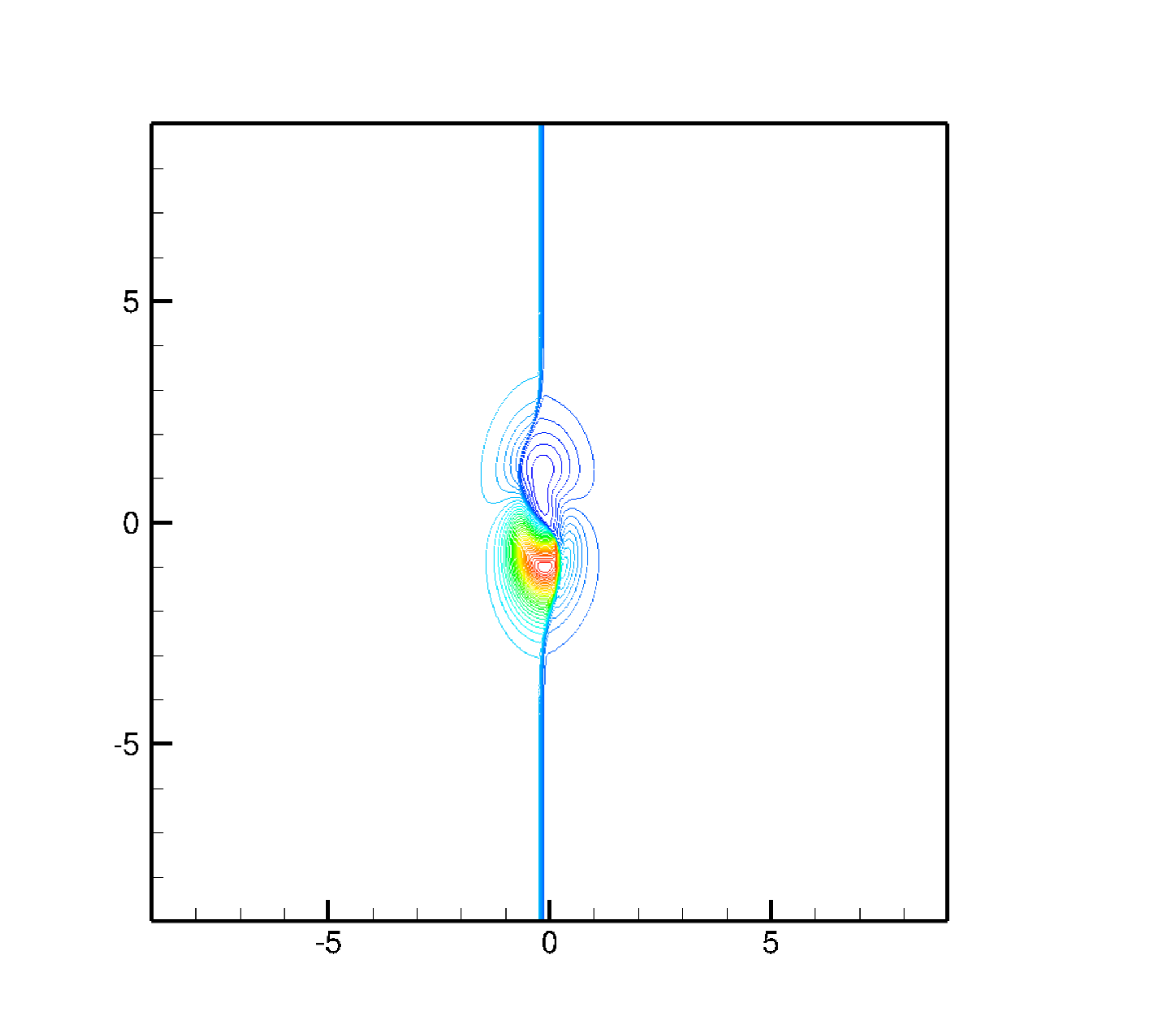}
    \caption{$W$}
  \end{subfigure}
  \begin{subfigure}[b]{0.32\textwidth}
    \centering
    \includegraphics[width=1.0\textwidth, trim=55 40 110 60, clip]{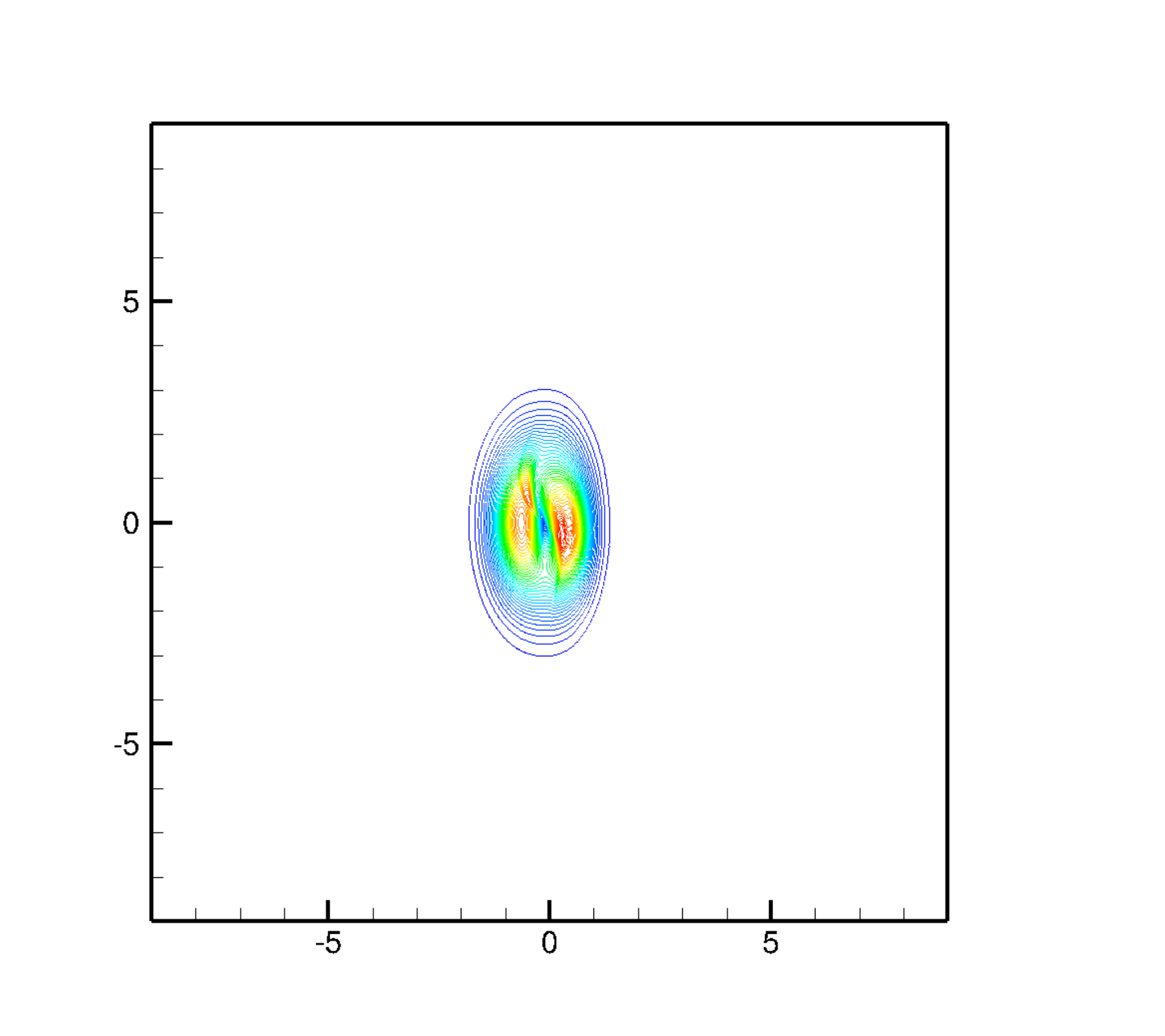}
    \caption{$|\bm{B}|$}
  \end{subfigure}
  \begin{subfigure}[b]{0.32\textwidth}
    \centering
    \includegraphics[width=1.0\textwidth, trim=55 40 110 60, clip]{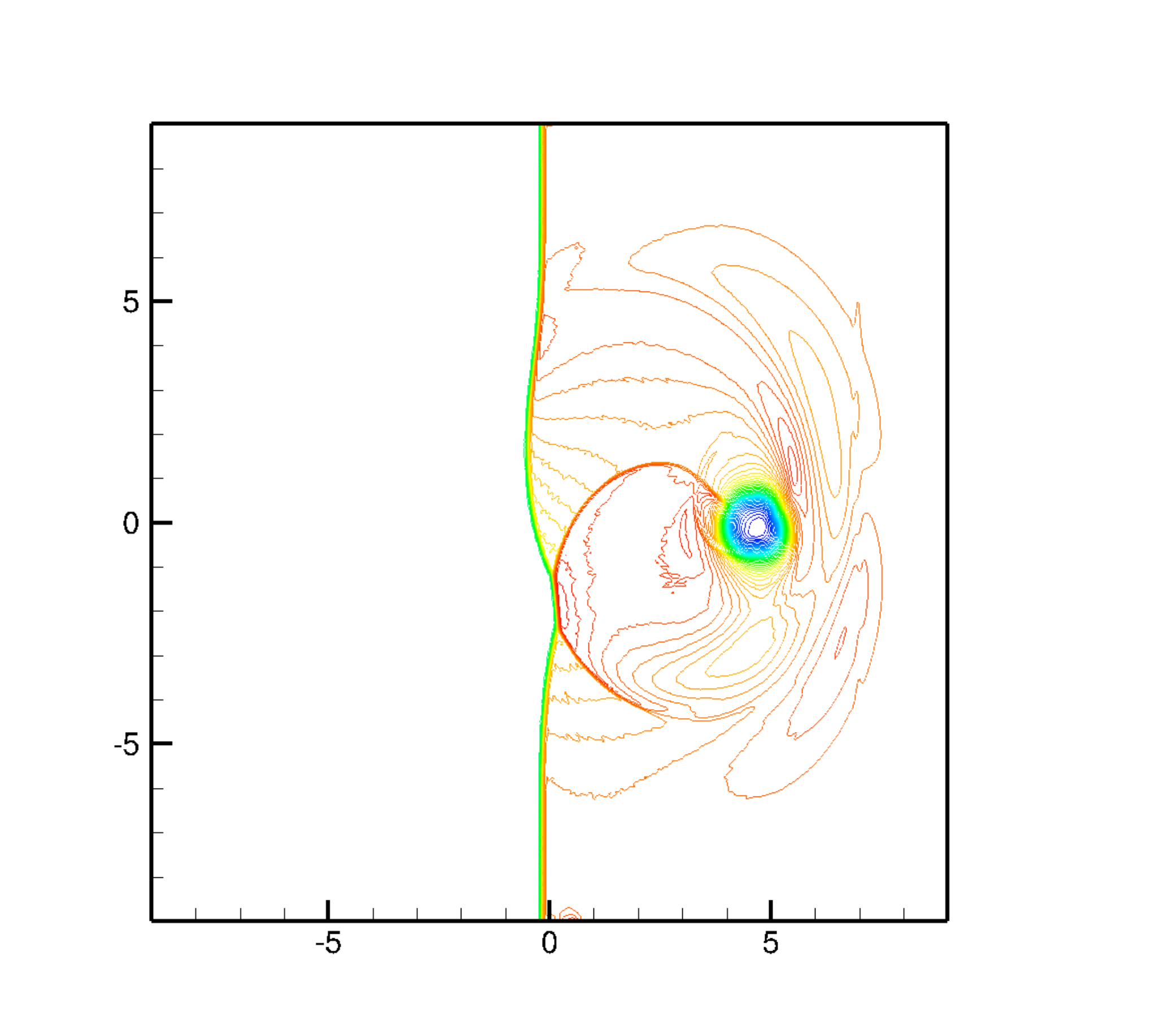}
    \caption{$\rho$}
  \end{subfigure}
  \begin{subfigure}[b]{0.32\textwidth}
    \centering
    \includegraphics[width=1.0\textwidth, trim=55 40 110 60, clip]{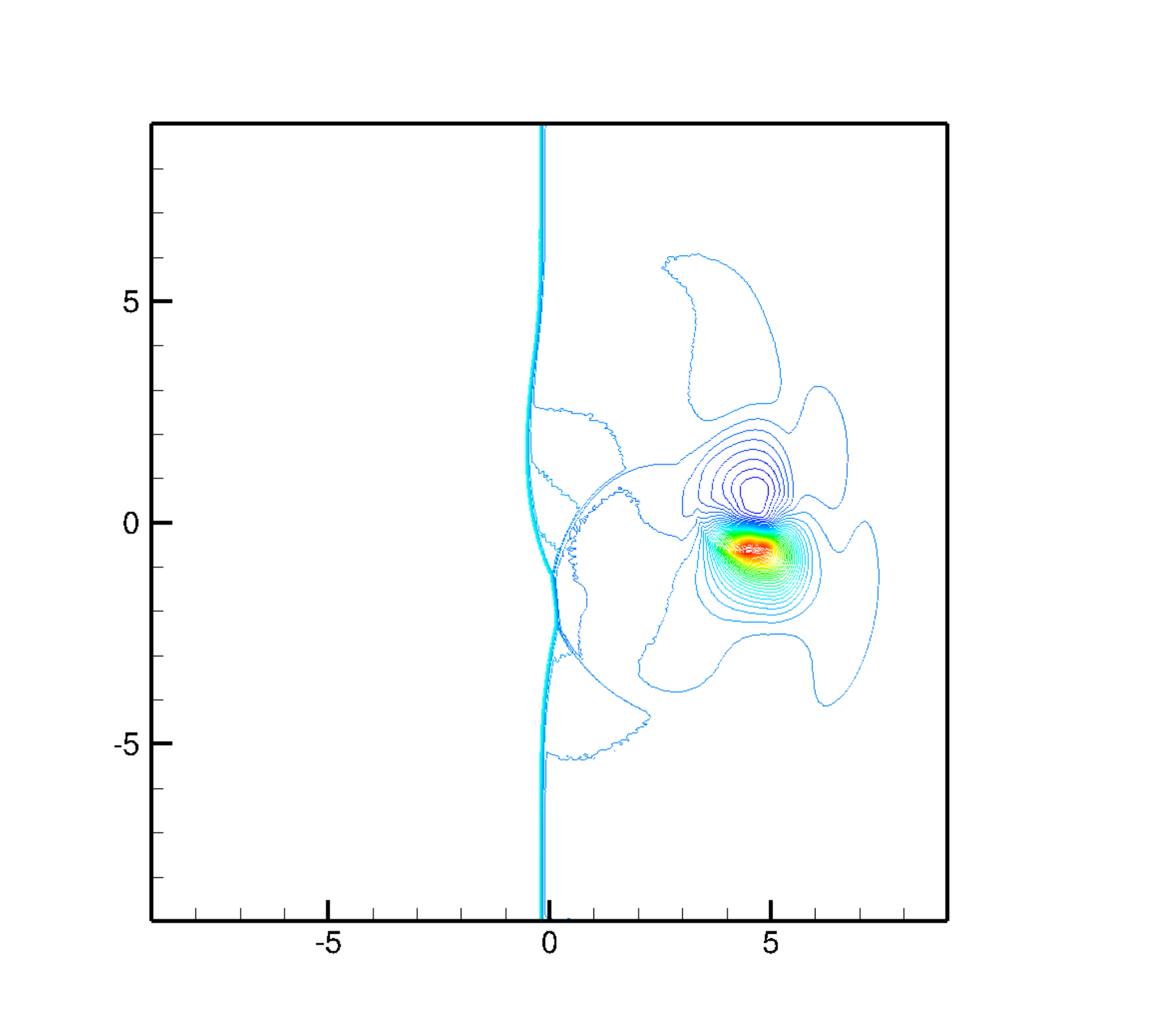}
    \caption{$W$}
  \end{subfigure}
  \begin{subfigure}[b]{0.32\textwidth}
    \centering
    \includegraphics[width=1.0\textwidth, trim=55 40 110 60, clip]{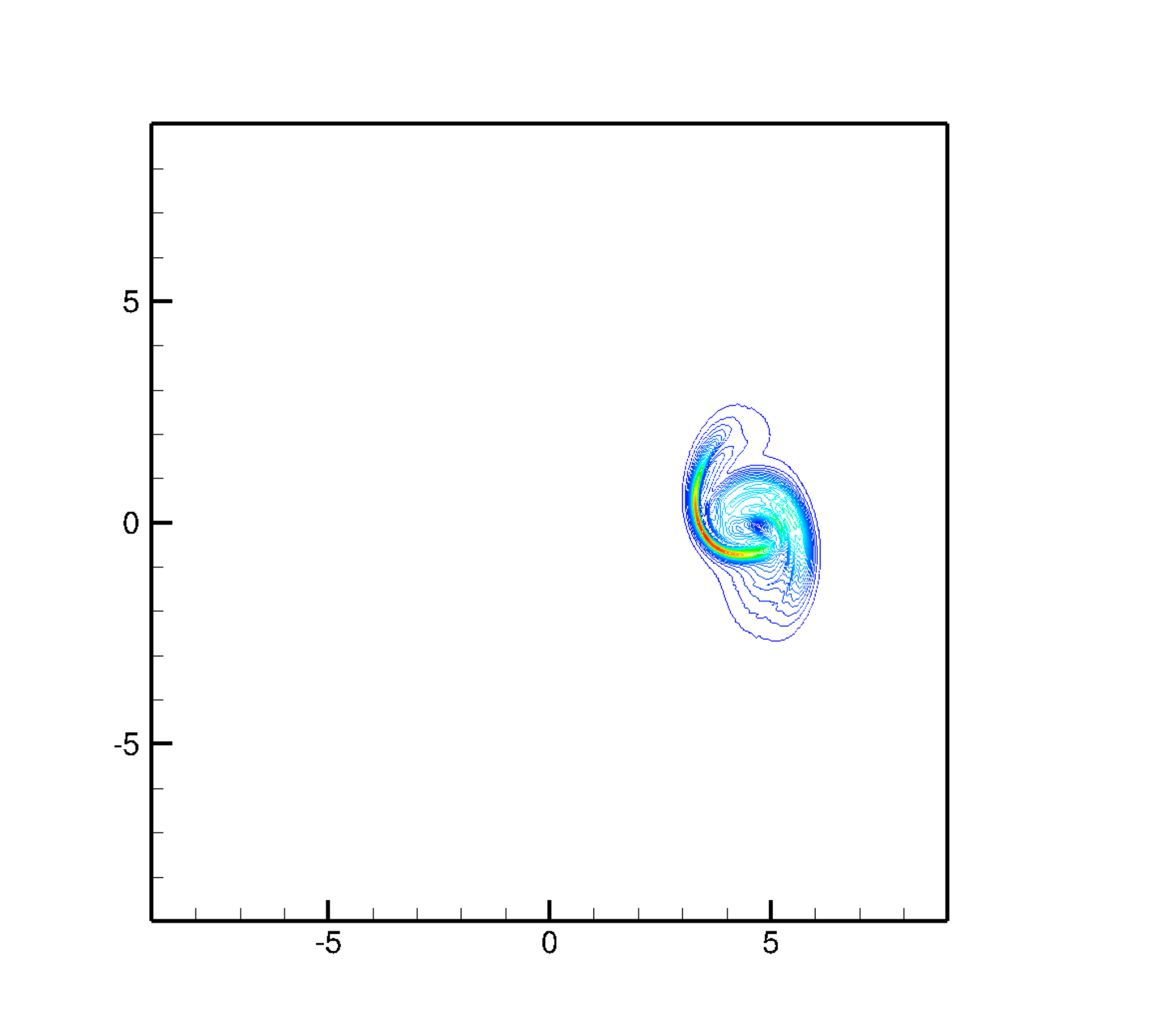}
    \caption{$|\bm{B}|$}
  \end{subfigure}
  \caption{Example \ref{ex:SV}: Numerical solutions at $t=3.4$ (first row)
  and $10$ (second row)
with $40$ equally spaced contour lines obtained
    by using the entropy stable scheme with $N_x=N_y=600$.}
  \label{fig:SV}
\end{figure}

\subsection{Comparison between two entropy conservative fluxes}
This section presents a numerical comparison of our two-point entropy conservative flux
and that in \cite{Wu2019}.
As stated in Section \ref{section:OneD}, the parallel magnetic component of our entropy
conservative flux is zero so that it is ``fully consistent'' with the physical flux.
However, the parallel magnetic component of the entropy conservation flux in \cite{Wu2019}
is not zero, thus one can expect that the entropy stable DG scheme with
our entropy conservative flux will behave better.

In the following, we take  three 1D Riemann problems in Examples \ref{ex:RP1}-\ref{ex:RP3} and a rotated
shock tube problem as examples.

Figure \ref{fig:Bx_cmp} shows the $\Bx$ components for three 1D Riemann problems
obtained by the entropy stable DG scheme using our two-point entropy
conservative flux \eqref{eq:ecFlux1D1}-\eqref{eq:ecFlux1D3} and the flux (3.6)
in \cite{Wu2019}, respectively.
The other components are nearly the same so that they are omitted  here
and the only difference between two schemes is just in the two-point entropy
conservative flux.
The solid lines and the symbol (``$\square$'')
denote the numerical solutions obtained
by using the flux \eqref{eq:ecFlux1D1}-\eqref{eq:ecFlux1D3} and the flux
(3.6) in \cite{Wu2019}, respectively.
We can see that, for those  Riemann problems,
the results with the present two-point entropy conservative flux can
   maintain $\Bx$ invariant exactly, while others do not have this property,
   and maintaining a zero parallel magnetic component
is useful to reduce the error in the parallel magnetic component.

\begin{figure}[ht!]
  \begin{subfigure}[b]{0.32\textwidth}
    \centering
    \includegraphics[width=1.0\textwidth, trim=10 0 50 30, clip]{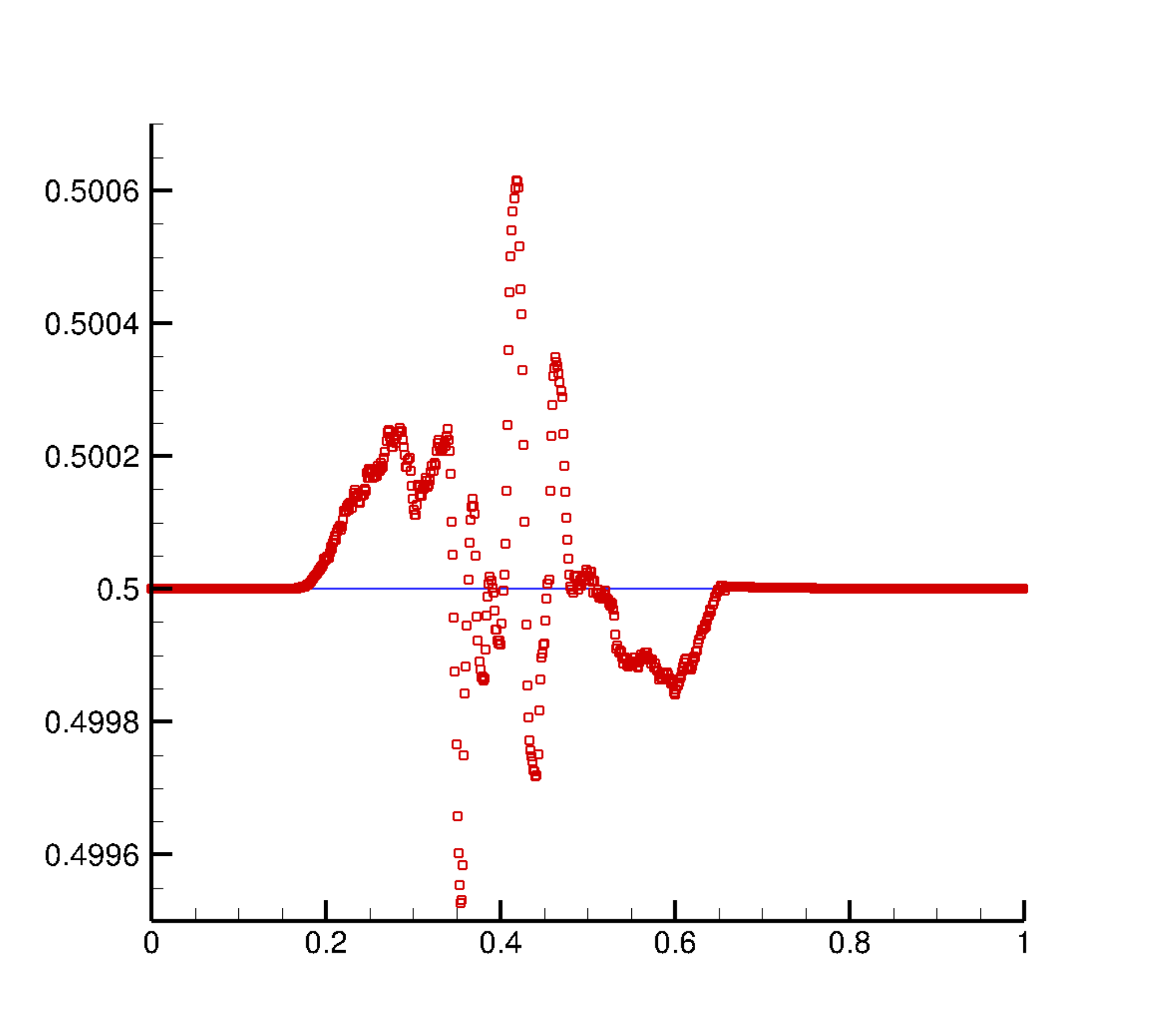}
    \caption{Example \ref{ex:RP1}}
  \end{subfigure}
  \begin{subfigure}[b]{0.32\textwidth}
    \centering
    \includegraphics[width=1.0\textwidth, trim=10 0 50 30, clip]{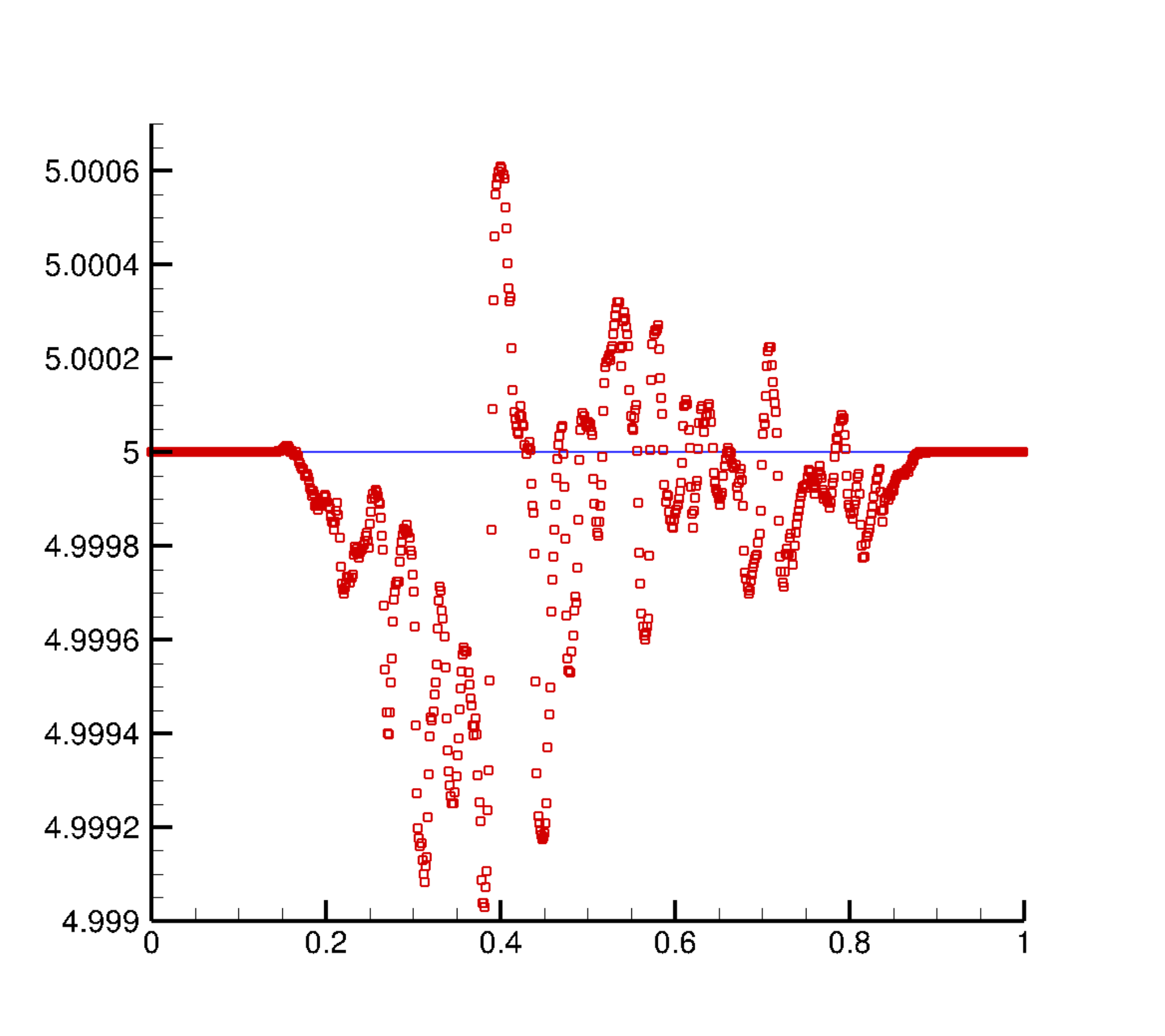}
    \caption{Example \ref{ex:RP2}}
  \end{subfigure}
  \begin{subfigure}[b]{0.32\textwidth}
    \centering
    \includegraphics[width=1.0\textwidth, trim=10 0 50 30, clip]{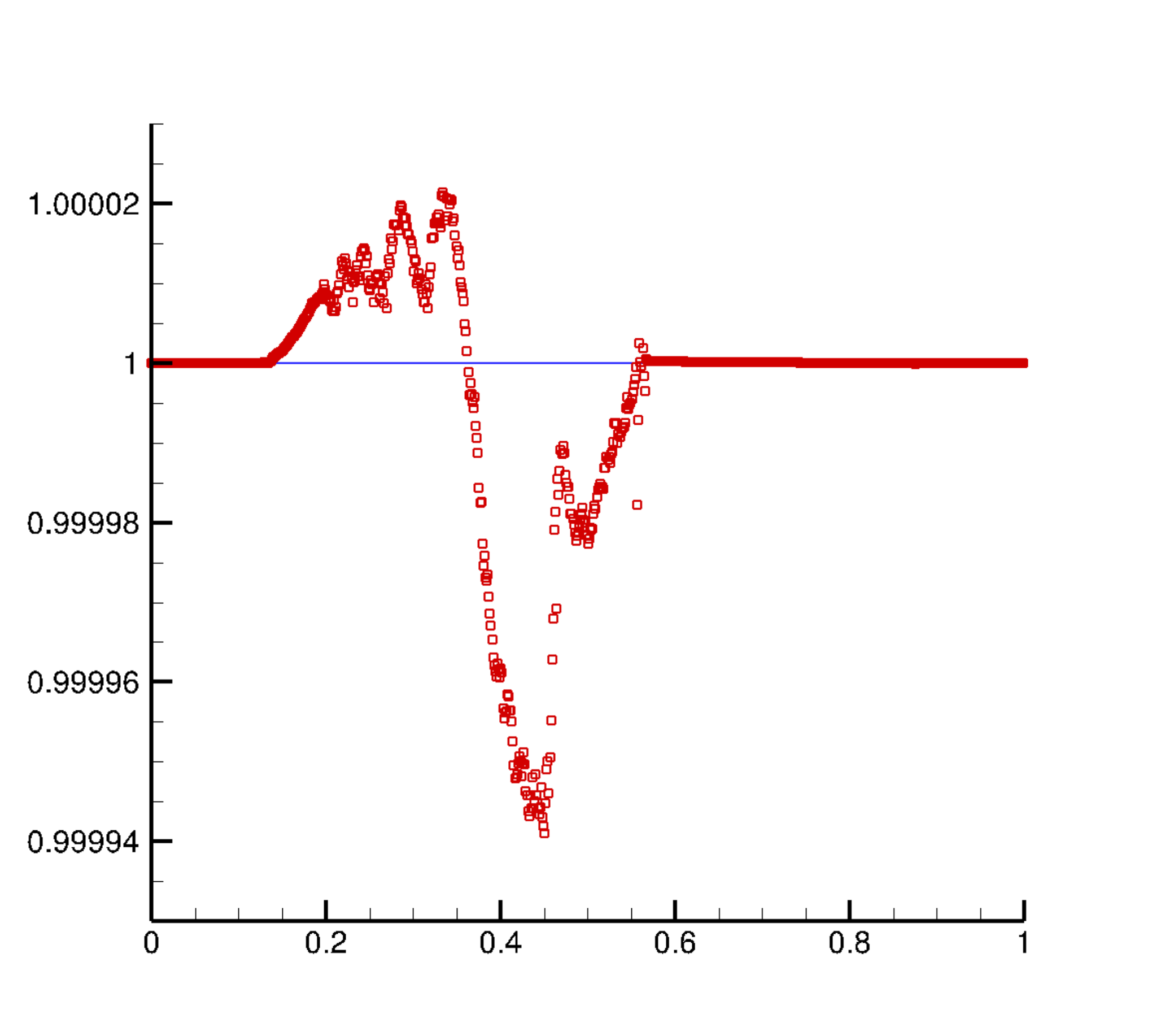}
    \caption{Example \ref{ex:RP3}}
  \end{subfigure}
  \caption{Examples \ref{ex:RP1}-\ref{ex:RP3}: The solid line and the symbol (``$\square$'') denote the numerical solutions obtained
    by using the flux  \eqref{eq:ecFlux1D1}-\eqref{eq:ecFlux1D3}  and the flux (3.6) in \cite{Wu2019},
  respectively.}
  \label{fig:Bx_cmp}
\end{figure}

\begin{example}[Rotated shock tube]\label{ex:RST} \rm
  Following \cite{Liu2018,Toth2000},  this rotated shock tube problem is designed to compare
  the results of the entropy stable DG schemes with different two-point entropy
  conservative fluxes. The initial left and right states are
  $(\rho,v_\parallel,v_\perp,B_\parallel,B_\perp,p)=(1,~0.5,~0,~0.5,~0.5,~1)$
  and  $(1,-0.5,~0,~0.5,~0.5,~0.1)$, respectively.
  A $800\times 2$  Cartesian mesh with $\Delta x=\Delta y=1/800$ is used.
  The top and bottom boundaries are translational
  symmetry in the $(-1,1)$ direction, while
  the left and right boundary conditions  are specified by the initial conditions,
  in view of the fact that  the waves do not reach those boundaries  at the output time $t=0.4$.

Figure \ref{fig:RST} plots the numerical solutions obtained by using the entropy stable DG schemes, where the symbols ``$\circ$'' and
``$+$'' denote the numerical solutions obtained by using the flux
\eqref{eq:ecFlux1D1}-\eqref{eq:ecFlux1D3}  and the flux (3.6) in \cite{Wu2019},
respectively, and the solid line  denotes the reference solution  obtained by using a 1D
first-order finite volume scheme on a mesh of $20000$ cells.
Our numerical solutions are in good agreement with the
reference solutions, except that the parallel component of the magnetic field
$B_\parallel$ is not a constant due to the non-conservative source term, which
can also be seen in \cite{Liu2018,Toth2000}.
In this example, the schemes with two entropy conservative fluxes give similar
results and the large error is shown in $B_\parallel$  when $\divB$ is not zero.
From this example, we can see that the error in $B_\parallel$ mainly results from the
error in $\divB$, which dominates the error in the
non-zero parallel component of the two-point entropy conservative flux, thus we
almost cannot distinguish the results in $B_\parallel$ in Figure \ref{fig:RST}.
In summary, the newly resulting two-point entropy conservative flux may serve as a
better base of the entropy conservative or   stable schemes for the RMHD
equations  since it gives better or at least comparable results.
\end{example}
\begin{figure}[ht!]
  \begin{subfigure}[b]{0.32\textwidth}
    \centering
    \includegraphics[width=1.0\textwidth, trim=40 0 50 30, clip]{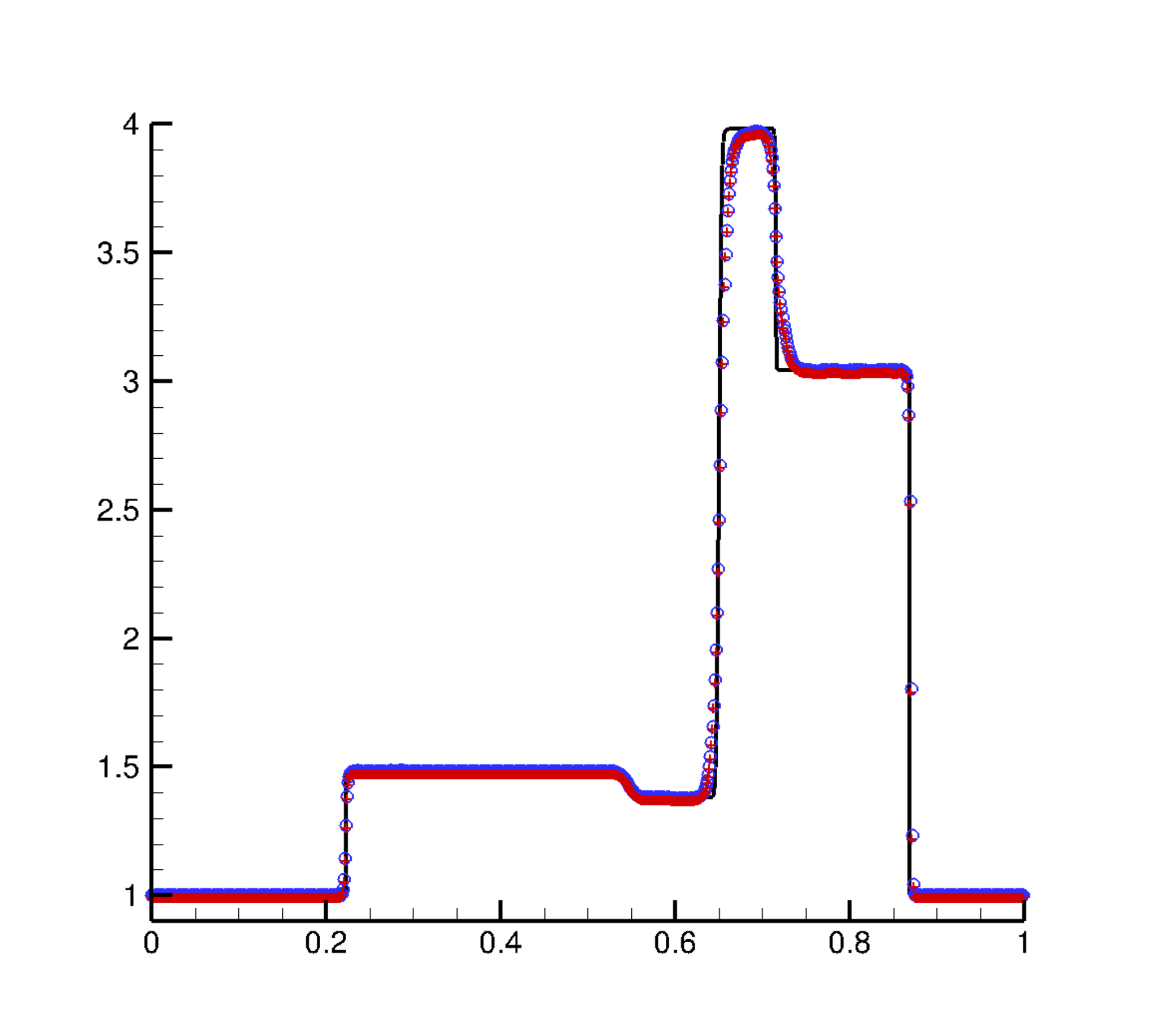}
    \caption{$\rho$}
  \end{subfigure}
  \begin{subfigure}[b]{0.32\textwidth}
    \centering
    \includegraphics[width=1.0\textwidth, trim=40 0 50 30, clip]{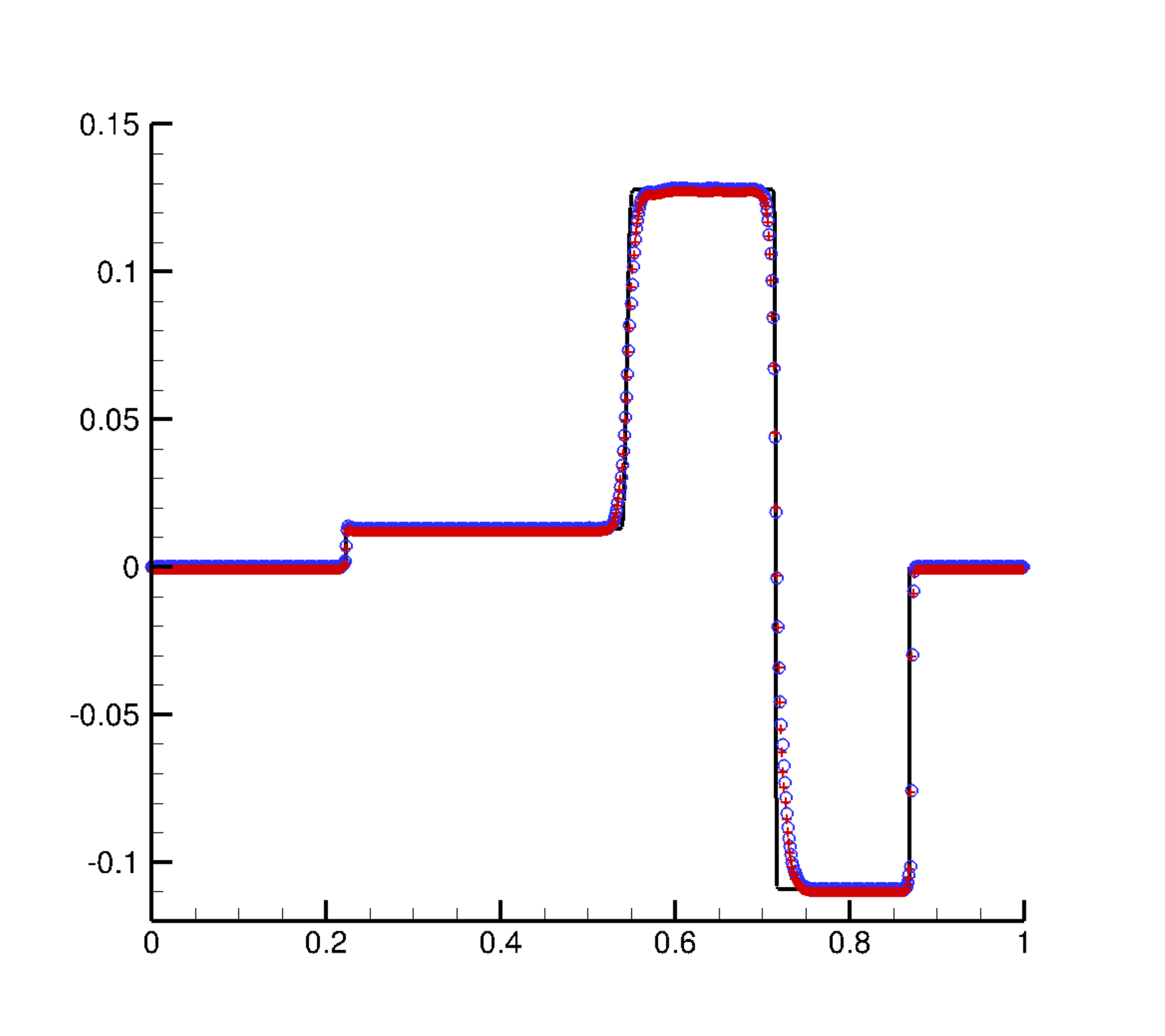}
    \caption{$v_\perp$}
  \end{subfigure}
    \begin{subfigure}[b]{0.32\textwidth}
    \centering
    \includegraphics[width=1.0\textwidth, trim=40 0 50 30, clip]{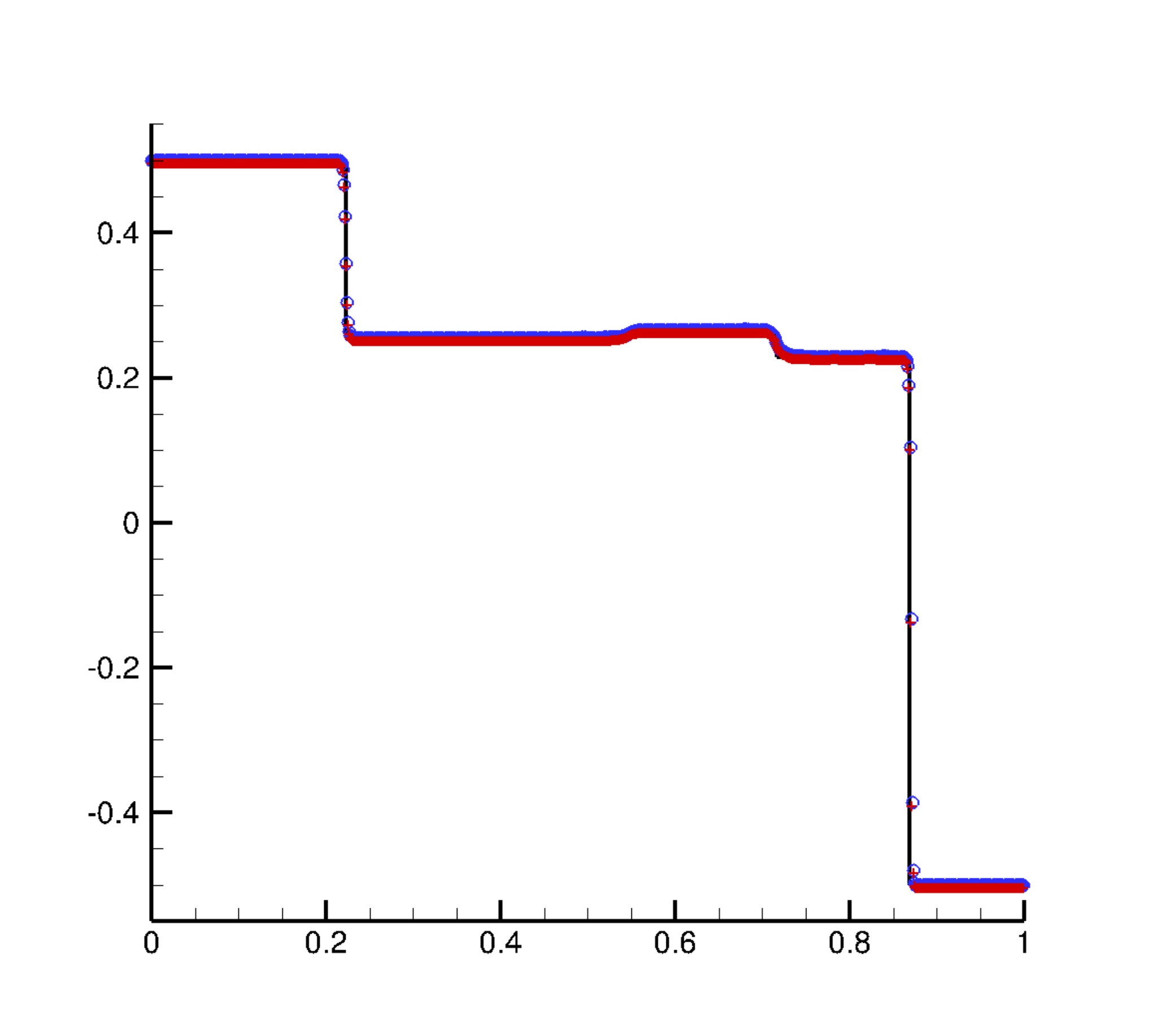}
    \caption{$v_\parallel$}
  \end{subfigure}
    \begin{subfigure}[b]{0.32\textwidth}
    \centering
    \includegraphics[width=1.0\textwidth, trim=40 0 50 30, clip]{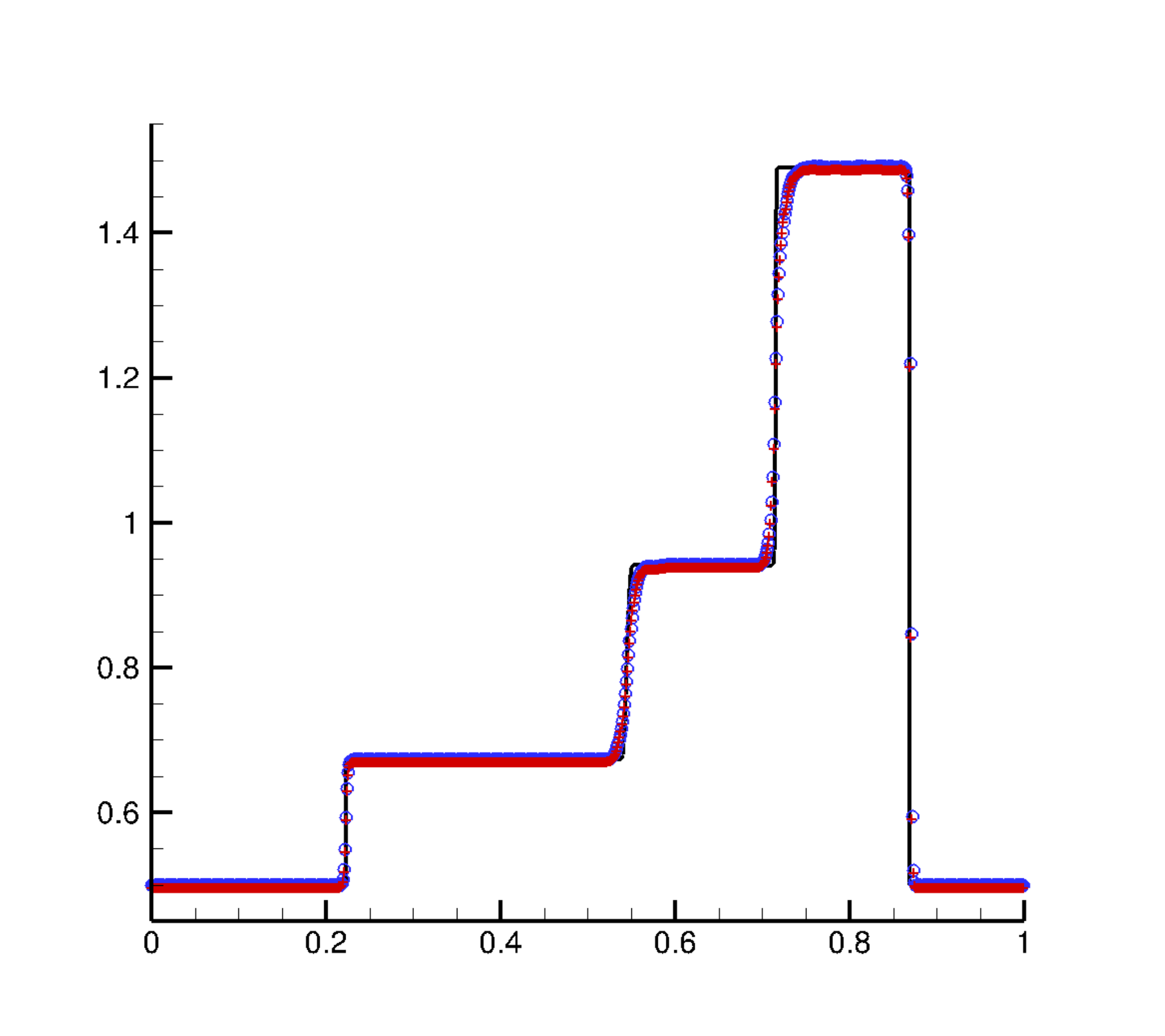}
    \caption{$B_\perp$}
  \end{subfigure}
  \begin{subfigure}[b]{0.32\textwidth}
    \centering
    \includegraphics[width=1.0\textwidth, trim=40 0 50 30, clip]{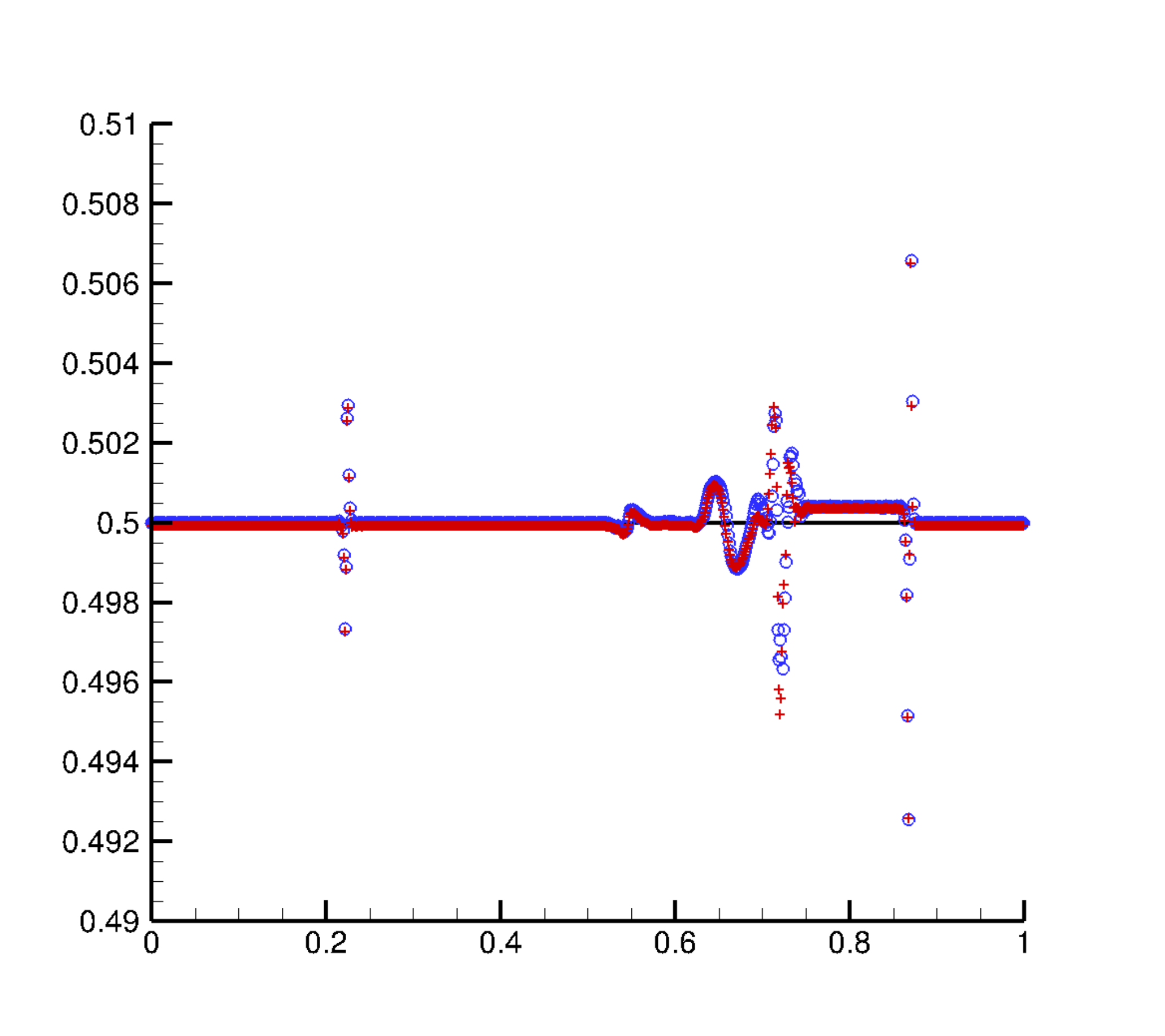}
    \caption{$B_\parallel$}
  \end{subfigure}
  \begin{subfigure}[b]{0.32\textwidth}
    \centering
    \includegraphics[width=1.0\textwidth, trim=40 0 50 30, clip]{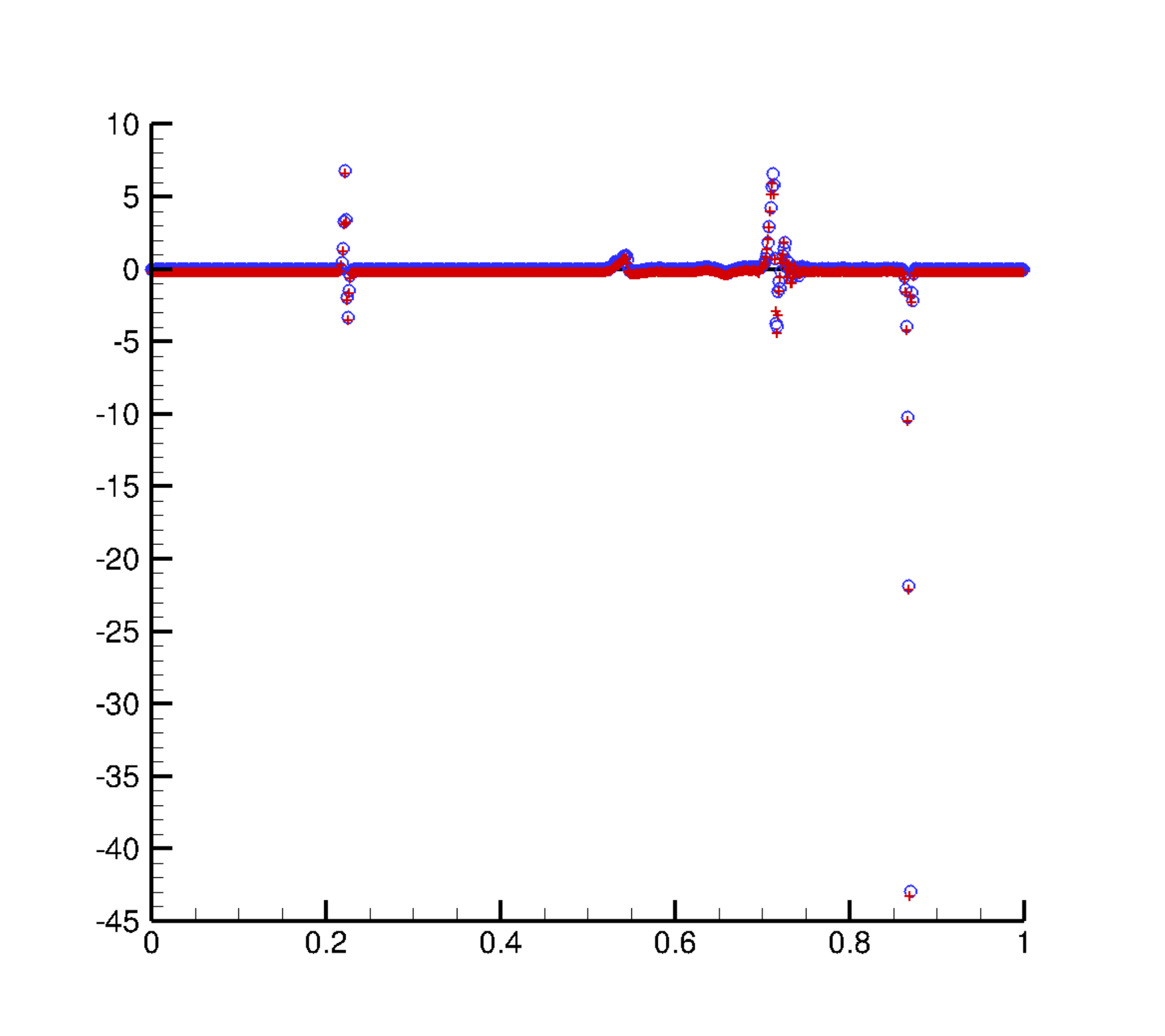}
    \caption{$\divB$}
  \end{subfigure}
  \caption{Example \ref{ex:RST}: The symbols ``$\circ$'' and
    ``$+$'' denote the numerical solutions obtained by using the flux
    \eqref{eq:ecFlux1D1}-\eqref{eq:ecFlux1D3} and the flux (3.6) in \cite{Wu2019},
    respectively. The solid line is the reference solution.}
  \label{fig:RST}
\end{figure}

\section{Conclusion}\label{section:Conclusion}
This paper has presented the high-order accurate entropy stable nodal
DG schemes for the ideal special RMHD equations.
The conservative RMHD equations we usually considered
cannot be symmetrized, thus a particular source term is added into the RMHD
equations to achieve the symmetrization of the RMHD equations,
and the corresponding convex entropy pair is found to symmetrize the modified
RMHD equations. For the modified RMHD equations, high-order entropy
stable DG schemes based on suitable quadrature rules are constructed to
satisfy the semi-discrete entropy inequality for the given entropy pair.
One key is to technically construct the affordable two-point entropy
conservative flux, which is used inside each cell.
Our two-point entropy conservative flux also maintains the zero parallel
magnetic component, which is shown to be useful to reduce the errors in the
parallel magnetic component in several one-dimensional Riemann problems,
while in two dimensions, the entropy stable DG schemes with our two-point
entropy conservative flux and with the existing two-point entropy conservative flux
give comparable results in the rotated shock tube test, thus
our newly derived entropy conservative flux may serve as a better
base of the entropy conservative or the entropy stable schemes for the RMHD equations.
At the cell interfaces, the entropy stable fluxes are used, resulting in an
entropy stable DG schemes satisfying the semi-discrete entropy inequality.
The semi-discrete schemes are integrated in time by using the high-order
explicit Runge-Kutta schemes.
Extensive numerical tests are conducted to validate the accuracy and the ability
to capture discontinuities of our entropy stable DG schemes.

\section*{Acknowledgments}
The authors were partially supported by the Special Project on High-performance Computing under the
National Key R\&D Program (No. 2016YFB0200603), Science Challenge Project (No. TZ2016002),
the National Natural Science Foundation of China (Nos. 91630310 \& 11421101),
and High-performance Computing Platform of Peking University.


\end{document}